\documentclass[reqno]{amsart}
\usepackage[dvipsnames]{xcolor} 
\usepackage{etex}
\usepackage[a4paper,hmarginratio=1:1]{geometry}
\usepackage[T1]{fontenc}
\usepackage{amssymb,amsfonts,amsmath}
\usepackage{comment} 
\usepackage[all,arc]{xy}
\usepackage{enumerate}
\usepackage{mathrsfs,mathtools}
\usepackage{todonotes,booktabs}
\usepackage{stmaryrd}
\usepackage{marvosym}
\usepackage{graphicx}
\usepackage{euscript} 
\usepackage{pdflscape} 	
\usepackage{bbm}

\usepackage[utf8]{inputenc}
\usepackage[english]{babel}

\usepackage{tikz}
\usepackage{tikz-cd}
\usetikzlibrary{trees}
\usetikzlibrary[shapes]
\usetikzlibrary[arrows]
\usetikzlibrary{patterns}
\usetikzlibrary{fadings}
\usetikzlibrary{backgrounds}
\usetikzlibrary{decorations.pathreplacing}
\usetikzlibrary{decorations.pathmorphing}
\usetikzlibrary{positioning}
	
\usepackage{graphicx}
\usepackage{faktor} 
\SelectTips{cm}{10}

\usepackage[pagebackref]{hyperref}

\definecolor{dark-red}{rgb}{0.5,0.15,0.15}
\definecolor{dark-blue}{rgb}{0.15,0.15,0.6}
\definecolor{dark-green}{rgb}{0.15,0.6,0.15}
\hypersetup{
    colorlinks, linkcolor=OliveGreen,
    citecolor=dark-blue, urlcolor=dark-green
}

\renewcommand*{\backref}[1]{}
\renewcommand*{\backrefalt}[4]{%
  \ifcase #1 %
No citations.
  \or
(cit. on p. #2).%
  \else
(cit on pp. #2).%
  \fi%
}

\usepackage[nameinlink,capitalise,noabbrev]{cleveref}
\usepackage[hyphenbreaks]{breakurl}

\newtheorem{thm}{Theorem}[section]
\newtheorem{theorem}[thm]{Theorem}
\newtheorem{corollary}[thm]{Corollary}
\newtheorem{proposition}[thm]{Proposition}
\newtheorem{lemma}[thm]{Lemma}
\newtheorem{conjecture}[thm]{Conjecture}

\newtheorem*{thm*}{Theorem}
\newtheorem*{thma}{Theorem A}
\newtheorem*{thmb}{Theorem B}
\newtheorem*{thmc}{Theorem C}
\newtheorem*{thmd}{Theorem D}
\newtheorem*{thme}{Theorem E}
\newtheorem*{thmf}{Theorem F}

\theoremstyle{definition}
\newtheorem{definition}[thm]{Definition}
\newtheorem{example}[thm]{Example}

\theoremstyle{remark}
\newtheorem{remark}[thm]{Remark}

\newtheorem{construction}[thm]{Construction}
\newtheorem{notation}[thm]{Notation}
\newtheorem{warning}[thm]{Warning}

\makeatletter
\let\c@equation\c@thm
\makeatother
\numberwithin{equation}{section}

\let\lim\relax

\newcommand{\euscr}[1]{\EuScript{#1}} 

\newcommand{\eD}{\euscr{D}}

\newcommand{\F}{\mathbb{F}}
\newcommand{\G}{\mathbb{G}}
\newcommand{\N}{\mathbb{N}}
\newcommand{\Q}{\mathbb{Q}}

\newcommand{\Z}{\mathbb{Z}}

\newcommand{\bOne}{E_{*}}

\DeclareMathOperator{\cA}{\mathcal{A}}

\newcommand{\fm}{\mathfrak{m}}

\DeclareMathOperator{\dual}{dual}

\DeclareMathOperator{\lim}{lim}

\DeclareMathOperator{\op}{op}

\DeclareMathOperator{\Sp}{Sp}
\DeclareMathOperator{\Ab}{Ab}

\DeclareMathOperator{\Hom}{Hom}

\DeclareMathOperator{\Ext}{Ext}

\DeclareMathOperator{\Tor}{Tor}
\DeclareMathOperator{\Tot}{Tot}
\DeclareMathOperator{\Dec}{Dec}
\DeclareMathOperator{\Spec}{Spec}
\DeclareMathOperator{\Mod}{Mod}
\DeclareMathOperator{\Mil}{Mil}

\DeclareMathOperator{\fib}{fib}

\DeclareMathOperator{\Ho}{Ho}
\DeclareMathOperator{\Thick}{Thick}

\DeclareMathOperator{\Fun}{Fun}

\DeclareMathOperator{\QCoh}{QCoh}

\DeclareMathOperator{\Alg}{Alg}

\DeclareMathOperator{\id}{id}

\DeclareMathOperator{\Shv}{Shv}

\DeclareMathOperator{\Map}{Map}

\DeclareMathOperator{\map}{map}

\DeclareMathOperator{\Sym}{Sym}
\DeclareMathOperator{\cts}{cts}

\DeclareMathOperator{\Comod}{Comod}
\DeclareMathOperator{\Inj}{Inj}
\DeclareMathOperator{\ho}{ho}

\newcommand{\acat}{\euscr{A}}

\newcommand{\ccat}{\euscr{C}}
\newcommand{\dcat}{\euscr{D}} 

\newcommand{\mfrak}{\mathfrak{m}} 
\newcommand{\spectra}{\euscr{S}p} 
\newcommand{\gzero}{\mathbf{G}_{0}} 
\newcommand{\g}{\mathbf{G}} 
\newcommand{\rmH}{\mathrm{H}} 
\newcommand{\gr}{\mathrm{gr}} 

\title{Morava $K$-theory and filtrations by powers}

\author{Tobias Barthel}
\address{Max Planck Institute for Mathematics, Vivatsgasse 7, 53111 Bonn, Germany}
\email{tbarthel@mpim-bonn.mpg.de}

\author{Piotr Pstr\k{a}gowski}
\address{Harvard University, Department of Mathematics, 1 Oxford Street, Cambridge MA 02139, USA}
\email{piotr@math.harvard.edu}

\date{\today}
\begin{document}

\begin{abstract}
We prove the convergence of the Adams spectral sequence based on Morava $K$-theory and relate it to the filtration by powers of the maximal ideal in the Lubin--Tate ring through a Miller square. We use the filtration by powers to construct a spectral sequence relating the homology of the $K$-local sphere to derived functors of completion and express the latter as cohomology of the Morava stabilizer group. As an application, we compute the zeroth limit at all primes and heights. 
\end{abstract}

\maketitle

\setcounter{tocdepth}{1}
\tableofcontents
\def\biblio{}

\section{Introduction}

In this paper, which consists of three closely related parts, we investigate the relationship between the filtration by powers of the maximal ideal in the Lubin--Tate ring and chromatic homotopy theory from several perspectives. In the first part, we focus on Morava $K$-theories and their associated Hopf algebroids. In particular, we give an invariant description of their cohomology and prove its finiteness. 

We then establish convergence of the $K$-based Adams spectral sequence and relate it to the $K$-local Adams--Novikov and filtration by powers spectral sequences through a Miller square, and describe it completely at large enough primes. In the last part, we construct Hopkins' spectral sequence relating homology of the $K$-local sphere with derived functors of completion, and identify the latter with cohomology of the Morava stabilizer group. As an application, we compute the zeroth limit at all primes and heights. 

\subsection*{Main results} A classical approach to the study of an arithmetic problem is to first consider its reduction to residue fields and to then reassemble the local solutions. Over the $p$-local integers, there are just two reside fields, namely the field $\Q$ of rational numbers and the finite field $\F_p$ corresponding to the closed point. Informally, these detect, respectively, the torsion-free and torsion phenomena. 

The approach of reducing to residue fields also works very well in stable homotopy theory, where one is interested in classifying stable homotopy classes of maps between finite complexes. Identifying a classical field $k$ with the corresponding Eilenberg--MacLane spectrum, we have a descent spectral sequence of signature
\begin{equation}\label{eq:cass}
E_2^{s,t}\cong \Ext_{\pi_*(k \otimes_{S^0} k)}^{s,t}(k, k) \implies \pi_{t-s}(S^0_{k}),
\end{equation}
where $S^{0}_{k}$ is the appropriate Bousfield localization. When $k = \Q$, this collapses immediately, recovering Serre's calculation that $\pi_{*} S^{0} \otimes _{\mathbb{Z}} \Q \simeq \Q$.

The situation is much more complicated when $k = \F_{p}$, in which case $\cA_* = \pi_*(\F_p\otimes_{S^0}\F_p)$ is the dual Steenrod algebra and the above recovers the classical Adams spectral sequence. In this case, the spectral sequence converges completely, but it does not collapse at any finite page and there is no known algorithmic way of understanding its structure\footnote{A common saying about the classical Adams spectral sequence is ``\emph{Every differential is a theorem.}''}. It is arguably the most important tool for computing the stable homotopy groups of spheres. 

A new feature in homotopy theory which is not visible in the algebraic contexts is that, even $p$-locally, $\Q$ and $\F_p$ are not the only residue fields of spectra. Instead, we have an infinite family of prime fields of ``intermediate characteristic'', given by the Morava $K$-theories $K(n)$ for $0<n<\infty$. These are homotopy ring spectra with the properties that 
\begin{enumerate}
\item $\pi_{*} K(n) \simeq k [u^{\pm 1}]$, where $k$ is a perfect field of characteristic $p$ and $u$ is an invertible variable of degree $|u| = 2$ and
\item the Quillen formal group $\textbf{G}_{0} := \textnormal{Spf}(K(n)^{0}(\mathbf{CP}^{\infty}))$ over $k$ is of finite height $n$. 
\end{enumerate}
One can show that such a homotopy ring spectrum exists for any $\textbf{G}_{0}$ of finite height over a perfect field. By a celebrated result of Devinatz, Hopkins, and Smith, together with $\Q$ and $\F_p$ these are the only residue fields of $p$-local spectra. 

Morava $K$-theories are very calculable for at least two distinct reasons: 

\begin{enumerate}
    \item the coefficients $K(n)_{*}$ form a graded field, so that we have an unrestricted K\"{u}nneth isomorphism $K(n)_{*}(X \otimes Y) \simeq K(n)_{*}(X) \otimes_{K(n)_{*}} K(n)_{*}(Y)$ for any $X, Y$;
    \item there's a canonical Poincare duality isomorphism $K(n)_{*}(X) \simeq K(n)^{*}(X)$ whenever $X$ is a $\pi$-finite space, in particular a classifying space of a finite group. 
\end{enumerate}
The second property, known as ambidexterity, is one way in which $K(n)$ behave as if they were of characteristic zero, despite $K(n)_{*}$ being of positive characteristic. This surprising feature makes Morava $K$-theories a powerful tool also in more geometric contexts, such as in the breakthrough work of Abouzaid and Blumberg on the Arnold conjecture in symplectic topology \cite{abouzaid2021arnold}.

In the present paper, we set up and analyse the descent spectral sequence associated to the Morava $K$-theories. As a first result, we verify that we indeed do have convergence. 
\begin{thma}[\ref{corollary:k_local_adams_conditionally_convergent}, \ref{proposition:complete_convergence_of_k_local_adams_for_dualizable_spectra}]
The $K(n)$-based Adams spectral sequence of signature
\[
\Ext^{s, t}_{K(n)_{*}K(n)}(K(n)_*,K(n)_*(X)) \implies \pi_{t-s}L_{K(n)}X
\]
is conditionally convergent for any spectrum $X$ and converges completely whenever $X$ is $K(n)$-locally dualizable.
\end{thma}

At first glance, this result is quite surprising, as each group on the second page of this spectral sequence for the sphere is torsion, while the abutment $\pi_*L_{K(n)}S^0$ is known to contain torsion-free summands. In particular, this implies that there cannot exist a horizontal vanishing line on any finite page of the spectral sequence. In order to establish conditional convergence, we instead reduce to the case of a finite $n$ complex, where such vanishing lines exist. Complete convergence is the consequence of the degreewise finiteness of the $E_{2}$-term, which we prove for dualizable spectra using the Cartan--Eilenberg spectral sequence

Going back to the case of non-localized spectra, the utility of the classical Adams spectral sequence based on $\F_{p}$ is further amplified when coupled with the Adams--Novikov spectral sequence, which is the descent spectral sequence based on the Brown--Peterson spectrum $BP$. As one spectacular application, Miller constructs a square relating the two spectral sequences and compares the differentials to compute the $v_1$-periodic homotopy of Moore spaces, thereby proving the telescope conjecture at height 1 and for all odd primes \cite{miller_adams}.

The situation is similar in the case of the $K(n)$-based Adams, which has a cousin which might be more familiar to the working chromatic homotopy theorist, namely the $K(n)$-local $E_n$-based Adams spectral sequence. Here, $E_n$ denotes the Lubin--Tate spectrum associated to the formal group $\mathbf{G}_{0}$, parametrizing derived deformations of the latter. We show that the relation between the two spectral sequences is again governed by a square of spectral sequences of the following explicit form. 

\begin{thmb}[\ref{corollary:millers_may_sseq_is_filtration_by_powers}, \ref{proposition:mahowald_sseq_is_ce_sseq}]
\label{thmb:introduction}
The $K(n)$-local Miller square associated to $(E,K)=(E_n, K(n))$ is of the form 
\[
\begin{tikzcd}
	& {\Ext_{E_{*}^{\vee}E}(E_{*}, \bigoplus_k \mfrak^{k}/\mfrak^{k+1} E_{*}^{\vee}X)} \\
	{\Ext_{E_{*}^{\vee}E}(E_{*}, E^{\vee}_{*}X)} && {\Ext_{K_{*}K}(K_{*}, K_{*}X)} \\
	& {\pi_{*}L_{K}X}
	\arrow["{E-\textnormal{Adams}}"', Rightarrow, from=2-1, to=3-2]
	\arrow["{K-\textnormal{Adams}}", Rightarrow, from=2-3, to=3-2]
	\arrow["{\textnormal{May}}"', Rightarrow, from=1-2, to=2-1]
	\arrow["{\textnormal{Mahowald}}", Rightarrow, from=1-2, to=2-3],
\end{tikzcd}
\]
where the May spectral sequence is obtained from the $\mfrak$-adic filtration of $E_{*}^{\vee}X$. 
\end{thmb}
We also identify the Mahowald spectral sequence in the square above with a Cartan--Eilenberg spectral sequence of a certain extension of Hopf algebroids canonically attached to each Morava $K$-theory, see \cref{sec:finiteheightmillersquare} for details. 

A theorem of Miller relates the $d_{2}$-differentials in the $K$-Adams and May spectral sequences, but at large primes we can say much more. If $2p - 2 > n^{2}+n$, then it is well-known that the $K$-local $E$-based Adams spectral sequence for the sphere collapses at the second page. Using a variation on the classical argument of Milnor, we equip $K_{*}K$ at odd primes with an additional grading which also forces the collapse of the Mahowald spectral sequence. 

If follows that at sufficiently large primes, two of the spectral sequences in Miller's square collapse, and it is thus natural to expect that the other two can be identified. This is indeed the case, as we show the following. 

\begin{thmc}[\ref{theorem:k_based_adams_at_large_primes_iso_to_filtration_by_powers}]
\label{thmc:introductionthmc}
If $2p-2 > n^{2}+n+1$, then the $K$-based Adams spectral sequence for $S^{0}_{K}$ can be given an additional grading so that it becomes isomorphic to the May spectral sequence  
\begin{equation}
\label{equation:may_spectral_seq_in_intro}
\Ext_{E_{*}^{\vee}E}(E_{*}, \bigoplus_k \mfrak^{k}/\mfrak^{k+1} E_{*}) \implies \Ext_{E_{*}^{\vee}E}(E_{*}, E_{*})
\end{equation}
induced by the filtration of $E_{*}$ by the powers of the maximal ideal. 
\end{thmc}
Note that one new major phenomenon in intermediate characteristic is that Morava $K$-theories do depend on choices; in particular, on the choice of a formal group. By a result of Lazard, any two such formal groups of the same height are isomorphic over the separable closure; it follows that for our purposes all $K(n)$ of the same height are essentially interchangeable, in particular, they determine the same Adams spectral sequence. 

To be more precise, while the Hopf algebroid $K(n)_{*}K(n)$ depends on the choice of a Morava $K$-theory, its category of comodules does not, by a result of Hovey and Strickland \cite{hovey2005comodules}. In the body of the text, we describe a different derivation of this result, identifying the category of $K(n)_{*}K(n)$-comodules with the category $\Mil_{abs}$ of \emph{absolute Milnor modules}, which are certain sheaves on finite spectra, see \cref{sec:milnormodules}. This gives an approach to the Adams spectral sequence based on Morava $K$-theory which only depends on the prime and the height, but no other choices. 

On a different note, the completion tower of comodules of the form 
\[
\ldots \rightarrow E_{*}/\mfrak^{3} \rightarrow E_{*}/\mfrak^{2} \rightarrow E_{*}/ \mfrak
\]
appearing in  \hyperref[thmc:introductionthmc]{Theorem C} also leads to a \emph{another} spectral sequence, first studied by Hopkins and Sadofsky\footnote{The construction of Hopkins and Sadofsky is unpublished.}. The latter spectral sequence approaches the transchromatic information contained in $L_{K(n)}S^0$ through its \emph{uncompleted} $E$-homology $E_*(L_{K(n)}S^0)$, as opposed to the homotopy groups, as we now describe in more detail.

To provide some context, recall that it is known that at large primes there exists an equivalence $h \spectra_{E} \simeq h \dcat^{per}(E_{*}E)$ between the homotopy categories of $E$-local spectra and differential comodules. Under this equivalence, the limit $\varprojlim  E_{*} / \mfrak^{k}$ in the derived $\infty$-category corresponds to the $K$-local sphere $S^{0}_{K}$ \cite{pstrkagowski2018chromatic}; thus, the derived completion of $E_{*}$ can be thought of as the algebraic analogue of the $K$-local sphere. 

At smaller primes, we do not have algebraic models for the homotopy category, but we instead have a spectral sequence relating the homology groups of the derived completion to $E_{*} (S^{0}_{K})$. This spectral sequence, originally due to Hopkins and Sadofsky, was subsequently worked out by Peterson \cite{peterson2020coalgebraic}, adapting the classical construction of the Adams spectral sequence. We follow a slightly different approach, based on the notion of an adapted homology theory and  Devinatz--Hopkins' modified Adams spectral sequence \cite{dev_morava}.

Our construction is closely related to work of Hovey on derived products of comodules \cite{hovey_product}, and in fact works in much greater generality, see \cref{section:homology_of_inverse_limits}. In the case of $E$-homology, we are able to express these derived limits as cohomology of the Morava stabilizer group, leading to the following statement. 

\begin{thmd}[\ref{thm:limss}, \ref{prop:limssconvergence}, \ref{prop:inverselimascontcohom}]
\label{thmd:introduction}
Let $E_{*}E := \pi_{*}(E \otimes E)$ be the uncompleted $E$-homology of itself, given the unique topology which restricts to the $\mfrak$-adic topology on each finitely generated $E_{*}$-submodule. Then, there exists a canonical isomorphism 
\[
\rmH^{s}_{\cts}(\mathbb{G}_{n}, E_{t}E) \cong (\textstyle\varprojlim_{E_{*}E}^{s} E_{*}/\mfrak^{k})_{t} 
\]
between the continuous cohomology of the Morava stabilizer group and the derived functors of the limit in $E_{*}E$-comodules. Moreover, there exists a spectral sequence of $E_*E$-comodules of the form
\begin{equation}
E_2^{s,t} \cong \rmH^{s}_{\cts}(\mathbb{G}_{n}, E_{t}E) \implies E_{t-s}(L_{K}S^0)
\end{equation}
with differentials $d_r^{s,t}\colon E_r^{s,t} \to E_r^{s+r,t+r-1}$. This spectral sequences converges completely and collapses at a finite page with a horizontal vanishing line. 
\end{thmd}

The subtle point in this result is the convergence of the spectral sequence, which relies crucially on the finite virtual cohomological dimension of the Morava stabilizer group. Note that here it is important to consider the derived sequential limits in the category of comodules as opposed to ordinary modules over $E_*$; indeed, in positive degrees the latter all vanish in this case and there would be no hope of having a convergent spectral sequence.

The above spectral sequence underlies one approach towards Hopkins' chromatic splitting conjecture  \cite{report_on_e_theory_conjectures}, \cite{hovey_csc}. In a nutshell, the latter provides a precise description of the transchromatic behaviour of the $K(n)$-local sphere; that is, its behaviour under applying chromatic localizations $L_h$ for intermediate heights $0 \le h <n$. As such it has been verified for all heights $n \le 2$ and all primes \cite{shimomura_yabe, goerss2012rational, beaudry2017chromatic}, with a minor modification to the original prediction at $n=p=2$ (\cite{bgh_chromaticsplitting}). Beyond height 2, the conjecture remains wide open. 

The standard way of approaching the splitting conjecture is through an explicit computation of the homotopy groups of $L_{K(n)}S^0$. However, as mentioned, its transchromatic information is also encoded in the \emph{uncompleted} $E$-homology, and this is what the spectral sequence \eqref{eq:generalss} abuts to. 

Using the identification of derived functors of the limit with continuous cohomology of the Morava stabilizer group, we are able to compute the zeroth limit at all primes and heights. 

\begin{thme}[{\ref{thm:zeroth_derived_limit_of_completion}}]
There is a canonical isomorphism 
\[
\textstyle\varprojlim_{E_{*}E} E_{*} / \mfrak^{k} E_{*}  \simeq E_{*} \otimes_{\mathbb{Z}} \mathbb{Z}_{p}
\]
where the limit on the left is taken in the category of $E_{*}E$-comodules. 
\end{thme}
Note that $E_{*} \otimes_{\mathbb{Z}} \mathbb{Z}_{p} \simeq E_{*}(S^{0}_{p})$, the homology of the $p$-complete sphere. Thus, the above result can be interpreted as detecting the conjectured copy of $L_{n-1}S^{0}_{p}$ inside $L_{n-1}S^{0}_{K}$ through $E$-homology. 

At height one, the Morava stabilizer group is particularly simple, and we are able to compute all of the derived limits. 

\begin{thmf}[\ref{theorem:cohomology_of_ee_at_height_one}]
At height $n=1$ and any prime, we have 
\[
\rmH_{\cts}^{s}(\mathbb{G}_{1}, E_{*}E) \simeq \begin{cases}
E_{*} \otimes_{\mathbb{Z}} \mathbb{Z}_{p} & \mbox{when } $s=0$,\\
E_{*} \otimes_{\mathbb{Z}} \mathbb{Q}_{p}& \mbox{when } $s=1$,\\
0 & \mbox{otherwise}.
\end{cases}
\]
\end{thmf}
Note that the above in particular completely recovers the height one case of the chromatic splitting conjecture, which states that $L_{0} L_{K(1)} S^{0} \simeq L_{0} S^{0}_{p} \oplus L_{0} S^{-1}_{p}$. The latter can be equivalently obtained by explicitly computing $\pi_{*} L_{K}S^0$, which if we ignore torsion contains exactly two $\mathbb{Z}_{p}$ summands, one for each sphere. However, the computation of homology gives more precise information, as the two copies of $\mathbb{Z}_{p}$ are isomorphic as abelian groups, but the comodules $E_{*} \otimes_{\mathbb{Z}} \mathbb{Z}_{p}$ and $E_{*} \otimes_{\mathbb{Z}} \mathbb{Q}$ are not. 

\subsection*{Outline of document}

We hope that the structure of the document becomes apparent from the choice of section titles. In brief, the first part of the paper, comprising \cref{sec:moravaktheories} through \cref{sec:milnorfiltration} deals with Morava $K$-theories, their Hopf algebroids, categories of comodules, and their cohomology, through the perspective of Milnor modules.  We pay special attention to the (in)dependence of our constructions on the choice of Morava $K$-theories, and revisit the relation to continuous cohomology. 

The $K$-based Adams spectral sequence and the finite height Miller square are the topic of Part 2. We establish convergence properties in \cref{sec:knbasedass}, then construct the finite height Miller square, and use it to relate it to the filtration by powers spectral sequence. In \cref{sec:somecomputations}, the final section of this part, we illustrate our results by going through explicit computations at height 1. 

The third part of the paper consists of \cref{section:homology_of_inverse_limits} to \cref{sec:heightone}. Here, the focus lies on the derived category of $E_*E$-comodules and the construction and study of the inverse limit spectral sequence. We then give our application to the algebraic chromatic splitting conjecture based on the connection to the continuous cohomology of the Morava stabilizer group. Each of the two parts ends with some explicit computations at height one, which we hope elucidate our methods and results.


\subsection*{Acknowledgements}

We would like to thank Agn\`es Beaudry, Robert Burklund, Mike Hopkins, Eric Peterson, and Hal Sadofsky for useful conversations related to this work. We acknowledge the hospitality of the Max Planck Institute for Mathematics in Bonn.

\part{Morava $K$-theories and Milnor modules} In this part of the paper, we study Morava $K$-theories, which play the role of residue fields of Lubin--Tate spectra. We describe their Hopf algebroids, categories of comodules, and cohomology. 

\section{Morava $K$-theories and their Hopf algebroids} 
\label{sec:moravaktheories}

Let $\kappa$ be a perfect field equipped with a choice of a height $n$ formal group $\gzero$. Associated to this data is the Lubin--Tate ring $E_{0}$ classifying deformations \cite{lubin1966formal}. This is a complete local $W(\kappa)$-algebra such that there exists a non-canonical choice of regular generators inducing an isomorphism of rings $E_{0} \simeq W(\kappa)\llbracket u_{1}, \ldots, u_{n-1}\rrbracket$. We write $\mfrak = (p, u_{1} \ldots, u_{n-1})$ for the maximal ideal. 

The ring $E_{0}$ equipped with the universal deformation is Landweber exact and so can be lifted to a $2$-periodic ring spectrum $E$ with $E_{*} \simeq E_{0}[u^{\pm 1}]$ for some unit $u \in E_{2}$, and the property that the formal group $\g := \textnormal{Spf}(E^{0}(BS^{1}))$ is the universal deformation of $\gzero$. It is well-known that $E$ admits a unique $\mathbf{E}_{\infty}$-ring structure, functorial in the choice of $\kappa$ and the formal group, see \cite{goerss2005moduli}. More recently, it was proven by Lurie that $E$ arises as a solution to a moduli problem involving formal groups over $\mathbf{E}_{\infty}$-rings, see \cite{lurie2018elliptic}.

The canonical nature of the Lubin--Tate spectrum makes it a good starting point for the study of chromatic phenomena. Since $E_{0}$ is local, one would like to construct a spectrum which plays the role of the residue field of $E$. Following Hopkins and Lurie \cite{hopkins2017brauer}, we make the following definition. 

\begin{definition}
\label{definition:morava_k_theory}
An $\mathbf{E}_{1}$-$E$-algebra $K$ is a \emph{Morava $K$-theory} if the unit map $E \rightarrow K$ induces an isomorphism $K_{*} \simeq E_{*} / \mfrak$. 
\end{definition}
One can show that an algebra satisfying the above conditions always exists, for any Lubin--Tate spectrum $E$; in fact: 
\begin{enumerate}
    \item there are uncountably many Morava $K$-theories which are not equivalent as $E$-algebras, 
    \item none of which is preferred and 
    \item none of which can be promoted to an $\mathbf{E}_{2}$-$E$-algebra,
\end{enumerate}
see \cite{hopkins2017brauer} for more information. 

\begin{warning}
Large parts of the literature are written in terms of the $(2p^{n}-2)$-periodic spectrum $K(n)$ with $\pi_{*}K(n) \simeq \mathbb{F}_{p}[v_{n}^{\pm 1}]$; this cannot be a made into a Morava $K$-theory according to the above definition as it not $2$-periodic. We will not use $K(n)$ in the current work, but in the interest of completness we collect some results about it in \S\ref{section:digression_minimal_morava_k_theory} below.
\end{warning}

Throughout the rest of this section, the letter $K$ will denote a choice of a Morava $K$-theory in the sense of \cref{definition:morava_k_theory}.

\begin{remark}
While two different Morava $K$-theories need not be equivalent as $E$-algebras, one can show that they are always equivalent as $E$-modules \cite[Corollary 3.6.6]{hopkins2017brauer}. In particular, they are always equivalent as spectra. 
\end{remark}

Observe that we have $K_{*} \simeq E_{*} / \mfrak E_{*}$ by definition, so $K_{*} \simeq \kappa[u^{\pm 1}]$, which is a graded field. Note that the first isomorphism is completely canonical, while the second is not, as it depends on the choice of a unit $u \in K_{2}$. In any case, it follows that the ring spectrum $K$ is even periodic, and so complex-orientable. The reduction map $E^{0}(BS^{1}) \rightarrow K^{0}(BS^{1})$ induces a canonical isomorphism  $\textnormal{Spf}(K^{0}(BS^{1})) \simeq \gzero$. 

\begin{remark}
It is not strictly necessary to start relative to a chosen Lubin--Tate spectrum. Instead, one could say that an $\mathbf{E}_{1}$-algebra $K$ in spectra is a \emph{Morava $K$-theory} if there exists an isomorphism $K_{*} \simeq \kappa[u^{\pm 1}]$ with $\kappa$ a perfect field and such that $\gzero := \textnormal{Spf}(K^{0}(BS^{1}))$ is of finite height. One can show that in this case $K$ can be canonically lifted to an algebra over the corresponding Lubin--Tate ring spectrum. Thus, we lose no generality by working relative to $E$.
\end{remark}
In this paper, we will be interested in the $K$-based Adams spectral sequence. By standard arguments, the $E_{2}$-page of this spectral sequence has a description in terms of homological algebra of comodules over the Hopf algebroid $K_{*}K$. Thus, we begin by giving a partial description of the latter. 

Our strategy is to exploit the $E$-algebra structure on $K$ to divide $K_{*}K$ into two parts, the first of which admits a convenient interpretation in terms of formal groups and the second of which is more mysterious, but manageably small. We start with the latter, for which we will need to work relative to $E$. 

\begin{notation}
If $M, N$ are $E$-modules, we will write
\[
M_{*}^{E}N := \pi_{*}(M \otimes_{E} N)
\]
for their homology relative to $E$ and 
\[
M^{*}_{E}N := [N, M]^{*}_{E} \simeq \pi_{-*} F_{E}(N, M)
\]
for their relative cohomology. Here, $F_{E}$ is the internal mapping $E$-module; that is, the right adjoint to the tensor product of $E$-modules.
\end{notation}
By usual arguments involving flatness of $K_{*}^{E}K$ over $K_{*}$, the map 
\[
K \otimes_{E} K \rightarrow K \otimes_{E} K \otimes_{E} K
\]
induces a $K_{*}$-coalgebra structure on $K_{*}^{E}K$. Note that this is an honest coalgebra structure; that is, the left and right units coincide, as they are necessarily maps of $E_{*}$-algebras, of which $K_{*}$ is a quotient. Together with multiplication, this makes $K_{*}^{E}K$ into a Hopf algebra, which a priori need not be either commutative nor cocommutative.

\begin{lemma}
\label{lemma:internal_k_homology_dual_to_exterior_algebra}
The coalgebra structure on $K_{*}^{E}K$ is dual to an exterior algebra on an $n$-dimensional vector space $V$ in degree $-1$. In particular, $K_{*}^{E}K$ is cocommutative.
\end{lemma}

\begin{proof}
Since the coefficients of $K$ form a field, we have an isomorphism of algebras
\[
\Hom_{K_{*}}(K_{*}^{E}K, K_{*}) \simeq [K, K]_{E},
\]
where on the right hand side we have maps of $E$-modules. One can show that the latter is always isomorphic to an exterior algebra over a vector space of the needed dimension, see \cite[Proposition 6.5.1]{hopkins2017brauer}.
\end{proof}

\begin{remark}
\label{remark:relative_homology_always_commutative}
Hopkins and Lurie show that $K_{*}^{E}K$ is actually isomorphic to an exterior algebra \emph{as a Hopf algebra}, though there is in general no canonical such isomorphism \cite[Proposition 5.2.4]{hopkins2017brauer}. In particular, it is also always commutative. 
\end{remark}

\begin{remark}
\label{remark:kunneth_spectral_sequence_collapses_for_kek}
As both sides are $K_{*}$-vector spaces of dimension $2^{n}$, the latter by \cref{lemma:internal_k_homology_dual_to_exterior_algebra}, the K\"{u}nneth spectral sequence
\[
\Tor^{t, s}_{E_{*}}(K_{*}, K_{*}) \implies K_{s+t}^{E}K
\]
collapses. Indeed, since $E_{0}$ is regular ring of dimension $n$, the relevang $\Tor$-groups are canonically isomorphic to an exterior algebra on $n$ generators. 
\end{remark}
The second, ``understandable'' part of $K_{*}K$ is the image of the canonical map from $E_{*}E$, which can be identified with $K_{*}E$. The analogue of this Hopf algebroid for the minimal Morava $K$-theory of \S\ref{section:digression_minimal_morava_k_theory} is denoted in Ravenel's book by $\Sigma(n)$, see \cite[6.2]{ravenel2003complex}.

Notice that since $E_{*}E$ is flat over $E_{*}$, we have canonical isomorphisms 
\[
K_{*}E \simeq K_{*} \otimes_{E_{*}} E_{*}E \simeq E_{*} E / \mfrak,
\]
which in turn has the following consequence.
\begin{proposition}
\label{proposition:ke_comodules_independnent_of_k}
The category of $K_{*}E$-comodules is independent of $K$; more precisely, for any two choices of Morava $K$-theories, possibly over different Lubin--Tate spectra, these categories of comodules are canonically equivalent as symmetric monoidal categories. 
\end{proposition}

\begin{proof}
By a result of Hovey--Strickland, see \cite[Theorem C]{hovey2005comodules}, the categories of $E_{*}E$-comodules for varying Lubin--Tate spectra $E$ are canonically equivalent as locally graded symmetric monoidal categories. 

The category of $K_{*}E$-comodules can be identified with $K_{*}$-modules in the symmetric moonidal category $\Comod_{E_*E}$, and $K_{*} \simeq E_{*} / \mfrak$ itself is uniquely determined as the minimal quotient of $E_{*}$ as a comodule by \cite[Theorem D]{hovey2005comodules}. Thus, the equivalence of Hovey and Strickland induces one between categories of $K_{*}E$-comodules. 
\end{proof}
In fact, we can give a direct algebro-geometric description of this category of comodules. By standard results about Landweber exact homology theories we have
\[
E_{*}E \simeq E_{*} \otimes_{BP_{*}} BP_{*}BP \otimes_{BP_{*}} E_{*}.
\]
and by again invoking flatness we see that 
\[
K_{*}E \simeq K_{*} \otimes_{BP_{*}} BP_{*}BP \otimes_{BP_{*}} E_{*}
\]
Since the ideal $I_{n}$ is invariant, we can instead write
\[
K_{*}E \simeq K_{*} \otimes_{BP_{*}} BP_{*}BP \otimes_{BP_{*}} K_{*}.
\]
This is a familiar Hopf algebroid, commutative ring homomorphisms out of which classify pairs of homomorphisms $f_{1}, f_{2}: K_{*} \rightarrow A$ together with a strict isomorphism of the resulting formal groups $(f_{1})^{*} \gzero = (f_{2})^{*} \gzero$. In other words, we have a pullback of algebraic stacks
\[\begin{tikzcd}
	{\mathrm{Spec}(K_{*}E)} & {\mathrm{Spec}(K_{*})} \\
	{\mathrm{Spec}(K_{*})} & {\mathcal{M}_{fg}^{\omega=\mathrm{triv}}}
	\arrow[from=1-2, to=2-2]
	\arrow[from=2-1, to=2-2]
	\arrow[from=1-1, to=2-1]
	\arrow[from=1-1, to=1-2],
\end{tikzcd}\]
where $\mathcal{M}_{fg}^{\omega=\mathrm{1}}$ is the moduli of formal groups with a trivialized Lie algebra. The maps from $\Spec(K_{*})$ are faithfully flat surjections onto the height $n$ point in this moduli stack, and we deduce that the groupoid $(\Spec(K_{*}), \Spec(K_{*}E))$ is its presentation. 

\begin{proposition}
\label{proposition:ke_comodules_same_as_qcoh_sheaves}
The category of even graded $K_{*}E$-comodules is equivalent to the category of quasi-coherent sheaves over the moduli stack $\mathcal{M}_{fg}^{n}$ of formal groups of height exactly $n$. 
\end{proposition}

\begin{proof}
By the discussion above, the category of ungraded $K_{*}E$-comodules can be identified with the category of quasi-coherent sheaves over the moduli stack $\mathcal{M}_{fg}^{n, \omega=1}$ of formal groups of height $n$ with trivialized Lie algebra. 

The even grading of $K_{*}E$ corresponds under this equivalence to the $\mathbf{G}_{m}$-action on the chosen trivialization of the Lie algebra, so that even graded comodules can be identified with the quasi-coherent sheaves over the quotient stack 
\[
\faktor{\mathcal{M}_{fg}^{n, \omega=1}}{\mathbf{G}_{m}} \simeq \mathcal{M}_{fg}^{n},
\]
which is exactly the claim. 
\end{proof}
The above discussion identifies $K_{*}E$ with familiar objects from the theory of formal groups, which is why we referred to it above as the ``understandable'' part of $K_{*}K$. The following gives some control over how these two parts are related. 

\begin{lemma}
\label{lemma:eotimesk_is_a_direct_sum_of_ks_as_an_emodule}
Consider $E \otimes K$ as an $E$-module with the module structure inherited only from the left factor. Then, $E \otimes K$ is equivalent as an $E$-module to a direct sum of $K$.
\end{lemma}

\begin{proof}
Since $E_{*}E$ is flat over $E_{*}$, it follows from  \cref{corollary:flat_emodule_is_klocally_profree} below that $L_{K}(E \otimes E)$ is equivalent as an $E$-modules to a $K$-local direct sum of $E$. This is a corollary to the work of Hovey and it is independent of the arguments given here.

Thus, it follows that there is a map $\bigoplus E \rightarrow L_{K}(E \otimes E)$ which is an equivalence after applying $- \otimes K$. As the relative tensor product $- \otimes_{E} K$ can be computed using the bar construction, it also takes $K$-local equivalences of $E$-modules to equivalences. We deduce that 
\[
\bigoplus K \simeq (\bigoplus E) \otimes_{E} K \simeq L_{K}(E \otimes E) \otimes_{E} K \simeq E \otimes E \otimes_{E} K \simeq E \otimes K,
\]
which is what we wanted to show. 
\end{proof}

\begin{remark}
Note that one can alternatively consider $E \otimes K$ as an $E$-module with the module structure inherited form the right factor. In this case, it is clear that it is a direct sum of $K$ as an $E$-module, as it is even a module over the $E$-algebra $K$. 
\end{remark}

\begin{corollary}
\label{corollary:kunneth_for_kk_collapses}
Consider $K \otimes E$ as an $E$-module with the module structure inherited only from the right. Then, the K\"{u}nneth spectral sequence
\[
\Tor_{E_{*}}(K_{*}E, K_{*}) \implies \pi_{*}((K \otimes E) \otimes_{E} K) \simeq K_{*}K
\]
in $E$-modules collapses. 
\end{corollary}

\begin{proof}
This is immediate from \cref{lemma:eotimesk_is_a_direct_sum_of_ks_as_an_emodule} and  \cref{remark:kunneth_spectral_sequence_collapses_for_kek}, as the K\"{u}nneth spectral sequence depends only on the $E$-module structure. 
\end{proof}

\begin{proposition}
\label{proposition:map_from_ke_to_kk_injective}
The map $K_{*}E \rightarrow K_{*}K$ is injective, central, and presents the target as a free $K_{*}E$-module of rank $2^{n}$. 
\end{proposition}

\begin{proof}
To see that the map is central, notice that $K \otimes K$ is canonically an $E \otimes E$-algebra. It follows that the map $E_{*}E \rightarrow K_{*}K$ is central, and thus so must be its image $K_{*}E \simeq K_{*} \otimes_{E_{*}} E_{*}E$.

As a consequence of \cref{corollary:kunneth_for_kk_collapses}, the K\"{u}nneth spectral sequence
\[
K_{*}E \otimes \Tor_{E_{*}}(K_{*}, K_{*}) \simeq \Tor_{E_{*}}(K_{*}E, K_{*}) \implies K_{*}K
\]
in $E$-modules collapses. Since the $\Tor$-groups on the left form a $K_{*}$-vector space of dimension $2^{n}$, it follows from the above collapse that $K_{*}K$ has a finite filtration as a $K_{*}E$-module such that the associated graded is free of rank $2^{n}$. Thus, $K_{*}K$ itself must be free of this rank as well. 
\end{proof}

\begin{proposition}
\label{proposition:k_homology_of_k_tensored_down}
Passing to homotopy groups in the diagram 
\[\begin{tikzcd}
	{K \otimes E} & {K \otimes K} \\
	{K \otimes_{E} E } & {K \otimes_{E} K}
	\arrow[from=1-1, to=2-1]
	\arrow[from=1-1, to=1-2]
	\arrow[from=2-1, to=2-2]
	\arrow[from=1-2, to=2-2]
\end{tikzcd}\]
induces an an isomorphism 
\[
K _{*}K \otimes_{K_{*}E} K_{*} \simeq K_{*}^{E}K.
\]
\end{proposition}

\begin{proof}
Both of the K\"{u}nneth spectral sequences for $K \otimes K \simeq (K \otimes E) \otimes _{E} K$ and $K \otimes_{E} K$ collapse, as a consequence of \cref{corollary:kunneth_for_kk_collapses} and \cref{remark:kunneth_spectral_sequence_collapses_for_kek}. Thus, the map $K_{*}K \rightarrow K_{*}^{E}K$ is surjective after passing to associated graded, and so surjective by an inductive argument. The conclusion follows by observing that $K_{*} \otimes_{K_{*}E} K _{*}K \rightarrow K_{*}^{E}K$ is a surjective map of $K_{*}$-vector spaces of the same dimension.
\end{proof}
We are now ready to assemble the above information. By standard arguments, the pair $(K_{*}, K_{*}K)$ acquires the structure of a Hopf algebroid. More precisely, the two maps $K \rightarrow K \otimes K$ induce left and right units, both of which are central as a consequence of \cref{proposition:map_from_ke_to_kk_injective}, as they factor through $K_{*}E$. These make $K_{*}K$ into a $K_{*}$-bimodule, and similarly we get a suitable comultiplication and antipode. 

Note that it is common in algebraic topology literature to assume in the definition of a Hopf algebroid that multiplication is commutative. However, this is not strictly necessary; in our case, as the units are central, $K_{*}K$ is a Hopf algebroid in the more general sense of Maltsiniotis \cite{maltsiniotis1992groupoides}. The category of comodules can be defined in the usual way, it will be monoidal with the tensor product lifting that of $K_{*}$-modules. It need not in general be symmetric monoidal. 

\begin{remark}
\label{remark:k_k_always_commutative_question}
We believe it is plausible that $K_{*}K$ is in fact always commutative; it is a finite algebra over the commutative $K_{*}E$ with the relative tensor product $K_{*} \otimes _{K_{*}E} K_{*}K \simeq K_{*}^{E} K$ also commutative, as a consequence of \cref{remark:relative_homology_always_commutative}. We were, however, not able to resolve this question. Note that $K_{*}K$ is automatically commutative if we choose a homotopy commutative Morava $K$-theory, as can always be done when $p > 2$ \cite[Proposition 3.5.2]{hopkins2017brauer}.
\end{remark}

\begin{theorem}
\label{theorem:ke_kk_kek_is_an_extension}
The natural maps 
\begin{align}
\label{align:short_exact_sequence_of_hopf_algebroids}
(K_{*}, K_{*}E) \rightarrow (K_{*}, K_{*}K) \rightarrow (K_{*}, K_{*}^{E}K)
\end{align}
form an extension of Hopf algebroids in the sense of Ravenel.
\end{theorem}

\begin{proof}
Since $K_{*}^{E}K$ is a cocommutative, consulting Ravenel's definition of an extension \cite[\S A.1.1]{ravenel2003complex}, we see that we have to verify that 
\[
\Ext^{0}_{K_{*}^{E}K}(K_{*}, K_{*}) \simeq K_{*}
\]
and 
\[
\Ext^{0}_{K_{*}^{E}K}(K_{*}, K_{*}K) \simeq K_{*}E.
\]
The first part is clear, as $K_{*}^{E}K$ is a Hopf algebra, while the second part requires a little bit of work. If we write
\[
K_{*}K \simeq \pi_{*}(K \otimes K) \simeq \pi_{*}(K \otimes_{E} E \otimes K) \simeq K_{*}^{E}(E \otimes K),
\]
then the left $K_{*}^{E}K$-comodule structure on $K_{*}K$ given by the quotient map $K_{*}K \rightarrow K_{*}^{E}K$ corresponds to the standard comodule structure
\[
K_{*}^{E}(E \otimes K) \rightarrow K_{*}^{E}K \otimes_{K_{*}} K_{*}^{E}(E \otimes K) 
\]
on the homology $K_{*}^{E}(E \otimes K)$ relative to the $\infty$-category of $E$-modules. 

The comodule structure map is determined by the structure of $E \otimes K$ as a left $E$-module with module structure coming only from the left factor. By \cref{lemma:eotimesk_is_a_direct_sum_of_ks_as_an_emodule}, $E \otimes K$ is equivalent as an $E$-module to a direct sum of $K$, and so admits a compatible structure of a $K$-module. Choosing such a $K$-module structure determines an isomorphism
\[
K_{*}^{E}(E \otimes K) \simeq K_{*}^{E}K \otimes \pi_{*}(E \otimes K) \simeq K_{*}^{E}K \otimes E_{*}K 
\]
so that 
\[
\Ext^{0}_{K_{*}^{E}K}(K_{*}, K_{*}K) \simeq \Ext^{0}_{K_{*}^{E}K}(K_{*}, K_{*}^{E}(E \otimes K)) \simeq \Ext^{0}_{K_{*}^{E}K}(K_{*}, K_{*}^{E}K \otimes E_{*}K) \simeq 
E_{*}K
\]
as needed. 
\end{proof}

\section{Digression: Minimal Morava $K$-theory} 
\label{section:digression_minimal_morava_k_theory}

Large parts of the literature on Morava $K$-theories are written in terms of the \emph{minimal Morava $K$-theory spectrum} $K(n)$, the unique up to equivalence $BP$-module with $K(n)_{*} \simeq \mathbb{F}_{p}[v_{n}^{\pm 1}]$. Note that it is not a Morava $K$-theory in the sense of \cref{definition:morava_k_theory} as it is not $2$-periodic, but as a consequence of nilpotence theorem any Morava $K$-theory is equivalent as a spectrum to a direct sum of $K(n)$ \cite[Lemma 1.8]{nilpotence2}. 

Our focus on $2$-periodic Morava $K$-theories stems from our preference for Lubin--Tate spectra as they admit canonical $\mathbf{E}_{\infty}$-structures, unlike the Brown--Peterson spectrum. In this short section, we collect some results about $K(n)$ to highlight similarities and differences with the $E$-algebra version. We do this for completness; these results will not be used elsewhere in the current work. 

It is a result of Robinson that $K(n)$ can be made into an $\mathbf{E}_{1}$-ring spectrum, but there is no canonical way to do so \cite{robinson1989obstruction}. A surprising result of Angelveit shows that all choices of $\mathbf{E}_{1}$-multiplication yield equivalent $\mathbf{E}_{1}$-algebras in spectra \cite{angeltveit2011uniqueness}.

One can also ask about less structured multiplication. By a theorem of W\"{u}rgler \cite{wurgler1991morava}, in the homotopy category $h \spectra$ of spectra, the spectrum $K(n)$ admits 
\begin{enumerate}
    \item if $p > 2$, a unique product $K(n) \otimes K(n) \rightarrow K(n)$;
    \item if $p = 2$, exactly two products $K(n) \otimes K(n) \rightarrow K(n)$ differing by the symmetry on $K(n) \otimes K(n)$,
\end{enumerate}
which make it into an associative $BP$-algebra in a way compatible with its $BP$-module structure \cite[Theorem 1.5]{wurgler1991morava}.

The ring of cooperations of $K(n)$ has been computed completely \cite[Theorem 2.4]{wurgler1991morava}. The map $K(n)_{*}BP \rightarrow K(n)_{*}K(n)$ is injective and we have an isomorphism 
\[
K(n)_{*}BP \simeq K(n)_{*}[t_{i}] / (v_{n} t_{i}^{p^{n}} - v_{n}^{p^{i}} t_{i}),
\]
where $|t_{i}| = 2p^{i}-2$ for $i \geq 1$. At odd primes, this extends to an isomorphism
\[
K(n)_{*}K(n) \simeq K(n)_{*}BP \otimes _{K(n)_{*}} \Lambda_{K(n)_{*}}(\tau_{0}, \ldots, \tau_{n-1}),
\]
where the latter factor is an exterior algebra on $|\tau_{i}| = 2p^{i}-1$. When $p = 2$, we instead have that 
\[
K(n)_{*}K(n) \simeq K(n)_{*}BP \otimes _{K(n)_{*}} K(n)_{*}[\tau_{0}, \ldots, \tau_{n-1}] / (\tau_{i}^{2} = t_{i+1}).
\]
This is analogous to the variation in the dual Steenrod algebra depending on whether $p$ is even or odd. Note that the latter computation implies that $K(n)_{*}K(n)$ is commutative also at $p=2$, despite $K(n)$ not being homotopy commutative. The corresponding question in the $2$-periodic case is open, see \cref{remark:k_k_always_commutative_question}.

\section{Milnor modules}\label{sec:milnormodules}

As we have seen in the previous section, there are many different Morava $K$-theory spectra, which makes any discussion of this subject dependent on these choices. However, as in the case of \cref{proposition:ke_comodules_independnent_of_k}, one show that many cohomological invariants are canonical and do not depend on any choices. In this section, we will make this canonical nature transparent by introducing an absolute variant of Hopkins and Lurie's categories of \emph{Milnor modules}, see \cite{hopkins2017brauer}. 

Before we work in the absolute case, let us first describe the case relative to a fixed Lubin--Tate spectrum $E$. Note that by the uniqueness part of Goerss--Hopkins--Miller, this is the same as choosing a perfect field equipped with a finite height formal group.

\begin{definition}
We say that an $E$-module is \emph{molecular} if it is equivalent to a finite sum of shifts of Morava $K$-theories of $E$, and denote the full subcategory of $E$-modules on the molecular objects by $\Mod^{mol}_{E}$.
\end{definition}

\begin{remark}
Note that since any two Morava $K$-theories of $E$ are equivalent as $E$-modules, the notion of being molecular does not depend on any choices. 
\end{remark}

\begin{definition}
The category of $E$-based \emph{Milnor modules} is the category 
\[
\Mil_{E} := P_{\Sigma}(\Mod^{mol}_{E}, \Ab)
\]
of additive presheaves of abelian groups on molecular $E$-modules. 
\end{definition}
By construction, the category of Milnor modules is a compactly generated Grothendieck abelian category. Since molecular $E$-modules are stable under the $E$-tensor product, it acquires a non-unital symmetric monoidal structure via left Kan extension.
One can show this extends uniquely to a unital symmetric monoidal structure \cite[Propositions 4.4.1 and 4.4.10]{hopkins2017brauer}.

The main property of Milnor modules is that it canonically captures cohomology of \emph{any} Morava $K$-theory relative to $E$, in the following sense. 
\begin{proposition}
\label{proposition:milnor_modules_same_as_kek_comodules}
Let $K$ be a Morava $K$-theory of $E$. Then, the left Kan extension of the functor 
\[
K_{*}^{E}(-): \Mod_{E}^{mol} \rightarrow \Comod_{K_{*}^{E}K}
\]
which associates to any molecular $E$-module its relative $K$-homology induces a monoidal equivalence
\[
\Mil_{E} \simeq \Comod_{K_{*}^{E}K}
\]
between Milnor modules and $K_{*}^{E}K$-comodules. 
\end{proposition}

\begin{proof}
This is \cite[Corollary 6.4.13]{hopkins2017brauer}, where we replaced the category of modules over $K^{*}_{E}K$ with comodules over its linear dual. 
\end{proof}

\begin{corollary}
\label{corollary:kek_comod_doesnt_depend_on_k}
As a monoidal category,  $\Comod_{K_{*}^{E}K}$ does not depend on the choice of $K$, but only on its corresponding formal group. 
\end{corollary}

\begin{remark}
Note that in the context of \cref{corollary:kek_comod_doesnt_depend_on_k}, the situation is somewhat better than just saying that the categories $\Comod_{K_{*}^{E}K}$ for fixed $E$ and varying $K$ are all equivalent to each other. In fact, they are all canonically equivalent to a single category which doesn't depend on any choices, namely the category of Milnor modules of $E$. 
\end{remark}

\begin{remark}
\label{remark:milnor_modules_and_homotopy_commutativity}
The category $\Mil_{E}$ is not only canonical, but in a certain sense gives the right answer even if $K_{*}^{E}K$-comodules do not. As an example, the restricted Yoneda embedding 
\[
y\colon \Mod_{E} \rightarrow \Mil_{E},
\]
which can be thought of as a way of taking relative $K$-homology without choosing $K$, is always a symmetric monoidal functor \cite[Variant 4.4.11]{hopkins2017brauer}. 

On the other hand, even if $K$ is not homotopy commutative, $K_{*}^{E}K$ is both commutative and cocommutative so that the category $\Comod_{K_{*}^{E}K}$ acquires a symmetric monoidal structure from $K_{*}$-vector spaces. However, in this case the relative $K$-homology functor 
\[
K_{*}^{E}(-): \Mod_{E} \rightarrow \Comod_{K_{*}^{E}K} 
\]
need not be symmetric monoidal. In fact, one can show that if $p = 2$ and $n>0$, this functor is \emph{never} symmetric monoidal, even when restricted to molecular modules. Indeed, any such functor would induce a symmetric monoidal equivalence $\Comod_{K_{*}^{E}K} \simeq \Mil_{E}$, which one can show is not possible at the even prime\footnote{Personal communication with Jacob Lurie.}.

If $K$ is homotopy commutative, then the above homology theory is canonically symmetric monoidal so that we have a symmetric monoidal equivalence $\Comod_{K_{*}^{E}K} \simeq \Mil_{E}$. Informally, any algebraic structure which one can put on $\Comod_{K_{*}^{E}K}$ in a way compatible with the homology functor from $E$-modules must already be present in the category of Milnor modules.
\end{remark}

Let us now describe the global analogue of this situation, which describes the Adams spectral sequence in spectra rather than $E$-modules; alternatively, which describes comodules over $K_{*}K$ rather than the relative variant $K_{*}^{E}K$. 

In Miller's approach to the Adams spectral sequence, the latter is determined by the class of maps which are $K$-split; that is, which are split epimorphisms after applying $K \otimes -$. This only depends on the structure of $K$ as a spectrum and does not require any coherent multiplication. 

As a consequence of the nilpotence theorem, all Morava $K$-theories at a fixed prime and height are, as spectra, direct sums of the minimal Morava $K$-theory spectrum of \cref{section:digression_minimal_morava_k_theory}. It follows that they all determine the same class of epimorphisms, and hence isomorphic Adams spectral sequences. In fact, the class of epimorphisms is enough to describe canonically the abelian category whose $\Ext$-groups form the $E_{2}$-page. To do so, we will consider sheaves with respect to a certain topology, as in the construction of synthetic spectra \cite{pstrkagowski2018synthetic}.

\begin{definition}
The \emph{$K_{*}$-epimorphism} topology on the $\infty$-category $\spectra^{fin}$ of finite spectra is the Grothendieck pretopology in which a covering families $\{ X \rightarrow Y \}$ consist of a single map such that $K_{*} X \rightarrow K_{*} Y$ is surjective.
\end{definition}

\begin{remark}
A reader familiar with \cite{pstrkagowski2018synthetic} might recall that in the construction of synthetic spectra, one works with the site of those finite spectra $X$ such that $K_{*}X$ is projective as a $K_{*}$-module \cite[Definition 3.12]{pstrkagowski2018synthetic}. Since Morava $K$-theories are fields, this yields the class of all finite spectra, as above. 
\end{remark}

\begin{definition}
The category of \emph{absolute Milnor modules} is the category
\[
\Mil_{abs} := \Shv_{\Sigma}(\Sp^{fin}, \Ab)
\]
of additive sheaves of abelian groups on finite spectra with respect to the $K_{*}$-epimorphism topology. 
\end{definition}

The category of absolute Milnor modules is a compactly-generated Grothendieck abelian category with a symmetric monoidal structure induced by that of the tensor product of finite spectra.

\begin{remark}
A reader familiar with synthethic spectra will immediately recognize that the category of Milnor modules can be canonically identified as $\Mil_{abs} \simeq \euscr{S}yn_K^{\heartsuit}$, the heart of the stable $\infty$-category of $K$-based synthetic spectra of \cite{pstrkagowski2018synthetic}.
\end{remark}
The following is the absolute analogue of \cref{proposition:milnor_modules_same_as_kek_comodules}.

\begin{proposition}
\label{proposition:absolute_milnor_modules_same_as_comodules}
If $K$ is a Morava $K$-theory, then the left Kan extension of the functor
\[
K_{*}: \Sp^{fin} \rightarrow \Comod_{K_{*}K}
\]
induces a monoidal equivalence 
\[
\Mil_{abs} \simeq \Comod_{K_{*}K}
\]
between absolute Milnor modules and $K_{*}K$-comodules. 
\end{proposition}

\begin{proof}
This is an instance of \cite[Remark 3.26]{pstrkagowski2018synthetic}.
\end{proof}
As a consequence, we deduce the following result proven earlier by Hovey and Strickland using slightly different methods \cite[\S 8]{hovey2005comodules}.

\begin{corollary}
\label{corollary:kk_comodules_independent_of_k}
The monoidal category $\Comod_{K_{*}K}$ of $K_{*}K$-comodules does not depend on the choice of Morava $K$-theory, but only on the prime and height. 
\end{corollary}

\begin{remark}
By the same argument as in the relative case of \cref{remark:milnor_modules_and_homotopy_commutativity}, if $K$ is homotopy commutative, then we have a canonical symmetric monoidal equivalence $\Mil_{abs} \simeq \Comod_{K_{*}K}$. 
\end{remark}

\begin{remark}[Homological residue fields]
In the terminology of Balmer, the category of Milnor modules is the \emph{homological residue field} of the tensor-triangulated category of finite spectra at the prime ideal of spectra of type at least $n$, see \cite{balmer2017nilpotence}, \cite{balmer2020computing}.
\end{remark}

\begin{remark}[Milnor modules as sheaves]
\label{remark:relation_between_milnor_modules_plus_independence_of_ce_sseq}
The reader can observe that there is a certain asymmetry in our definitions of Milnor modules and absolute Milnor modules; namely, in the former case we use as our indexing category only molecular $E$-modules and in the latter case we use \emph{all} finite spectra. Moreover, in the latter case we use sheaves while in the former presheaves are already enough. 

In fact, this difference is but a trick of light, as there is an alternative description of Milnor modules which is more in line with the absolute case. Namely, one can show that $\Mil_E$ is equivalent to the category of sheaves on the $\infty$-category $\Mod^{perf}_{E}$ of perfect $E$-modules with respect to the $K_{*}^{E}(-)$-epimorphism topology. Under this topology, the functor 
\[
E \otimes -\colon \spectra^{fin} \rightarrow \Mod_{E}^{perf}
\]
becomes a morphism of sites, so that left Kan extension yields a cocontinuous, symmetric monoidal functor 
\[
\Mil_{abs} \rightarrow \Mil_{E}.
\]
In this sense absolute Milnor modules determine a relative Milnor module, for any choice of $E$, justifying the terminology. 

Under the equivalences of \cref{proposition:milnor_modules_same_as_kek_comodules} and \cref{proposition:absolute_milnor_modules_same_as_comodules} the above left Kan extension corresponds to the forgetful functor 
\[
\Comod_{K_{*}K} \rightarrow \Comod_{K_{*}^{E}K}.
\]
The image of the monoidal unit under the corresponding right adjoint can be identified with the Hopf algebroid $K_{*}E$ studied previously. As the Adams spectral sequence based on $K_{*}E$ in the appropriate derived $\infty$-categories can be identified with the Cartan--Eilenberg spectral sequence associated to this extension of Hopf algebroids \cite[1.1]{belmont2020cartan}, we deduce that the Cartan--Eilenberg spectral sequence itself can be constructed using only the categories of Milnor modules, so it only depends on $E$. 
\end{remark}

\section{Cohomology}
\label{sec:cohomology}

In this section, we establish finiteness properties of the cohomology of $K_{*}K$-comodules using the Cartan--Eilenberg spectral sequence. 

We have seen in \cref{theorem:ke_kk_kek_is_an_extension} that we have an extension of Hopf algebroids 
\[
(K_{*}, K_{*}E) \rightarrow (K_{*}, K_{*}K) \rightarrow (K_{*}, K_{*}^{E}K)
\]
which induces a Cartan--Eilenberg spectral sequence of signature 
\[
\Ext^{s}_{K_{*}E}(K_{*}, \Ext^{s^{\prime}}_{K_{*}^{E}K}(K_{*}, K_{*})) \implies \Ext^{s + s^{\prime}}_{K_{*}K}(K_{*}, K_{*}),
\]
see \cite[A.1.3.14]{ravenel2003complex}. This can be thought as approximating the cohomology of $K_{*}K$ using simplier pieces, namely the cohomology of $K_{*}E$ and $K_{*}^{E}K$. To make best use of this, we need to begin by understanding the latter two. 

We have seen in \cref{sec:moravaktheories} that the Hopf algebroid $(K_{*}, K_{*}E)$ is an even graded presentation of the moduli of formal groups of height exactly $n$, see \cref{proposition:ke_comodules_same_as_qcoh_sheaves}. To identify its cohomology, it is convenient to choose a particularly small presentation of the latter moduli stack, which is what we will do. 

\begin{notation}
\label{notation:small_choice_of_morava_e_theory}
Let $\gzero$ be the Honda formal group law over $\mathbb{F}_{q}$, where $q = p^{n}$, and let $E_{n}$ be the associated Lubin--Tate spectrum and $K_{n}$ a Morava $K$-theory associated to $E_{n}$. Note that we have a non-canonical isomorphism 
\[
\pi_{0} E_{n} \simeq W(\mathbb{F}_{q})[[u_{1}, \ldots, u_{n-1}]].
\]

We write $\mathbb{G}_{n} \simeq \mathbb{S}_{n} \rtimes \textnormal{Gal}(\mathbb{F}_{q} / \mathbb{F}_{p})$ for the \emph{extended Morava stabilizer group}, the automorphism group of the pair $\mathbb{F}_{q}$ together with $\gzero$. By the uniquenses part of Goerss--Hopkins--Miller, the action of $\mathbb{G}_{n}$ extends to an action on the Lubin--Tate spectrum $E_{n}$. 
\end{notation}

\begin{lemma}
\label{lemma:iso_of_hopf_algebroid_of_kn_and_the_gn_one}
The  association $\mathbb{G}_{n} \times (K_{n})_{*}E_{n} \rightarrow (K_{n})_{*}$ sending an element $g \in \mathbb{G}_{n}$ and a homotopy class $S^{k} \rightarrow K_{n} \otimes E_{n}$ to the homotopy class of the composite 
\[
\begin{tikzcd}
	{S^{k}} & {K_{n} \otimes E_{n}} & {K_{n} \otimes E_{n} } & K_{n}
	\arrow["", from=1-1, to=1-2]
	\arrow["{id \otimes g}", from=1-2, to=1-3]
	\arrow["m", from=1-3, to=1-4],
\end{tikzcd}
\]
induces an isomorphism $\map_{\cts}(\mathbb{G}_{n}, (K_{n})_{*}) \simeq (K_{n})_{*}E_{n}$ of algebras between the $K_{n}$-homology of $E_{n}$ and continuous functions on the Morava stabilizer group. 
\end{lemma}

\begin{proof}
This is a result of Strickland, reduced modulo $\mathfrak{m}$ \cite[Theorem 12]{strickland2000gross}.
\end{proof}

One can check that the above isomorphism is compatible with the Hopf algebroid structure, which is induced from the group structure on $\mathbb{G}_{n}$ and its action on $(K_{n})_{*} \simeq \mathbb{F}_{q}[u^{\pm 1}]$,  which yields the following. 

\begin{corollary}
\label{corollary:ke_comodules_same_as_vector_spaces_with_gn_action}
For any Morava $K$-theory $K$, there is a canonical equivalence of categories between $K_{*}E$-comodules and graded $\mathbb{F}_{q}[u^{\pm 1}]$-vector spaces equipped with a continuous action of the extended Morava stabilizer group $\mathbb{G}_{n}$. 
\end{corollary}

\begin{proof}
We have proven previously in  \cref{proposition:ke_comodules_independnent_of_k} that the category of $K_{*}E$-comodules is indendent of the choice of $K$. For the particular choice of $K = K_{n}$ as above, this isomorphism is induced by \cref{lemma:iso_of_hopf_algebroid_of_kn_and_the_gn_one}. 
\end{proof}

\begin{remark}
In terms of the equivalence $\Comod_{K_{*}E}^{ev} \simeq \QCoh(\mathcal{M}_{fg}^{n})$ of \cref{proposition:ke_comodules_same_as_qcoh_sheaves}, \cref{corollary:ke_comodules_same_as_vector_spaces_with_gn_action} is a consequence of the fact that the map $\mathbb{F}_{q} \rightarrow \mathcal{M}_{fg}^{n}$ classifying the Honda formal group law is a Galois covering with automorphism group $\mathbb{G}_{n}$; we refer to \cite{goerss2008stacks} for further details. 
\end{remark}

\begin{corollary}
\label{corollary:k_e_cohomology_same_as_morava_stabilizer}
For any Morava $K$-theory, there is a canonical isomorphism
\[
\Ext_{K_{*}E}^{s, t}(K_{*}, K_{*}) \simeq \mathrm{H}^{*}_{\cts}(\mathbb{G}_{n}, \mathbb{F}_{q}[u^{\pm 1}])
\]
between the cohomology of the Hopf algebraid $(K_{*}, K_{*}E)$ and the continuous cohomology of the Morava stabilizer group.
\end{corollary}
The group $\mathbb{G}_{n} \simeq \mathbb{S}_{n} \rtimes \textnormal{Gal}(\mathbb{F}_{q} / \mathbb{F}_{p})$ is a semidirect product of the automorphism group $\mathbb{S}_{n}$ of $\gzero$ and the Galois group. The former is a $p$-adic Lie group, in fact a group of units in integers in a central division algebra over $\mathbb{Q}_{p}$ of Hasse invariant $\frac{1}{n}$, which gives its excellent cohomological properties. 

The action of the Morava stabilizer group on $\mathbb{F}_{q}[u^{\pm 1}]$ is easy to describe. The Galois group acts by its standard action on $\mathbb{F}_{q}$. The action of $\mathbb{S}_{n}$ factors through the leading coefficient homomorphism $\rho\colon \mathbb{S}_{n} \rightarrow \mathbb{F}_{q}^{\times}$ and is determined by the formula $s \cdot u = p(s)u$ for $s \in \mathbb{S}_{n}$. The kernel of $\rho$ is the strict Morava stabilizer group $S_{n}$; it is a pro-$p$-group acting trivially. Morava \cite{morava1985noetherian} essentially proved:

\begin{theorem}[Morava]
\label{thm:cohomfiniteness}
We have isomorphisms of continuous cohomology groups
\[
\rmH_{\cts}^*(\G_n,\F_{q}[u^{\pm 1}]) \simeq \rmH^{*}_{\cts}(\mathbb{S}_{n}, \mathbb{F}_{q}[u^{\pm 1}])^{\mathrm{Gal}} \simeq \rmH^{*}_{\cts}(S_{n}, \mathbb{F}_{q}[u^{\pm 1}])^{\mathbb{F}_{q}^{\times} \rtimes \mathrm{Gal}}.
\]
Moreover, these groups are finite in each bidegree. 
If $(p-1) \nmid n$, this is a Poincar\'{e} duality algebra over $\rmH_{\cts}^{0} \simeq \mathbb{F}_{p}[v_{n}^{\pm 1}]$ of dimension $n^{2}$, where $v_{n} := u^{p^{n}-1}$, in particular their total dimension is finite.
\end{theorem}

\begin{proof}
These isomorphism are induced by the collapse of the respective Lyndon--Hochschild--Serre spectral sequences, as neither $\mathrm{Gal}$ nor $\mathbb{F}_{q}^{\times}$ have any higher cohomology with these coefficients, the first one by Galois theory and the second one as it is of order coprime to $p$. 

The finiteness statements are due to Lazard, as $S_{n}$ is a $p$-adic analytic Lie group of dimension $n^{2}$ and so has a finite index subgroup whose cohomology is isomorphic to an exterior algebra \cite[Theorem 5.1.5]{symonds2000cohomology} of dimension $n^{2}$. If $(p-1) \nmid n$, then $S_{n}$ has no $p$-torsion, see \cite[Theorem 3.2.1]{henn1998centralizers}, and so it has the same cohomological dimension as all of its finite index subgroups \cite{serre1965dimension}.
\end{proof}

\begin{corollary}\label{cor:cohomfiniteness}
\label{corollary:coh_finiteness_of_ke_comodules}
For any finite-dimensional $K_{*}E$-comodules $M, N$, the groups $\Ext^{t, s}_{K_{*}E}(M, N)$ are finite in each bidegree. If $(p-1) \nmid n$, then they vanish for $s > n^{2}$. 
\end{corollary}

\begin{proof}
We first claim that any finite-dimensional $K_{*}E$-comodule $M$  has an invariant element; that is, there exists a non-zero map $K_{*} \rightarrow M$ of comodules. Observe that $M$ can be considered as an $E_{*}E$-comodule using restriction along the maps $E_{*}E \rightarrow K_{*}E$ and $E_{*} \rightarrow K_{*}$. As $M$ is finitely generated over $E_{*}$, it admits a non-zero map $E_{*} \rightarrow M$ of comodules by a result of Hovey and Strickland \cite[Theorem D]{hovey2005comodules}. As the target is $\mfrak$-torsion, the map will factor through a map $K_{*} \rightarrow M$ of $K_{*}E$-comodules, which is what we were looking for. 

Observe that the result holds when $M = N = K_{*}$ as a combination of  \cref{corollary:k_e_cohomology_same_as_morava_stabilizer} and \cref{thm:cohomfiniteness}. Since $K_{*}$ is a field, it follows that from the above that any finite-dimensional $K_{*}E$ comodule can be obtained as an iterated extension of $K_{*}$. Thus, the general case is proven by presenting $M, N$ as such iterated extensions and using the corresponding long exact sequences of $\Ext$-groups. 
\end{proof}

To get hold of the Cartan--Eilenberg spectral sequence, we also need to have some control over the cohomology of $K_{*}^{E}K$. By a result of Hopkins and Lurie, see  \cref{remark:relative_homology_always_commutative}, the latter is isomorphic to an exterior algebras as a Hopf algebra and so its cohomology is necessarily a polynomial ring. In fact, the same authors identify the generators of this polynomial ring, as we now recall. 

Suppose that $x \in \mfrak$ is an element of the maximal ideal. We then have a fibre sequence
\[
\Sigma^{-1} E \rightarrow E/x \rightarrow E
\]
of $E$-modules which induces a short exact sequence
\[
0 \rightarrow K_{*}[-1] \rightarrow K_{*}^{E}(E/x) \rightarrow K_{*} \rightarrow 0
\]
of $K_{*}^{E}K$-comodules. Let us denote by $\psi(x)$ the correspoding element of $\Ext^{1, -1}_{K_{*}^{E}K}(K_{*}, K_{*})$. 

\begin{proposition}[Hopkins--Lurie]
\label{proposition:hl_cohomology_of_kek}
The association $x \mapsto \psi(x)$ induces a map 
\[
\mfrak / \mfrak^{2} \rightarrow \Ext_{K_{*}^{E}K}^{1, -1}(K_{*}, K_{*})
\]
which extends to an isomorphism of bigraded rings
\[
K_{*} \otimes_{\kappa} \Sym_{\kappa}^{*} (\mfrak / \mfrak^{2}) \simeq \Ext_{K_{*}^{E}K}(K_{*}, K_{*}).
\]
\end{proposition}

\begin{proof}
As we show in \cref{proposition:milnor_modules_same_as_kek_comodules}, the category of $K_{*}^{E}K$-comodules can be identified, as a monoidal category, with the category of Milnor modules associated to $E$. It follows that we have 
\[
\Ext^{*}_{K_{*}^{E}K}(K_{*}, K_{*}) \simeq \Ext^{*}_{\Mil_{E}}(y(E), y(E)),
\]
the cohomology in Milnor modules, where $y(E)$ is the Milnor module associated to $E$. For the right hand side, this isomorphism is \cite[Proposition 7.2.7]{hopkins2017brauer}.
\end{proof}

\begin{remark}
If we again let $K_{n}$ be any Morava $K$-theory associated to the Honda formal group over $\mathbb{F}_{q}$, it follows that the Cartan--Eilenberg spectral sequence for cohomology of the unit has signature 
\[
\rmH_{\cts}^{s_{0}}(\mathbb{G}_{n}, \mathbb{F}_{q}[u^{\pm 1}] \otimes_{\mathbb{F}_{q}} \Sym^{s_{1}}_{\mathbb{F}_{q}}(\mfrak / \mfrak^{2})) \implies \Ext^{s_{0} + s_{1}}_{K_{*}K}(K_{*}, K_{*})
\]

To determine the action on the symmetric algebra, it is convenient to work in the context of Milnor modules, as in the proof of \cref{proposition:hl_cohomology_of_kek}. The category of Milnor modules $\Mil_{E}$ is by construction functorial in automorphisms of $E$ and it it is clear that the association $x \mapsto \psi(x)$ is natural with respect to the action of the Morava stabilizer group. It follows that the action of $\mathbb{G}_{n}$ on this symmetric algebra is the one induced by its action on the cotangent space $\mfrak / \mfrak^{2}$ of the Lubin--Tate ring. 
\end{remark}

Using the Cartan--Eilenberg spectral sequence, we can prove the following finiteness statement.

\begin{theorem}
\label{theorem:finiteness_of_ext_kk}
Let $M, N$ be finite-dimensional $K_{*}K$-comodules. Then, the group $\Ext^{s, t}_{K_{*}K}(M, N)$ is finite in each bidegree.
\end{theorem}

\begin{proof}
By replacing $N$ by $N \otimes M^{\vee}$, where $M^{\vee} := \Hom_{K_{*}}(M, K_{*})$ is the linear dual, we can assume that $M = K_{*}$. We have a Cartan--Eilenberg spectral sequence of signature 
\[
\Ext^{s_{0}}_{K_{*}E}(K_{*}, \Ext^{s_{1}}_{K_{*}^{E}K}(K_{*}, N)) \implies \Ext^{s_{0}+s_{1}}_{K_{*}K}(K_{*}, N)
\]
The groups contributing to $\Ext^{s}_{K_{*}K}(K_{*}, N)$ are those indexed by $s_{0}, s_{1}$ with $s_{0} + s_{1} = s$, of which there are finitely many. For a fixed $s_{1}$, the group $\Ext^{s_{1}}_{K_{*}^{E}K}(K_{*}, N)$ is finite-dimensional over $K_{*}$ since both $M$ and $K_{*}^{E}K$ are, and we can compute this $\Ext$-group using the cobar complex. The finiteness of $\Ext^{s_{0}}_{K_{*}E}(K_{*}, -)$ applied to these groups follows from \cref{cor:cohomfiniteness}.
\end{proof}

\section{Milnor filtration}\label{sec:milnorfiltration} 

In this section we describe a filtration on $K_{*}K$ which arises from its module structure over $E$, which we call the \emph{Milnor filtration}. We show that at odd primes this filtration can be canonically refined to a grading which forces the Cartan--Eilenberg spectral sequence associated to the extension $K_{*}E \rightarrow K_{*}K \rightarrow K_{*}^{E}K$ to collapse. 

\begin{remark}
This existence of an additional grading at $p > 2$ is a finite height analogue of the classical argument involving the Cartan--Eilenberg spectral sequence associated to $H \mathbb{F}_{p}$, see \cite[Theorem 4.3.3]{ravenel2003complex}.
\end{remark}

The Milnor filtration will be induced by a suitable filtration of the spectrum $K \otimes K$. Before we work with such objects, we first recall some standard definitions and establish terminology. 

A non-negatively \emph{filtered spectrum} is a functor $P_{\bullet}\colon \mathbb{N} \rightarrow \spectra$, where we consider the natural numbers as a partially ordered set. If $X$ is a spectrum, then a \emph{filtration} on $X$ is a filtered spectrum $P_{\bullet}$ together with an equivalence $\varinjlim P_{\bullet} \simeq X$.

\begin{remark}
Note that according to our terminology, any filtered spectrum is canonically a filtration of its colimit. 
\end{remark}

To any filtration $P_{\bullet}$ we can attach an associated non-negatively graded spectrum
\[
\gr(P_{\bullet}) = \bigoplus_s\mathrm{gr}_{s}(P_{\bullet})
\]
with the degree $s$ part defined by 
\[
\mathrm{gr}_{s}(P_{\bullet}) := \mathrm{cofib}(P_{s-1} \rightarrow P_{s}),
\]
where $P_{-1} = 0$ by convention. The boundary maps 
\[
\Sigma^{-1} \mathrm{gr}_{s}(P_{\bullet}) \rightarrow P_{s-1} \rightarrow \mathrm{gr}_{s-1}(P_{\bullet})
\]
make $\Sigma^{-s} \gr_{s}(P_{\bullet})$ into a chain complex in the homotopy category of spectra. 
Any filtration has a corresponding spectral sequence of signature
\[
E^{1}_{s, t} := \pi_{t} \mathrm{gr}_{s}(\Sigma^{-s} P_{\bullet}) \implies \pi_{s+t} X,
\]
where the first differential is induced by the chain complex structure on the associated graded \cite[\S 1.2.2]{lurie2016higher}. The increasing filtration on $\pi_{*}X$ associated to this spectral sequence is given by the images of maps $\pi_{*} P_{s} \rightarrow \pi_{*}X$. 

There is a natural tensor product on the $\infty$-category of filtered spectra, given by Day convolution. This has the properties that 
\begin{enumerate}
    \item if $P_{\bullet}$ and $S_{\bullet}$ are filtrations of $X, Y$, then $P_{\bullet} \otimes S_{\bullet}$ is a filtration of $X \otimes Y$ and
    \item $\mathrm{gr}(P_{\bullet} \otimes S_{\bullet}) \simeq \gr(P_{\bullet}) \otimes \gr(S_{\bullet})$, where the latter is the tensor product of graded objects. 
\end{enumerate}

In fact, the latter equivalence can be upgraded to an isomorphism in the category of chain complexes in the homotopy category of spectra, where we consider the latter as symmetric monoidal with the differential on the tensor product given by the usual Leibniz rule \cite[3.2.15]{raksit2020hochschild}.

\begin{example}
We will need certain special filtrations of $E$-modules, the main reference here being Lurie \cite[\S 7.2.1]{lurie2016higher}, who works with simplicial objects. By the stable Dold--Kan equivalence, this is equivalent to working with filtered objects as we do. 

If $P_{\bullet}$ is an $E$-module filtration of some $M$, then we say $P_{\bullet}$ is \emph{$E$-free} if the associated simplicial object is $S$-free with respect to  the set $S = \{ \ \Sigma^{n} E \ | \ n \in \mathbb{Z} \ \}$. We say it is an $E$-\emph{resolution} if it is $E$-free and an $S$-hypercover. For precise definitions of these terms, see \cite[Definition 7.2.1.2.]{lurie2016higher}. The two main facts we will need are that 
\begin{enumerate}
    \item if $P_{\bullet}$ is $E$-free, then the associated graded $E$-module is levelwise free and 
    \item if it is a resolution, then $\pi_{*} \Sigma^{-n} \gr_{n}(P_{\bullet})$ is a projective resolution of $M_{*}$ in $E$-modules. 
\end{enumerate}
Lurie shows that any $E$-module $M$ admits a resolution, which is unique (as an augmented simplicial object) up to simplicial homotopy. 

If $P_{\bullet}$ is an $E$-module resolution of $M$ and $X$ is an arbitrary spectrum, then $P_{\bullet} \otimes X$ is a filtration of $M \otimes X$ and so we get a spectral sequence whose signature is readily shown to be 
\[
E^{2}_{t, s} := \Tor^{s, t}_{E_{*}}(M_{*}, E_{*}X) \implies M_{*}X
\]
This is the K\"{u}nneth spectral sequence which we previously mentioned in \cref{corollary:kunneth_for_kk_collapses}.
\end{example}
We are now ready to make the key definition of this section. 
\begin{definition}
The \emph{Milnor filtration} on $K_{*}K \simeq \pi_{*}(K \otimes K)$ is the filtration induced by the filtration $P_{\bullet} \otimes P_{\bullet}$ of $K \otimes K$, where $P_{\bullet}$ is an $E$-free resolution of $K$ as an $E$-module. 
\end{definition}
In fact, the above filtration is readily identified with a familiar one. 
\begin{lemma}
\label{lemma:milnor_filtration_coincides_with_the_kunneth_filtration}
The Milnor filtration coincides with the filtration induced by the K\"{u}nneth spectral sequence of signature
\[
\Tor_{E_{*}}^{s, t}(K_{*}, E_{*}K) \implies K_{t+s}K.
\]
\end{lemma}

\begin{proof}
The K\"{u}nneth spectral sequence is induced by the filtration $P_{\bullet} \otimes K$, and the obvious map $P_{\bullet} \rightarrow K$ into the constant filtration of the latter induces a comparison map $P_{\bullet} \otimes P_{\bullet} \rightarrow P_{\bullet} \otimes K$ of filtrations of $K \otimes K$. We claim this induces an isomorphism of spectral sequences starting from the second page onwards, and hence an isomorphism of filtrations of $K_{*}K$. 

On the first pages of the corresponding spectral sequences, we are looking at the map 
\[
\pi_{*} (\gr(P_{\bullet}) \otimes \gr(P_{\bullet})) \rightarrow \pi_{*}(\gr(P_{\bullet} \otimes K)),
\]
where we have suppressed the suspensions needed to make the above into chain complexes of graded abelian groups. 

Note that if $M$ is an arbitrary $E$-module, then $E_{*}M \simeq E_{*}E \otimes_{E_{*}} M$ as $E_{*}E$ is flat over $E_{*}$ so that the corresponding K\"{u}nneth spectral sequence collapses. It follows that since each $\gr(P_{\bullet})$ is levelwise free as an $E$-module, we can rewrite the above map on first pages of the corresponding spectral sequences as 
\[
\pi_{*} \gr(P_{\bullet}) \otimes_{E_{*}} E_{*}E \otimes_{E_{*}} \pi_{*} \gr(P_{\bullet}) \rightarrow \pi_{*} \gr(P_{\bullet}) \otimes_{E_{*}} E_{*}E \otimes_{E_{*}} K_{*}.
\]
The tensor product of the first two factors is a non-negatively graded complex of flat $E_{*}$-modules, so that tensoring with it preserves quasi-isomorphism. As $\gr_{*}(P_{\bullet}) \rightarrow K_{*}$ is a quasi-isomorphism, we deduce that so is the map on first pages, ending the argument. 
\end{proof}

\begin{corollary}
The Milnor filtration does not depend on the choice of the $E$-free resolution.
\end{corollary}

\begin{proof}
This is immediate from \cref{lemma:milnor_filtration_coincides_with_the_kunneth_filtration}, as any two $E$-free filtrations $P_{\bullet}^{1}$, $P_{\bullet}^{2}$ induce the same K\"{u}nneth filtration, as $P_{\bullet}^{1} \otimes K$ and $P_{\bullet}^{2} \otimes K$ are simplicially homotopic. 
\end{proof}

\begin{remark}
\label{remark:milnor_filtration_also_coincides_with_other_kunneth_filtration}
By symmetry, \cref{lemma:milnor_filtration_coincides_with_the_kunneth_filtration} implies that the Milnor filtration also coincides with the filtration induced by the other K\"{u}nneth spectral sequence, namely that of signature 
\[
\Tor_{E_{*}}^{s, t}(K_{*}E, K_{*}) \implies K_{t+s}K.
\]
The main advantage of using the Milnor filtration is that its definition is manifestly symmetric, so that it is easier to work with. 
\end{remark}

Note that since the Milnor filtration coincides with the K\"{u}nneth one, we immediately deduce that the associated graded object is given by 
\[
\Tor_{E_{*}}^{*, *}(K_{*}, E_{*}K) \simeq \Tor_{E_{*}}^{*, *}(K_{*}, K_{*}) \otimes E_{*}K,
\]
where the additional grading comes from the external $\Tor$-grading. In particular, the elements of filtration zero are given by 
\[
F_{0}^{Mil}(K_{*}K) = E_{*}K,
\]
the image of the canonical map $E_{*}E \rightarrow K_{*}K$. As the latter is a sub-Hopf algebroid, this suggests that the Milnor filtration is compatible with the Hopf algebroid structure. This is indeed the case, as we will now show. 

\begin{lemma}
\label{lemma:milnor_filtration_is_multiplicative}
The Milnor filtration on $K_{*}K$ is multiplicative. 
\end{lemma}

\begin{proof}
We first claim that the multiplication map $K \otimes K \rightarrow K$ lifts to a morphism $P_{\bullet} \otimes P_{\bullet} \rightarrow P_{\bullet}$ of filtrations. 

Observe that since $K$ is an $E$-algebra, the multiplication factors canonically through a map $K \otimes_{E} K \rightarrow K$. Observe that $P_{\bullet} \otimes_{E} P_{\bullet}$ is an $E$-free filtration of $K \otimes_{E} K$ (although not a resolution, but this is not needed). Then, since $P_{\bullet}$ is an $E$-module resolution of $K$, it follows from \cite[Proposition 7.2.1.5]{lurie2016higher} (where $C = K$) that the map $K \otimes_{E} K \rightarrow K$ of $E$-modules lifts to a map $P_{\bullet} \otimes_{E} P_{\bullet} \rightarrow P_{\bullet}$ of filtrations. The needed lift of $K \otimes K \rightarrow K$ is then given as composition
\[
P_{\bullet} \otimes P_{\bullet} \rightarrow P_{\bullet} \otimes_{E} P_{\bullet} \rightarrow P_{\bullet},
\]
where the first map is the canonical one to the relative tensor product. 

We now claim that the multiplication on $K \otimes K$ also lifts to a map of corresponding filtrations of spectra, showing that it respects filtrations on homotopy groups. This product map is given up to homotopy by the composite
\[
\begin{tikzcd}
	{K \otimes K \otimes K \otimes K} & {} & {K \otimes K \otimes K \otimes K} && {K \otimes K}
	\arrow["{K \otimes \sigma \otimes K}", from=1-1, to=1-3]
	\arrow["{m \otimes m}", from=1-3, to=1-5]
\end{tikzcd},
\]
where $\sigma$ is the twist in the symmetric monoidal structure. This is lifted to a map of filtrations by 
\[\begin{tikzcd}
	{(P_{\bullet} \otimes P_{\bullet}) \otimes (P_{\bullet} \otimes P_{\bullet})} && {P_{\bullet} \otimes P_{\bullet} \otimes P_{\bullet} \otimes P_{\bullet}} && {P_{\bullet} \otimes P_{\bullet} }
	\arrow["{P_{\bullet} \otimes \sigma \otimes P_{\bullet}}", from=1-1, to=1-3]
	\arrow["{\widetilde{m} \otimes \widetilde{m}}", from=1-3, to=1-5],
\end{tikzcd}\]
where $\widetilde{m}$ is the lift of the multiplication map constructed above. This ends the argument. 
\end{proof}

\begin{theorem}
\label{theorem:milnor_filtration_makes_kk_into_a_filtered_hopf_algebroid}
The Milnor filtration makes $K_{*}K$ into a filtered Hopf algebroid over $K_{*}$. 
\end{theorem}

\begin{proof}
We have already checked that the Milnor filtration is preserved by multiplication on $K_{*}K$ in \cref{lemma:milnor_filtration_is_multiplicative}. Thus, we are left with comultiplication and antipode. The latter is induced by the twist map $K \otimes K$, which can be lifted to a map of filtration by the twist on $P_{\bullet} \otimes P_{\bullet}$, hence it preserves the filtration on homotopy groups. 

We move on to comultiplication $\Delta\colon K_{*}K \rightarrow K_{*}K \otimes_{K_{*}} K_{*}K$. This is induced by the map 
\[
\begin{tikzcd}
	{K \otimes K \simeq K \otimes S^{0} \otimes K} & {} & {K \otimes K \otimes K \simeq (K \otimes K) \otimes_{K} (K \otimes K)}
	\arrow["{K \otimes u \otimes K}", from=1-1, to=1-3],
\end{tikzcd}
\]
where $u\colon S^{0} \rightarrow K$ is the unit and we implicitly use the K\"{u}nneth isomorphism
\[
\pi_{*}((K \otimes K) \otimes_{K} (K \otimes K)) \simeq K_{*}K \otimes_{K_{*}} K_{*}K,
\]
which holds in full generality since $K_{*}$ is a field. 

As $P_{\bullet}$ is an $E$-free filtration of $K_{*}$, it induces a trivial filtration on $K_{*}$, so that the unit factors through some map $S^{0} \rightarrow P_{0}$. The latter induces a map $\widetilde{u}\colon S^{0} \rightarrow P_{\bullet}$ of filtrations, where we equip the sphere with the trivial filtration, which makes it into the monoidal unit of filtered spectra. 

Consider the composite 
\[
\begin{tikzcd}
	{P_{\bullet} \otimes P_{\bullet} } && {P_{\bullet} \otimes P_{\bullet} \otimes P_{\bullet} } & {P_{\bullet} \otimes K \otimes P_{\bullet} \simeq (P_{\bullet} \otimes K) \otimes_{K} (K \otimes P_{\bullet}) }
	\arrow["{P_{\bullet} \otimes \widetilde{u} \otimes P_{\bullet}}", from=1-1, to=1-3]
	\arrow[from=1-3, to=1-4]
\end{tikzcd}
\]
which by construction lifts the map inducing comultiplication. As the filtrations of $K_{*}K$ induced by $P_{\bullet} \otimes K$ and $K \otimes P_{\bullet}$ coincide with the Milnor filtration by \cref{lemma:milnor_filtration_coincides_with_the_kunneth_filtration} and 
\cref{remark:milnor_filtration_also_coincides_with_other_kunneth_filtration}, we deduce that comultiplication preserves the Milnor filtration as needed.
\end{proof}

We will show that if $p > 2$ and $K_{*}K$ is commutative---for example, if $K$ is homotopy commutative---then the Milnor filtration can be canonically promoted to a grading. We will prove that the existence of this grading implies that the Cartan--Eilenberg spectral sequence associated to the extension $K_{*}E \rightarrow K_{*}K \rightarrow K_{*}^{E}K$ collapses. 

Assume that $p > 2$ and $K_{*}K$ is commutative, we will first use the Milnor filtration to describe $K_{*}K$ completely as an algebra. By \cref{lemma:milnor_filtration_coincides_with_the_kunneth_filtration}, the associated graded to the Milnor filtration is given by 
\[
\Tor^{*, *}_{E_{*}}(K_{*}, E_{*}K) \simeq \Tor^{*, *}_{E_{*}}(K_{*}, K_{*}) \otimes E_{*}K.
\]
Observe that since $F_{0}^{Mil}(K_{*}K) \simeq E_{*}K$ is in even degree, the external degree one part in the associated graded object, which is
\[
\Tor^{1}_{E_{*}}(K_{*}, K_{*}) \otimes E_{*}K,
\]
can be identified with elements of Milnor filtration at most one in $K_{*}K$ of odd internal degree. Thus, the inclusion 
\[
\Tor^{1}_{E_{*}}(K_{*}, K_{*}) \hookrightarrow \Tor^{1}_{E_{*}}(K_{*}, K_{*}) \otimes E_{*}K,
\]
can be canonically lifted to a unique map \[
\Tor^{1}_{E_{*}}(K_{*}, K_{*}) \rightarrow K_{*}K.
\]
which reduces to the canonical inclusion relative to elements of Milnor filtration zero.  
Since $\Tor^{*, *}_{E_{*}}(K_{*}, K_{*})$ is an exterior algebra on elements of external degree one, when $p > 2$ this $\Tor$-algebra is a free commutative $K_{*}$-algebra on the vector space $\Tor^{1, *}$. If $K_{*}K$ is commutative, it follows that the above lift uniquely extends to a map of $K_{*}$-algebras 
\[
\Tor^{*}_{E_{*}}(K_{*}, K_{*}) \rightarrow K_{*}K.
\]
with the following property.

\begin{proposition}
\label{proposition:kk_at_odd_primes_a_tensor_of_ke_and_tor}
Let $p > 2$ and assume that $K_{*}K$ is commutative. Then, the induced map of $E_{*}K$-algebras
\begin{align}
\label{align:ravenel_iso_of_algebras_involving_kk}
\Tor_{E_{*}}^{*}(K_{*}, K_{*}) \otimes_{K_{*}} E_{*}K \rightarrow K_{*}K
\end{align}
is an isomorphism.
\end{proposition}

\begin{proof}
By construction, the above map is surjective on the associated graded of $K_{*}K$ with respect to the Milnor filtration. Thus, it is surjective by an inductive argument, and as a surjective map of free $E_{*}K$-modules of the same rank it must be an isomorphism. 
\end{proof}
As a consequence of  \cref{proposition:kk_at_odd_primes_a_tensor_of_ke_and_tor}, $K_{*}K$ inherits an additional external grading from the grading of $\Tor$-groups. This grading can be made explicit, as in the following definition. 

\begin{definition}
\label{definition:milnor_grading}
Let $p > 2$ and assume that $K_{*}K$ is commutative. Then, the \emph{Milnor grading} on $K_{*}K$ is the unique algebra grading such that
\begin{enumerate}
    \item elements of Milnor filtration zero are of Milnor degree zero and 
    \item odd internal degree elements of Milnor filtration at most one are of Milnor degree one. 
\end{enumerate}
\end{definition}

\begin{lemma}
The Milnor grading is well-defined. 
\end{lemma}

\begin{proof}
This is immediate from  \cref{proposition:kk_at_odd_primes_a_tensor_of_ke_and_tor}, as the external grading on the algebra tensor product is $\Tor^{*}_{E_{*}}(K_{*}, K_{*}) \otimes_{K_{*}} E_{*}K$ is uniquely determined by which elements are of degrees zero and one. 
\end{proof}

\begin{remark}
The increasing filtration associated to the Milnor grading is exactly the Milnor filtration. Thus, this grading can be considered as a refinement of the latter canonical filtration, which exists at all primes and with no assumptions on commutativity. 
\end{remark}

\begin{theorem}
\label{theorem:existence_and_independence_of_milnor_degree}
Let $p > 2$ and assume that $K_{*}K$ is commutative. Then, the Milnor grading on $K_{*}K$ is compatible with comultiplication and antipode, making it into a bigraded Hopf algebroid. 
\end{theorem}

\begin{proof}
As $K_{*}K$ is generated as an algebra by elements in Milnor degrees zero and one, it is enough to check that these two degrees are preserved. The subspace of elements of Milnor degree most one coincides with elements of filtration at most one, which get preserved by comultiplication and antipode by \cref{theorem:milnor_filtration_makes_kk_into_a_filtered_hopf_algebroid}.

Since the ones in Milnor degree zero are precisely those of even internal degree, and the ones in Milnor degree one are exactly the ones of odd internal degree, we deduce that each degree gets preserved separately by both comultiplication and antipode. 
\end{proof}

\begin{remark}
To see what goes wrong if $p = 2$, observe that even if $K_{*}K$ were commutative, odd degree elements do not have to square to zero and so there need not be a multiplicative extension of the canonical map $\Tor^{1}_{E_{*}}(K_{*}, K_{*}) \rightarrow K_{*}K$ to a map on the whole $\Tor$-algebra. 

One can try to define the Milnor grading at $p = 2$ directly similarly to  \cref{definition:milnor_grading}, by declaring that elements in Milnor degree $s$ are the $s$-fold products of those in Milnor degree one. This is bound to fail; the calculations of W\"{u}rgler in the case of the minimal Morava $K$-theory given in \cref{section:digression_minimal_morava_k_theory} show that the elements $\tau_{i}$ of Milnor degree one square to non-zero elements of Milnor degree zero. 
\end{remark}

\begin{theorem}
\label{theorem:collapse_of_ce_sseq_at_odd_p}
Suppose that $p > 2$ and that $K_{*}K$ is commutative. Then, the Cartan--Eilenberg spectral sequence associated to the extension
\[
K_{*}E \rightarrow K_{*}K \rightarrow K_{*}^{E}K
\]
collapses on the second page, inducing a canonical isomorphism 
\[
\Ext_{K_{*}K}(K_{*}, K_{*}) \simeq \Ext_{K_{*}E}(K_*,\Ext_{K_{*}^{E}K}(K_{*}, K_{*})).
\]
\end{theorem}

\begin{proof}
Under the given assumptions, $K_{*}K$ is bigraded using its internal grading and the Milnor grading of \cref{definition:milnor_grading}. This grading is compatible with the above extension, with all of $K_{*}E$ of Milnor degree zero. 

As $K_{*}^{E}K \simeq \Lambda_{K_{*}}(V)$ as Hopf algebras, where $V$ is the $K_{*}$-vector subspace of primitive elements, we have 
\[
\Ext^{*, *, *}_{K_{*}^{E}K}(K_{*}, K_{*}) \simeq \Sym_{K_{*}}(V^{\vee}),
\]
an isomorphism with the symmetric algebra. Here, $V^{\vee}$ is of homological degree one and Milnor degree one, and so the whole above $\Ext$-algebra is concentrated on the plane of elements for which Milnor degree is equal to the homological one. 

Since the Cartan--Eilenberg differentials lower the homological degree and preserve the Milnor degree, we deduce that they are all zero. This gives an isomorphism as above up to passing to associated graded to the homological filtration, as the extension problems all get trivialized by the additional grading.  
\end{proof}

\part{The $K$-based Adams spectral sequence}

In this part, we study the Adams spectral sequence based on a Morava $K$-theory $K$ at arbitrary height $n$ and prime $p$, constructed as part of a finite height analogue of Miller's square of spectral sequences. In particular, we establish its (perhaps surprisingly) good convergence properties and identify it with the filtration by powers spectral sequence, at least for sufficiently large primes $p$. We then conclude this part with an illustration of the spectral sequence at heights 1 and odd primes, already exhibiting the existence of arbitrary long differentials.

\begin{remark}[Adams spectral sequence based on Milnor modules]
\label{remark:k_based_adams_spectral_sequence_in_terms_of_milnor_modules}
We will not proceed in this way, but one can phrase the $K$-based Adams spectral sequence entirely in terms of absolute Milnor modules, following Miller's observation that the Adams spectral sequence is determined by the class of $K_{*}$-injective maps.

To see this, notice that the restricted Yoneda embedding, followed by a sheafification, gives a symmetric monoidal functor
\[
y\colon \Sp \rightarrow \Mil_{abs}
\]
which can be considered a Milnor modules-valued homology theory. That is, the functor $y$ takes cofibre sequences to exact sequences, preserves arbitrary direct sums, and takes the suspension in spectra to the distinguished auto-equivalence of Milnor modules induced by suspension functor on $\Sp^{fin}$. 

By construction, any injective object in Milnor modules can be lifted to a representing object in spectra, and this allows one to construct a ``$y$-based'' Adams spectral sequence based on injectives, where the $E_{2}$-terms will be canonically given by $\Ext$-groups in absolute Milnor modules \cite{patchkoria2021adams}. Through the equivalence of \cref{proposition:absolute_milnor_modules_same_as_comodules}, this spectral sequence will be canonically isomorphic to the $K$-based Adams spectral sequence for any choice of Morava $K$-theory.
\end{remark}

\section{Convergence of the $K$-based Adams spectral sequence}\label{sec:knbasedass}

In this short section, we prove that the $K$-based Adams spectral sequence is conditionally convergent for any $K$-local spectrum, and it converges completely in the dualizable case. 

\begin{remark}
As observed in \cref{remark:k_based_adams_spectral_sequence_in_terms_of_milnor_modules}, the $K$-based Adams spectral sequence depends only on the choice of the prime and the height. However, this independence from choices will not be needed in this section, as the arguments we use apply equally well to any Morava $K$-theory. 
\end{remark}

By standard arguments, the $K$-based Adams spectral sequence is isomorphic to the totalization spectral sequence obtained by mapping into the  \emph{canonical $K$-Adams resolution} 
\[
X \rightarrow K \otimes X \rightrightarrows K \otimes K \otimes X \Rrightarrow \ldots
\]
Thus, it follows that the $K$-based Adams spectral sequence converges to the homotopy of the \emph{$K$-nilpotent completion of $X$}: 
\[
\Tot(K^{\otimes \bullet +1} \otimes X) := \textstyle\varprojlim_{s \in \Delta} K^{\otimes \bullet +1} \otimes X.
\]
Since the totalization is $K$-local as a limit of $K$-local spectra, there is a canonical comparison map from the $K$-localization of $X$. Our first task is to determine when this map is an equivalence.

\begin{proposition}
\label{prop:knnilpotentcompletion}
For any spectrum $X$, the $K$-nilpotent completion of $X$ is equivalent to the $K$-localization of $X$; that is, the canonical map
\begin{equation}
\label{equation:k_localization_into_k_completion_map_in_thm}
L_{K}X \rightarrow \Tot(K^{\otimes \bullet+1} \otimes X)
\end{equation}
is an equivalence.
\end{proposition}
\begin{proof}
Since both sides are $K$-local, it is enough to show that this map is a $K_{*}$-isomorphism. The latter can be checked after tensoring with any non-$K$-acyclic spectrum as $K_{*}$ is a field, so that we have the K\"{u}nneth isomorphism, and tensoring with a non-zero $K_{*}$-vector space reflects isomorphisms. 

After tensoring with a finite spectrum $F(n)$ of type $n$, the above map becomes 
\[
L_{K}(X \otimes F(n)) \rightarrow \Tot(K^{\otimes \bullet+1} \otimes X \otimes F(n)),
\]
as both sides of (\ref{equation:k_localization_into_k_completion_map_in_thm}) define exact functors in $X$ and so commute with tensoring with a finite spectrum. Thus, it is enough to check the result holds for $X \otimes F(n)$.

The proof of the smash product theorem in \cite[Chapter 8]{orangebook} shows that $S^0$ is $E$-prenilpotent; that is, $L_ES^0 \in \Thick^{\otimes}(E)$, where $\Thick^{\otimes}(Y)$ denotes the smallest full subcategory of $\spectra$ which contains $Y$, is thick, and is closed under smashing with arbitrary objects in $\spectra$. It follows that 
\[
L_{K}F(n) \simeq L_E F(n) \in \Thick^{\otimes}(E\otimes F(n)) \subseteq \Thick^{\otimes}(K).
\]
In other words, $F(n)$ is $K$-prenilpotent. By definition, $X \otimes F(n)$ is thus also $K$-prenilpotent for all $X \in \spectra$. Therefore, Bousfield's theorem \cite[Theorem 6.10]{bousfield_locspectra} applies to provide an equivalence $L_{K}(X \otimes F(n)) \simeq \Tot (K^{\otimes \bullet+1} \otimes X \otimes F(n))$, which is what we wanted. 
\end{proof}
If $X, Y$ are spectra, then by applying $[Y, -]$ to the canonical $K$-Adams resolution of $X$ we obtain the Adams spectral sequence as the totalization spectral sequence. Since $K_{*}$ is a field, $K_{*}Y$ is projective over the base ring, so that by standard arguments the resulting spectral sequence has signature 
\[
E_2^{s,t}(Y, X) \cong \Ext_{K_*K}^{s,t}(K_*Y,K_*X) \implies [Y, L_{K}X]^{t-s}.
\]
Indeed, if we use the canonical Adams resolution, then the first page of the Adams spectral sequence will be given by the cobar complex, which in the projective case computes the $\Ext$-groups in comodules \cite[Corollary A.1.2.12]{ravenel2003complex}.

\begin{corollary}
\label{corollary:k_local_adams_conditionally_convergent}
For any $K$-local spectrum $X$ and any $Y$, the $K$-based Adams spectral sequence converges conditionally in the sense of Boardman to $[Y, L_{K}X]$.
\end{corollary}

\begin{proof}
The totalization spectral sequence is the spectral sequence induced by the tower of partial totalizations. Conditional convergence to $[-, L_{K}X]$ in the sense of Boardman is equivalent to the natural map from $L_{K}X$ into the limit of this tower being an equivalence, so that the statement follows from \cref{prop:knnilpotentcompletion}.
\end{proof}
Let us write $E_{r}^{*, *}(Y, X)$ for the $r$-th page of the $K$-based Adams spectral sequence converging conditionally to $[Y, L_{K}X]$, so that we have $E_{2}^{s, t}(Y, X) \simeq \Ext_{K_{*}K}^{s, t}(K_{*}Y, K_{*}X)$. The differentials in this spectral sequence are of the form
\[
d_r\colon E_r^{s,t}(Y, X) \to E_r^{s+r,t+r-1}(Y, X). 
\]
In particular, for $r>s$ there are inclusions $E_{r+1}^{s,t}(Y, X) \subseteq E_{r}^{s,t}(Y, X)$, and we write 
\[
E_{\infty}^{s,t}(Y, X) = \lim_rE_{r}^{s,t}(Y, X).
\]
Let $F^*[Y,L_{K}X]_*$ denote the filtration on $[Y,L_{K}X]_*$ induced by the spectral sequence. Following Bousfield \cite[\S6]{bousfield_locspectra}, we say that the $K$-based Adams spectral sequence for $X$ \emph{converges completely} if the canonical maps
\begin{equation*}
    [Y,L_{K}X]_* \twoheadrightarrow \lim_s [Y,L_{K}X]_* / F^s[Y,L_{K}X]_*
\end{equation*}
and
\begin{equation*}
    F^s[Y,L_{K}X]_* / F^{s+1}[Y,L_{K}X]_* \hookrightarrow E_{\infty}^{s,s+*}(Y, X)
\end{equation*}
are isomorphisms. By the work of Bousfield, see \cite[Proposition 6.3]{bousfield_locspectra}, this is equivalent to the condition that $\lim_{r>s}^1E_{r}^{s,t}(Y, X)$ vanishes.

\begin{proposition}
\label{proposition:complete_convergence_of_k_local_adams_for_dualizable_spectra}
The $K$-based Adams spectral sequence of signature 
\[
\Ext_{K_*K}^{s,t}(K_*Y,K_*X) \implies [Y, L_{K}X]^{t-s}.
\]
converges completely for any spectra $X$ and $Y$ with $K_{*}X$ and $K_{*}Y$  finite-dimensional.
\end{proposition}

\begin{proof}
By \cref{theorem:finiteness_of_ext_kk}, the groups in the $E_{2}$-term are all finite. Therefore, the system
\[
(E_r^{s,t}(Y, X))_{r>s}
\]
is Mittag-Leffler, so that Bousfield's convergence criterion \cite[Proposition 6.3]{bousfield_locspectra} applies.
\end{proof}

\begin{remark}
For a spectrum $X$, the homology $K_{*}X$ being finite-dimensional is equivalent to $X$ being $K$-locally dualizable, see \cite[Theorem 8.6]{hovey_morava_1999}.
\end{remark}

\begin{remark}
In general, the $K$-based Adams spectral sequence may fail to converge completely in the sense of Bousfield, see \cref{ex:noncompleteconvergence}.
\end{remark}

\section{The finite height Miller square}\label{sec:finiteheightmillersquare}

In his seminal paper \cite{miller_adams} on relations between Adams spectral sequences, Miller uses the relationship between the Adams--Novikov and classical Adams spectral sequences to prove the telescope conjecture at height one and odd primes. In this section, we apply his axiomatic framework to the case of Lubin--Tate and Morava $K$-theory spectra.

In Miller's framework, one starts with a map $A \rightarrow B$ of homotopy ring spectra; that is, an arrow in the category $\Alg(\ho(\Sp))$. In this context, it is not hard to show that a map $X \rightarrow Y$ of spectra which is \emph{$A$-monic}, in the sense that $A \otimes X \rightarrow A \otimes Y$ is a split inclusion, is also $B$-monic. Thus, any $B$-injective spectrum is also $A$-injective, see \cite[Lemma 2.1]{miller_adams}.

\begin{definition}[Miller]
\label{definition:ab_primary_spectrum}
A spectrum $X$ is \emph{$(A, B)$-primary} if there exists an $A$-Adams resolution $X \rightarrow I^{\bullet}$ such that for each ${m}$, the $B$-Adams spectral sequence computing $\pi_{*} I^{m}$ converges and collapses on the second page.
\end{definition}

When $X$ is $(A, B)$-primary, Miller constructs a square of spectral sequences
\[\begin{tikzcd}
	& {E^{2}_{A, B}(X)} \\
	{E^{2}_{A}(X)} && {E^{2}_{B}(X)} \\
	& {\pi_{*}X}
	\arrow["{A-\textnormal{Adams}}"', Rightarrow, from=2-1, to=3-2]
	\arrow["{B-\textnormal{Adams}}", Rightarrow, from=2-3, to=3-2]
	\arrow["{\textnormal{May}}"', Rightarrow, from=1-2, to=2-1]
	\arrow["{\textnormal{Mahowald}}", Rightarrow, from=1-2, to=2-3]
\end{tikzcd}\]
which we will now explain. The terms $E_{A}^{2}(X), E^{2}_{B}(X)$ form the second pages of the relevant Adams spectral sequences, and the bottom two arrows are the corresponding Adams spectral sequences. 

The construction of the other two spectral sequences is more involved, but let us briefly recall the arguments of Miller. Observe that $E^{2}_{A}(X)$ arises as homology of the cochain complex $\pi_{*} I^{\bullet}$, where $I^{\bullet}$ is an $A$-Adams resolution as in \cref{definition:ab_primary_spectrum}. As this cochain complex is given by homotopy groups, it has an additional filtration given by the $B$-Adams filtration. By definition, the term $E^{2}_{A, B}(X)$ is the cohomology of the associated graded of $\pi_{*} I^{\bullet}$ with respect to this filtration, and the top left arrow is the associated cohomology of filtered complex spectral sequence, which Miller calls the \emph{May spectral sequence}. 

If $X$ is primary in the sense of \cref{definition:ab_primary_spectrum}, the cohomology $E^{2}_{A, B}(X)$ of the associated graded of $\pi_{*} I^{\bullet}$ is the same as the cohomology of the complex $E^{2}_{B}(I^{\bullet})$ consisting of the $B$-Adams $E_{2}$-pages of the spectra $I^{\bullet}$. As all $A$-split maps are $B$-split, these $E_{2}$-pages are related by long exact sequences and can be collected into an $E_{1}$ of an exact couple using that $I^{\bullet}$ is a resolution of $X$, yielding a spectral sequence computing $E^{2}_{B}(X)$. This is the spectral sequence in the top right corner, which Miller calls the \emph{Mahowald spectral sequence}. 

\begin{remark}
Miller proves that the above diagram of spectral sequences is commutative in the sense that, after making appropriate choices, the $d_{2}$-differentials in the May spectral sequence can be identified with $d_{2}$-differentials in the $B$-Adams spectral sequence \cite[Theorem 4.2]{miller_adams}. 
In the case of Morava $K$-theory and under the assumption that the prime is large enough compared to the height, we will generalize this result to higher differentials in \cref{proposition:k_adams_in_algebraic_model}. 
\end{remark}

Let us give an explicit description of these spectral sequences for Lubin--Tate spectra and Morava $K$-theories. Note that in this case, the $E$-based Adams spectral sequence computes the $E$-local homotopy groups, which are quite different from the $K$-local ones. Thus, to get a commutative square, we should work internally to $K$-local spectra.

Our first goal is to show that the $K$-based Adams spectral sequence has a particularly simple form for $E$-modules with flat homotopy groups. To do so, we will reduce to the pro-free case using the work of Hovey. 

Before we proceed, let us recall that the completion functor $M_{*} \mapsto (M_{*})_{\mfrak}^{\vee} := \varprojlim M_{*} / \mfrak^{n} M_{*}$ is neither left or right exact, but it has left derived functors $L_{s}\colon \Mod_{E_{*}} \rightarrow \Mod_{E_{*}}$. For more properties of these functors, see \cite[Appendix B]{hovey_morava_1999} or \cite[Appendix A]{barthelfrankland2015}.

\begin{lemma}
\label{lemma:flat_modules_complete_to_profree_ones}
Let $M_*$ be an $E_{*}$-module such that $\Tor^{s}_{E_{*}}(M_{*}, K_{*}) = 0$ for $s > 0$; for example, $M_*$ can be flat. Then, the completion $(M_{*})^{\vee}_{\mfrak}$ is pro-free.
\end{lemma}

\begin{proof}
This is proven by Hovey in  \cite[1.2]{hovey2004some}.
\end{proof}

\begin{lemma}
\label{lemma:k_localization_of_flat_e_module_is_completion}
Let $M$ be an $E$-module such that $M_{*}$ is flat over $E_{*}$. Then, the $K$-localization map $M \rightarrow L_{K} M$ induces an isomorphism $\pi_{*} L_{K} M \simeq (M_{*})^{\vee}_{\mfrak}$ with the completion. 
\end{lemma}

\begin{proof}
Hovey constructs for any $E$-module $M$ a strongly convergent spectral sequence of signature
\[
L_{s} M_{t} \rightarrow \pi_{t+s} L_{K} M.
\]
In the present case, the spectral is concentrated on the zero line by \cref{lemma:flat_modules_complete_to_profree_ones}, hence the spectral sequence collapses inducing the needed isomorphism.
\end{proof}

\begin{corollary}
\label{corollary:flat_emodule_is_klocally_profree}
Let $M$ be an $E$-module such that $M_{*}$ is flat. Then, $L_{K}M$ is equivalent as an $E$-module to the $K$-localization of a direct sum of copies of $E$. 
\end{corollary}

\begin{proof}
We have $\pi_{*} L_{K} M \simeq (M_{*})^{\vee}_{\mfrak}$ by \cref{lemma:k_localization_of_flat_e_module_is_completion} and $(M_{*})^{\vee}_{\mfrak}$ is pro-free by \cref{lemma:flat_modules_complete_to_profree_ones}. Choosing pro-generators of the latter induces a map $L_{K}(\bigoplus E) \rightarrow M$ from a direct sum of shifts of $E$ which is an equivalence by inspection of homotopy groups.
\end{proof}

\begin{proposition}
\label{proposition:collapse_of_k_based_ass_for_flat_e_modules}
If $M$ is an $E$-module such that $M_{*}$ is $E_{*}$-flat, then the $K$-Adams filtration on $M_{*}$ is the $\mfrak$-adic one and the $K$-based Adams spectral sequence for $M$ converges completely to $\pi_{*} L_{K} M := (M_{*})^{\vee}_{\mfrak}$ and collapses on the second page.
\end{proposition}

\begin{proof}
By flatness we have $\pi_{*} L_{K}M \simeq (M_{*})^{\vee}_{\mfrak}$. Since the $\mfrak$-adic filtration on $M_{*}$ is induced by that of completion, we can replace $M$ by its $K$-localization. In this case, $M$ is $K$-locally equivalent to a direct sum of $E$ by \cref{corollary:flat_emodule_is_klocally_profree}. As the Adams spectral sequence of a direct sum is a direct sum of spectral sequences, we can assume that $M = E$. 

We have a base-change isomorphism 
\[
\Ext_{K_{*}K}(K_{*}E, K_{*}) \simeq \Ext_{K_{*}^{E}K}(K_{*}, K_{*}) 
\]
and by \cref{proposition:hl_cohomology_of_kek}
a further isomorphism
\[
\Ext_{K_{*}^{E}K}(K_{*}, K_{*})  \simeq K_{*} \otimes_{\kappa} \Sym_{\kappa}^{*} (\mfrak / \mfrak^{2}).
\]
Since $E_{0}$ is regular, $\Sym_{\kappa}^{*} (\mfrak / \mfrak^{2})$ can be identified with the associated graded to the $\mfrak$-adic filtration on $E_{0}$. 

We deduce from the above that the $E_{2}$-term of the $K$-based Adams for $E$ is abstractly isomorphic to the associated graded of the $\mfrak$-adic filtration on $E_{*}$. Moreover, since the $E_{2}$-term is concentrated in even Adams degree, we deduce that the Adams spectral sequence collapses. 

We are left with showing that the $K$-based Adams filtration on $E_{*}$ coincides with the $\mfrak$-adic one. Observe that since $E \rightarrow K$ is a $K_{*}$-monomorphism by \cref{proposition:map_from_ke_to_kk_injective}, the elements of positive Adams filtration are precisely given by $\mfrak E_{*} = \mathrm{ker}(E_{*} \rightarrow K_{*})$. 

Using compatibility of the Adams filtration with products, we deduce that any element of $\mfrak^{n} E_{*}$ is of Adams filtration at least $n$. Thus, one filtration is contained in the other. As their associated gradeds have the same dimension over $K_{*}$ by what was said above, we deduce that the two filtrations must coincide, as needed. 
\end{proof}

\begin{proposition}
\label{proposition:primarity_from_flat_homology}
Let $X$ be a spectrum such that $E_{*}^{\vee}(X) := \pi_{*} L_{K} (E \otimes X)$ is pro-free; for example, such that $E_{*}X$ is flat. Then, $X$ is $(E, K)$-primary in the $K$-local category. 
\end{proposition}

\begin{proof}
The implication that $E_{*}X$ flat implies $E_{*}^{\vee}X$ pro-free is a combination of \cref{lemma:flat_modules_complete_to_profree_ones} and \cref{proposition:collapse_of_k_based_ass_for_flat_e_modules}. 

Since $E_{*}E$ is flat over $E_{*}$, this in particular implies that $E_{*}^{\vee} E$ is pro-free, which can be also proven directly using evenness, see \cite[THeorem 8.6]{hovey_morava_1999}. It follows (as in \cite[Corollary 1.24]{bh_enass}, for example) that
\[
\pi_{*} L_{K}(E^{\bullet} \otimes X)  \simeq (E_{*}^{\vee}E)^{\otimes^{\vee}_{E_{*}} \bullet-1} \otimes^{\vee}_{E_{*}} E_{*}^{\vee} X,
\]
where on the right hand side we have the completed tensor products, which are again pro-free. This together with \cref{proposition:collapse_of_k_based_ass_for_flat_e_modules} implies that the standard $K$-local $E$-based Adams resolution
\[\begin{tikzcd}
	L_{K}X & {L_{K}(E \otimes X)} & {L_{K}(E \otimes \overline{E} \otimes X)} & \ldots \\
	& {L_{K}(\overline{E} \otimes M)} & {L_{K}(\overline{E} \otimes \overline{E} \otimes X) }
	\arrow[from=1-1, to=1-2]
	\arrow[from=1-2, to=2-2]
	\arrow[from=1-2, to=1-3]
	\arrow[from=1-3, to=2-3]
	\arrow[from=1-3, to=1-4]
	\arrow[dashed, from=2-2, to=1-1]
	\arrow[dashed, from=2-3, to=1-2]
\end{tikzcd}\]
satisfies the primarity condition.
\end{proof}
If $E_{*}^{\vee}X$ is pro-free, then the $E_{1}$-page of the $K$-local $E$-based Adams spectral sequence is given by the cobar complex
\[
\pi_{*} L_{K}(E^{\bullet} \otimes X)  \simeq (E_{*}^{\vee}E)^{\otimes^{\vee}_{E_{*}} \bullet-1} \otimes_{E_{*}} E_{*}^{\vee} X.
\]
It follows from \cref{proposition:collapse_of_k_based_ass_for_flat_e_modules} that the $K$-Adams filtration on these homotopy groups coincides with the filtration by powers of the maximal ideal $\mfrak$; this is the ``filtration by powers'' appearing in the title. This has the following consequence.

\begin{corollary}
\label{corollary:millers_may_sseq_is_filtration_by_powers}
If $E_{*}^{\vee}X$ is pro-free, then Miller's May spectral sequence associated to $(E, K)$ coincides with the spectral sequence of signature 
\[
\Ext_{E_{*}^{\vee}E}^{s,t}(E_{*}, \bigoplus \mfrak^{n}/\mfrak^{n+1} E_{*}^{\vee}X) \implies \Ext_{E_{*}^{\vee}E}^{t-s}(E_{*}, E^{\vee}_{*}X)
\]
obtained by filtering $E_{*}^{\vee}X$ by powers of the maximal ideal; that is, by $\mfrak^{n} E_{*}^{\vee}X$.
\end{corollary}

\begin{proof}
Since Miller's May spectral sequence is induced by the $K$-Adams filtration on the homotopy groups of the $K$-local $E$-Adams resolution, this is a consequence of \cref{proposition:collapse_of_k_based_ass_for_flat_e_modules}, as under the above assumption all of the spectra in the resolution are pro-free. 
\end{proof}

\begin{remark}
The identification of the $E_{2}$-page of the $K$-local $E$-based Adams spectral sequence is somewhat subtle. However, in the case where $E_{*}^{\vee}X$ is pro-free, it follows from \cite{bh_enass} that it acquires a structure of a suitably complete comodule over $E_{*}^{\vee}E$ and that the $E_{2}$-page is given by the $\Ext$-groups in comodules as above.
\end{remark}

We deduce that the $K$-local Miller square associated to $(E, K)$ is of the form 

\begin{equation}
\label{equation:miller_square}
\begin{tikzcd}
	& {\Ext_{E_{*}^{\vee}E}(E_{*}, \bigoplus \mfrak^{n}/\mfrak^{n+1} E_{*}^{\vee}X)} \\
	{\Ext_{E_{*}^{\vee}E}(E_{*}, E^{\vee}_{*}X)} && {\Ext_{K_{*}K}(K_{*}, K_{*}X)} \\
	& {\pi_{*}L_{K}X}
	\arrow["{E-\textnormal{Adams}}"', Rightarrow, from=2-1, to=3-2]
	\arrow["{K-\textnormal{Adams}}", Rightarrow, from=2-3, to=3-2]
	\arrow["{\textnormal{May}}"', Rightarrow, from=1-2, to=2-1]
	\arrow["{\textnormal{Mahowald}}", Rightarrow, from=1-2, to=2-3],
\end{tikzcd}
\end{equation}
whenever $E^{\vee}_{*}X$ is profree. This leaves the question of identifying the Mahowald spectral sequence, which is not too difficult. 

\begin{proposition}
\label{proposition:mahowald_sseq_is_ce_sseq}
If $E_{*}^{\vee}X$ is pro-free, then Miller's Mahowald spectral sequence based on $(E, K)$ can be identified with the Cartan--Eilenberg spectral sequence associated to the extension of Hopf algebroids 
\[
K_{*}E \rightarrow K_{*}K \rightarrow K_{*}^{E}K.
\]
\end{proposition}

\begin{proof}
Taking the standard $K$-local $E$-Adams resolution
\[
X \rightarrow L_{K}(E \otimes X) \rightarrow L_{K}(E \otimes \overline{E} \otimes X) \rightarrow \ldots
\]
and applying $K$-homology gives the long exact sequence of comodules of the form 
\[
K_{*}X \rightarrow K_{*}E \otimes K_{*}X \rightarrow K_{*}E \otimes_{K_{*}} \overline{K_{*}E} \otimes_{K_{*}} K_{*}X \rightarrow \ldots.
\]
Applying $\Ext_{K_{*}K}(K_{*}, -)$ to the above leads to an exact couple which yields the Mahowald spectral sequence.

Note that from the description of the $E_{1}$-term we see that this is just the $K_{*}E$-based Adams spectral sequence in the derived category of $K_{*}K$-comodules, sometimes called the Margolis--Palmieri spectral sequence. This is known to be isomorphic to the Cartan--Eilenberg spectral sequence associated to the extension, see  \cite[Theorem 4.2]{belmont2020cartan}.
\end{proof}

\begin{remark}
\label{remark:collapses_of_mahowald_sseq}
We have shown in \cref{theorem:collapse_of_ce_sseq_at_odd_p} that if $p > 2$, then the Cartan--Eilenberg spectral sequence computing $\Ext_{K_{*}K}(K_{*}, K_{*})$ collapses on the second page. It follows from the isomorphism of \cref{proposition:mahowald_sseq_is_ce_sseq} that the same is true for Miller's $(E,K)$-Mahowald spectral sequence associated to $X = S^{0}$. 
\end{remark}

\begin{remark}
Note that applying \cref{proposition:mahowald_sseq_is_ce_sseq} to the starting terms of the respective spectral sequences implies in particular that
\[
\Ext_{E_{*}^{\vee}E}(E_{*}, \bigoplus \mfrak^{n}/\mfrak^{n+1} E_{*}) \simeq \Ext_{E_{*}K}(K_{*}, \Ext_{K_{*}^{E}K}(K_{*}, K_{*})).
\]
This can be proven directly, without resorting to Miller's work, using Hopkins and Lurie's calculation of the cohomology of $K_{*}^{E}K$, which we stated as \cref{proposition:hl_cohomology_of_kek}. The key is that since $E_{0}$ is regular, we have $\Sym_{\kappa}(\mfrak / \mfrak^{2}) \simeq \bigoplus \mfrak^{n} / \mfrak^{n+1}$. 
\end{remark}

\begin{remark}
If we again pick $E = E_{n}$, the Lubin--Tate spectrum associated to the Honda formal group law of height $n$ over $\mathbb{F}_{p^{n}}$, then the relevant $\Ext$-groups appearing in the Miller square can be described in terms of the cohomology of the Morava stabilizer group, see   \cite[Theorem 4.3]{bh_enass}. Keeping this choice of $E$ in mind, we can rewrite the Miller square as 

\[\begin{tikzcd}
	& {\rmH^{*}(\mathbb{G}_{n}, \bigoplus \mfrak^{n}/\mfrak^{n+1} E_{*}^{\vee}X)} \\
	{\rmH^{*}(\mathbb{G}_{n},E^{\vee}_{*}X)} && {\Ext_{K_{*}K}(K_{*}, K_{*}X)} \\
	& {\pi_{*}L_{K}X}
	\arrow["{E-\textnormal{Adams}}"', Rightarrow, from=2-1, to=3-2]
	\arrow["{K-\textnormal{Adams}}", Rightarrow, from=2-3, to=3-2]
	\arrow["{\textnormal{May}}"', Rightarrow, from=1-2, to=2-1]
	\arrow["{\textnormal{Mahowald}}", Rightarrow, from=1-2, to=2-3].
\end{tikzcd}\]
\end{remark}

\begin{remark}
In the important case of $(MU, H \mathbb{F}_{p})$, the Miller square can also be described explicitly, see for example \cite{gheorghe2018special}. In this context, the Miller square foreshadows the later development of the $\mathrm{C}\tau$-philosophy of Gheorghe, Isaksen, Wang and Xu which applied motivic methods to breakthrough results in stable homotopy groups of spheres \cite{isaksen2020stable}.
\end{remark}

\section{The $K$-based Adams spectral sequence at large primes}
\label{sec:algebraicmodelkbasedass}

It is well-known that if $p$ is large compared to the height, then the $E$-based Adams--Novikov spectral sequence for the $K$-local sphere collapses. As the same is true for the Mahowald spectral sequence, as we have shown in \cref{remark:collapses_of_mahowald_sseq}, we deduce that under these assumptions, two of the sides of the Miller square of (\ref{equation:miller_square}) represent collapsing spectral sequences. 

It is then only natural to expect that the other two sides can be identified up to regrading; in this section, we show that this is indeed the case. More precisely, we will show that if $2p-2 > n^{2}+n+1$, then the $K$-based Adams spectral sequence  for the sphere is isomorphic to the filtration by powers spectral sequence 
\[
\rmH^{*}(\mathbb{G}_{n}, \bigoplus \mfrak^{n}/\mfrak^{n+1}) \implies \rmH^{*}(\mathbb{G}_{n}, E_{*}),
\]
which we can identify with Miller's May spectral sequence by \cref{corollary:millers_may_sseq_is_filtration_by_powers}.

\begin{remark}
Note that a result of Miller \cite[Theorem 4.2]{miller_adams} relates the $d_{2}$-s in the May and $K$-based Adams spectral sequence, at all primes and heights. Our first approach to this problem was to try to push Miller's methods further into higher differentials, hoping that the collapses of the other two spectral sequences in the square would simplify the situation. This approach was unsuccessful, and below we use quite different methods. 
\end{remark}

\begin{remark}
While we will show in this section that at large primes the $K$-based Adams spectral sequence admits an algebraic description, it \emph{never} collapses for $n>0$, unlike the Adams--Novikov spectral sequence. Indeed, the element $v_{n} = u^{p^{n}-1} \in \Ext^{0, 2p^{n}-2}_{K_{*}K}(K_{*}, K_{*})$ is never a permanent cycle, as the $K$-local sphere is of type $0$ and so cannot admit a $v_{n}$-periodic map. 
\end{remark}

To prove our comparison result, we will use an algebraicity result for homotopy categories of $E$-local spectra at large primes due to the second author, strengthened in recent work with Patchkoria, which compares it the following algebraic construction. 

\begin{notation}
Let us write $\eD(E_{*}E)$ for the derived $\infty$-category of the abelian category of comodules. This is a symmetric monoidal stable $\infty$-category, with monoidal unit given by $E_{*}$ itself, considered as a chain complex concentrated in degree zero. The unit defines for any $X \in \eD(E_{*}E)$ its \emph{homotopy groups}
\[
\Ext_{\eD(E_{*}E)}^{s, t}(X) := [\Sigma^{s} E_{*}[t], X]_{\eD(E_{*}E)}.
\]

Notice that the homotopy groups are bigraded, since the abelian category of $E_{*}E$-comodules has an internal grading. This endows the derived $\infty$-category with \emph{two} gradings,

\begin{enumerate}
    \item the \emph{homological} one given by the suspension and corresponding above to $s$ and
    \item the \emph{internal one} induced from the grading shift functor on $\Comod_{E_*E}$ and corresponding above to $t$.
\end{enumerate}
In the special case of $X$ being a comodule, considered as an object of the heart, these homotopy groups encode the $E_{2}$-term of the $E$-based Adams spectral sequence in the $\infty$-category of spectra. 
\end{notation}

\begin{definition}
\label{definition:periodicity_algebra}
The \emph{periodicity algebra} $P(E_{*})$ is the $\mathbf{E}_{\infty}$-algebra in $\dcat(E_{*}E)$ determined by the commutative algebra
\[
E_{*}[\tau^{\pm 1}], \textnormal{where } | \tau | = (1, -1)
\]
in chain complexes of comodules, equipped with the zero differential.
\end{definition}

\begin{notation}
For brevity, let us write 
\[
\Mod_{P(E_{*})} := \Mod_{P(E_{*})}(\dcat(E_{*}E)).
\]
\end{notation}
The following result shows that modules over the periodicity algebra give an algebraic model for the homotopy category of $E$-local spectra. 

\begin{theorem}
\label{theorem:chromatic_algebraicity}
If $2p-2 > n^{2}+n$, there exists an equivalence
\[
\phi\colon h \spectra_{E} \simeq h \Mod_{P(E_{*})}\dcat(E_{*}E)
\]
of homotopy categories of $E$-local spectra and $P(E_{*})$-modules which is moreover compatible with homology functors in the sense that the following diagram commutes 
\[
\begin{tikzcd}
	{h \spectra_{E}} & {} & {h \Mod_{P(E_{*})}\dcat(E_{*}E)} \\
	& \Comod_{E_*E}.
	\arrow["{E_{*}(-)}"', from=1-1, to=2-2]
	\arrow["{\rmH_{0}}", from=1-3, to=2-2]
	\arrow["\phi", from=1-1, to=1-3]
\end{tikzcd}
\]
\end{theorem}

\begin{proof}
This is proven under the assumption $2p-2 > 2(n^{2}+n)$ in the second author's thesis \cite{pstrkagowski2018chromatic}. The bound was improved to the one given above in joint work with Patchkoria \cite{patchkoria2021adams}.
\end{proof}

\begin{remark}
Note that as an $\mathbf{E}_{1}$-algebra in $\dcat(E_{*}E)$, the connective cover of the periodicity algebra $P(E_{*})$ can be identified with the free $\mathbf{E}_{1}$-algebra on the object $\Sigma E_{*}[-1]$, and so is uniquely determined by its homology groups. 

The category of modules over $P(E_{*})$ (considered only as a stable $\infty$-category) can be identified with the subcategory of modules over the connective cover spanned by those modules on which $\tau$ acts invertibly. Thus, to define the target of the equivalence of \cref{theorem:chromatic_algebraicity}, it is not strictly necessary to work with algebras in chain complexes as in \cref{definition:periodicity_algebra}, but we do so out of convenience.
\end{remark}

\begin{remark}
\label{remark:the_many_models_for_e_local_category}
For any $M, N \in \Mod_{P(E_{*})}$ we have an Adams spectral sequence of signature
\[
\Ext^{s, t}(\rmH_{0}(M), \rmH_{0}(N)) \implies [M, N]_{t-s},
\]
where on the right hand side we have homotopy classes of maps of $P(E_{*})$-modules. 

This observation, due to Franke \cite{franke_exotic}, is the starting point for the comparison of $\Mod_{P(E_{*})}$ to the $E$-local category. The general form of Franke's conjecture implies that all  stable $\infty$-categories admitting a suitably convergent Adams-type spectral sequence of this signature must have equivalent homotopy categories \cite{patchkoria2021adams}.
\end{remark}

One of the advantage of $P(E_{*})$-modules over $E$-local spectra is the existence of the free module functor
\[
P(E_{*}) \otimes -\colon \eD(E_{*}E) \rightarrow \Mod_{P(E_{*})}
\]
(also called the \emph{periodicization}) which is cocontinuous and symmetric monoidal. The homotopy groups of free modules can be understood explicitly, as the following shows. 

\begin{remark}
\label{remark:periodicization_simple_homotopy_groups}
Since the periodicity algebra is a sum of shifts of a unit as an object of the derived $\infty$-category, the homotopy groups of the free module are simple to understand. More precisely, for any $M \in \eD(E_{*}E)$ we have
\[
[P(E_{*}), P(E_{*}) \otimes_{E_{*}} M]_{P(\bOne)}^{t}  \simeq [\bOne, P(\bOne) \otimes_{E_{*}} M)]_{\dcat(E_{*}E)}^{t} \simeq \bigoplus_{s} \Ext^{s, t-s}_{E_{*}E}(E_{*}, M).
\]

As a consequence, whenever a $P(\bOne)$-module can be written as a periodicization of some $M$ as above, its homotopy groups acquire an additional grading. Note that any given module can be written as a periodicization in many different ways and this additional grading depends on that choice. 
\end{remark}

Our plan is to use \cref{theorem:chromatic_algebraicity} to compare the $K$-based Adams spectral sequence in $E$-local spectra to the one in its algebraic model, and subsequently give an explicit description of the latter in terms of filtration by powers. 

\begin{lemma}
\label{lemma:exotic_equivalence_preserves_sphere_and_k}
Let $X$ be either the $E$-local sphere or a Morava $K$-theory; that is, $X = S^{0}_{E}$ or $X = K$. Then, there exists a canonical isomorphism
\[
\phi(X) \simeq P(E_{*}) \otimes_{E_{*}} E_{*}X,
\]
in $h \Mod_{P(E_{*})}$, where $\phi$ is the algebraicity equivalence of \cref{theorem:chromatic_algebraicity}.
\end{lemma}

\begin{proof}
It will be convenient to work here with the Johnson-Wilson homology, the $p$-local Landweber exact homology theory with 
\[
\pi_{*}E(n) \simeq \mathbb{Z}_{(p)}[v_{1}, \ldots, v_{n-1}, v_{n}^{\pm 1}]
\]
Both this ring and $E(n)_{*}E(n)$ are concentrated in degrees divisible by $2p-2$, which is a key ingredient in the construction of the equivalence $\phi$ \cite[\S 2]{pstrkagowski2018chromatic}. 

Choosing coordinates for the Quillen formal group over $E_{*}$ such that $v_{i} = 0$ for $i > n$ we obtain a classifying map $E(n)_{*} \rightarrow E_{*}$ which is faithfully flat and hence 
\[
E_{*}X \simeq E_{*} \otimes_{E(n)_{*}} E(n)_{*}X
\]
The induced functor $\Comod_{E(n)_{*}E(n)} \rightarrow \Comod_{E_{*}E}$ is an equivalence of categories by a result of Hovey and Strickland \cite{hovey2005comodules}.

Now suppose that $X$ is a spectrum with $E(n)_{*}X$ concentrated in degrees divisible by $2p-2$. Since $\phi$ is compatible with taking homology, we have 
\[
\rmH_{0}(\phi(X)) \simeq E_{*}X \simeq E_{*} \otimes_{E(n)_{*}} E(n)_{*}X
\]
as comodules. By \cref{remark:the_many_models_for_e_local_category}, we have an Adams spectral sequence
\[
\Ext^{s, t}(E_{*}X, E_{*}X) \implies [\phi(X), P(E_{*}) \otimes_{E_{*}} E_{*}X]^{t-s}
\]
Using the equivalence of categories of comodules, we can rewrite the $E_{2}$-page as 
\[
\Ext^{s, t}(E(n)_{*}X, E(n)_{*}X) \simeq \Ext^{s, t}(E_{*}X, E_{*}X) 
\]
This vanishes above $n^{2}+n$ by \cite[Remark 2.5]{pstrkagowski2018chromatic} and is concentrated in degrees $t$ divisible by $2p-2$. Since $2p-2 > n^{2}+n$ by assumption, we deduce that the spectral sequence collapses on the second page, so that the identity of $E_{*}X$ descends to an equivalence $\phi(X) \simeq P(E_{*}) \otimes_{E_{*}} E_{*}X$. 

This establishes the result for $X = S^{0}_{E}$, which has $E(n)_{*}X$ homology concentrated in degrees divisible by $2p-2$. If $K$ is a Morava $K$-theory, then as a consequence of the nilpotence theorem it is a module over the minimal Morava $K$-theory of \S\ref{section:digression_minimal_morava_k_theory} and so equivalent as a spectrum to a direct sum of $K(n)$. As $\phi$ preserves direct sums and
\[
E(n)_{*}K(n) \simeq K(n)_{*} \otimes_{E(n)_{*}} E(n)_{*}E(n)
\]
is concentrated in degrees divisible by $2p-2$, the result follows also for $X = K$. 
\end{proof}

\begin{remark}
The detour into Johnson-Wilson homology \cref{lemma:exotic_equivalence_preserves_sphere_and_k} can be avoided, by directly defining a splitting of order $(2p-2)$ on the category $\Comod_{E_{*}E}$ by using the equivalence with quasi-coherent sheaves on the moduli of formal groups of height at most $n$. We decided against it, as the degree divisibility argument involving $E(n)$ is classical and well-known.
\end{remark}

Let us write $K_{alg} := P(E_{*}) \otimes _{E_{*}} E_{*}K$ for the algebraic model for the Morava $K$-theory spectrum. Note that its homotopy groups satisfy
\[
\pi_{s} K_{alg} := [P(E_{*}), K_{alg}]^{P(E_{*})}_{s} \simeq K_{s}, 
\]
so that $K_{alg}$ is a field object in $P(E_{*})$. Consequently, for a map $M \rightarrow N$ of $P(E_{*})$-modules, the following two conditions are equivalent:
\begin{enumerate}
    \item $K_{alg} \otimes_{P(E_{*})} M \rightarrow K_{alg} \otimes_{P(E_{*})} N$ is a split monomorphism of $P(E_{*})$-modules or 
    \item $[N, K_{alg}]_{*} \rightarrow [M, K_{alg}]_{*}$ is an epimorphism of graded abelian groups. 
\end{enumerate}
This forms a class of monomorphisms and so determines an Adams spectral sequence \cite[\S3.1]{patchkoria2021adams}.
\begin{definition}
We call the Adams spectral sequence in $P(E_{*})$-modules determined by the above class of monomorphisms the \emph{$K_{alg}$-based Adams spectral sequence}. 
\end{definition}
We are now ready to verify the correspondence between the topological and algebraic Adams spectral sequences; we learned the following elegant argument from Robert Burklund. 

\begin{proposition}
\label{proposition:k_adams_in_algebraic_model}
If $2p-2 > n^{2}+n+1$, then the $K$-based Adams spectral sequence in $\spectra_{E}$ for $S^{0}_{E}$ is isomorphic to the $K_{alg}$-based Adams spectral sequence for $P(E_{*})$ in $P(E_{*})$-modules. 
\end{proposition}

\begin{proof}
We claim that the equivalence $\phi$ of \cref{theorem:chromatic_algebraicity} induces an isomorphism of exact couples leading to these spectral sequences. Both of these are obtained by mapping into suitable towers, so we have to verify that $\phi$ takes a $K$-Adams resolution of $S^{0}_{E}$ to a $K_{alg}$-Adams resolution of $\phi(S^{0}_{E)}) \simeq P(E_{*})$. 

If $2p-2 > n^{2}+n+1$, the equivalence $\phi$ is compatible with triangulated structures \cite[B.8]{pstrkagowski2018chromatic}, so it remains to show that it identifies $K$-injectives with $K_{alg}$-injectives, as the corresponding class of monomorphism is then uniquely determined. This follows immediately from \cref{lemma:exotic_equivalence_preserves_sphere_and_k}, as on both sides each injective is a direct sum of shifts of, respectively, $K$ and $K_{alg}$. 
\end{proof}

In our case, in  \cref{proposition:k_adams_in_algebraic_model} we considered a $K_{alg} = P(\bOne) \otimes E_{*}K$-based Adams spectral sequence for the monoidal unit $P(\bOne)$. Both objects arise through periodicization, giving us additional information, which we will now make explicit. 

\begin{proposition}
\label{proposition:k_based_adams_has_an_additional_grading_for_which_it_is_ek_adams}
If $2p-2 > n^{2}+n+1$, then the $K$-based Adams spectral sequence for $S^{0}_{E}$ can be given an additional grading so that it becomes isomorphic to the $E_{*}K$-based Adams spectral sequence in $\eD(E_{*}E)$. 
\end{proposition}

\begin{proof}
By \cref{proposition:k_adams_in_algebraic_model}, the $K$-Adams spectral sequence in $\spectra_{E}$ is isomorphic to the $K_{alg}$-based Adams spectral sequence in $P(E_{*})$-modules. The latter is isomorphic to the spectral sequence induced by the Amitsur resolution
\[
P(E_{*}) \rightarrow K_{alg} \rightrightarrows K_{alg} \otimes_{P(E_{*})} K_{alg} \Rrightarrow \ldots
\]
which is image of the $E_{*}K$-Adams resolution of $E_{*}$ in $\eD(E_{*}E)$ under the periodicization functor. The statement follows from the isomorphism of homotopy groups of \cref{remark:periodicization_simple_homotopy_groups} applied to the exact couple giving rise to the spectral sequence.
\end{proof}
The above result reduces the study of the $K$-based Adams spectral sequence at large primes to the study of the algebrac $E_{*}K$-based in the derived $\infty$-category of comodules. The latter can be described explicitly in terms of the cohomology of the Morava stabilizer group, with no assumptions on the prime, which is our next step. 

We will need to recall some results on d\'{e}calage, following Deligne and Levine \cite{levine2015adams}. 

\begin{construction}
\label{construction:decalage}
If $X^{\bullet}$ is a cosimplicial spectrum, we have an associated conditionally convergent spectral sequence
\[
\prescript{\Tot}{}{E}^{2}_{s, t} := \rmH^{s} \pi_{t} X^{\bullet} \implies \pi_{t-s} \Tot(X^{\bullet}).
\]
The \textit{d\'{e}calage} of this cosimplicial object is the tower 
\[
\Dec_{n} := \Tot(\tau_{\leq n} X^{\bullet})
\]
obtained by totalizing the Postnikov truncations of $X$. Since limits commute with limits we have
\[
\varprojlim \Dec_{n} \simeq \varprojlim \Tot(\tau_{\leq n} X^{\bullet}) \simeq \Tot(\varprojlim \tau_{\leq n} X^{\bullet}) \simeq \Tot(X^{\bullet}).
\]
The above equivalence induces a spectral sequence associated to the tower on the left, which is a conditionally convergent and of signature
\[
\prescript{\Dec}{}{E}^{1}_{s, t} := H^{s} \pi_{t}X^{\bullet} \implies \pi_{t-s} \Tot(X^{\bullet}).
\]
Levine proves that the evident isomorphisms between the $E^{2}$-term of the $\Tot$-spectral sequence and the $E_{1}$-term of the spectral sequence associated to the d\'{e}calage extends to an isomorphism of spectral sequences \cite{levine2015adams}. 
\end{construction}

In our case, we will be working with cosimplicial objects in $\eD(E_{*}E)$, rather than in spectra. Thus, the role of the Postnikov towers will be played by the truncations in the standard $t$-structure on the derived category. 

\begin{definition}
Let $X^{\bullet}\colon \Delta \rightarrow \eD(E_{*}E)$ be a cosimplicial object in the derived $\infty$-category of comodules. The \emph{homological d\'{e}calage} of $X^{\bullet}$ is the tower 
\[
h\Dec_{n} := \Tot(\tau_{\leq n} X^{\bullet}),
\]
where $\tau_{\leq n}$ denotes the truncation in the standard $t$-structure on the derived $\infty$-category.
\end{definition}

\begin{warning}
Note that if $X^{\bullet}$ is a cosimplicial object in $\eD(E_{*}E)$, then the homological d\'{e}calage tower induces a spectral sequence obtained by applying $\Ext(E_{*}, -)$. Since Postnikov towers converge in $\dcat(E_{*}E)$, this will be conditionally convergent to $\Ext(E_{*}, \Tot(X^{\bullet}))$. 

The layers of the tower are given, up to a shift, by $\Tot(H_{k}(X^{\bullet}))$ so that this spectral sequence has signature 
\[
\Ext(E_{*}, \Tot(H_{k}(X^{\bullet})) \implies \Ext(E_{*}, \Tot(X^{\bullet})).
\]

This spectral sequence is usually not the same as the one obtained by applying $\Ext(E_{*}, -)$ to $X^{\bullet}$ itself, because the homological $t$-structure does not interact in an easy way with $E_{*}$-homotopy groups, unlike in the case of the standard $t$-structure on the $\infty$-category of spectra.
\end{warning}
Despite the above warning, for a certain restricted class of cosimplicial objets, we \emph{do} have an isomorphism between the spectral sequence of a cosimplicial object and its homological d\'{e}calage. 

\begin{lemma}
\label{lemma:homological_decalage_and_spectral_decalage}
Let $X^{\bullet}\colon \Delta \rightarrow \eD(E_{*}E)$ be a cosimplicial object and assume that each $X^{m}$ is a direct sum of shifts of objects in the heart, each of which is a cofree $E_{*}E$-comodule. Then
\[
F(E_{*}[t], h\Dec_{n} X^{\bullet}) \simeq \Dec_{n} F(E_{*}, X^{\bullet}),
\]
where $F(E_{*}[t], -)$ is the internal mapping spectrum in $\eD(E_{*}E)$. That is, for such cosimplicial objects, mapping out of $E_{*}[t]$ takes homological d\'{e}calage to the spectral one. 
\end{lemma}

\begin{proof}
Since $F(E_{*}, -)$ preserves totalizations, it is enough to check that $F(E_{*}, -)$ takes homological truncations to Postnikov truncations. Since $X^{\bullet}$ is levelwise a direct sum of shifts of cofree comodules, it is enough to check this in the latter case, where it is clear since $\Ext^{s}(E_{*}[t], C) = 0$ for $s > 0$ whenever $C$ is cofree.
\end{proof}

\begin{corollary}
\label{corollary:adams_spectral_sequence_for_cofree_things_same_as_homological_decalage}
The $E_{*}K$-based Adams spectral sequence in $\eD(E_{*}E)$ has a second page isomorphic to the first page of the spectral sequence associated to the homological d\'{e}calage $h\Dec_{n} (E_{*}K)^{\otimes \bullet}$ of the Amitsur resolution.
\end{corollary}

\begin{proof}
Both the Adams spectral sequence and the spectral sequence of the tower are obtained by applying $F(E_{*}[t], -)$ (for all $t$ at once) and using the relevant spectral sequence in spectra. Thus, the result then follows from Levine's work and  \cref{lemma:homological_decalage_and_spectral_decalage}, since
\[
E_{*}K^{\otimes n} \simeq (E_{*}E)^{\otimes n} \otimes_{E_{*}} K_{*} ^{\otimes_{E_*} n} 
\]
and 
\[
K_{*}^{\otimes_{E_*} n} \simeq \Tor_{E_{*}}(K_{*}, K_{*}) \otimes_{K_{*}} \ldots \otimes_{K_{*}} \Tor_{E_{*}}(K_{*}, K_{*}),
\]
where the right hand side has $(n-1)$ factors and the unadorned tensor products are the derived tensor products over $E_{*}$. This is a direct sum of shifts of $K_{*}$, as needed. 
\end{proof}
The idea is now to relate the homological d\'{e}calage of $E_{*}K$ to that of $K_{*}$. 

\begin{lemma}
\label{lemma:k_mapping_into_ek_a_decalage_isomorphism}
The morphism $K_{*} \rightarrow E_{*}E \otimes K_{*} \simeq E_{*}K$ of algebras in the derived category of $E_*E$-comodules induces an isomorphism of spectral sequence associated to the homological d\'{e}calage of Amitsur resolutions. 
\end{lemma}

\begin{proof}
Since $E_{*}E$ is flat, for any $X \in \eD(E_{*}E)$ we have $\rmH_{n}(E_{*}E \otimes X) \simeq E_{*}E \otimes \rmH_{n}(X)$. Thus, the map between the layers of the homological d\'{e}calage towers is given by applying totalization to the map 
\[
\rmH_{n}(K_{*}^{\otimes \bullet}) \rightarrow E_{*}E^{\otimes \bullet} \otimes \rmH_{n}(K_{*}^{\otimes \bullet}).
\]
The map is obtained by tensoring the map 
\[
E_{*}^{\otimes \bullet} \rightarrow E_{*}E^{\otimes \bullet},
\]
which is a quasi-isomorphism of levelwise flat cosimplicial comodules with $\rmH_{n}(K_{*}^{\otimes \bullet})$, so it will be a quasi-isomorphism again. It follows that it is an equivalence after taking totalizations. 
\end{proof}

\begin{proposition}
\label{proposition:adams_sseq_assocaited_to_e_k}
From the second page on, the Adams spectral sequence associated to $E_{*}K$ is isomorphic to the spectral sequence obtained from the tower 
\[
\ldots \rightarrow E_{*} / \mfrak^{2} \rightarrow E_{*} / \mfrak,
\]
that is, it is the filtration by powers spectral sequence. 
\end{proposition}

\begin{proof}
We know from \cref{corollary:adams_spectral_sequence_for_cofree_things_same_as_homological_decalage} that this Adams spectral sequence has second page isomorphic to the spectral sequence of the homological d\'{e}calage tower, which is in turn equivalent to the homological d\'{e}calage tower associated to the $K_{*}$-Adams resolution by \cref{lemma:k_mapping_into_ek_a_decalage_isomorphism}. Thus, it is enough to identify the latter. 

We're interested in the cosimplicial object 
\[
K_{*} \rightrightarrows K_{*} \otimes _{E_{*}} K_{*} \Rrightarrow \ldots,
\]
all tensor products being implicitly derived. Everything here is $2$-periodic and concentrated in even degrees in the internal grading, so that we can instead focus on the cosimplicial object 
\[
K_{0} \rightrightarrows K_{0} \otimes_{E_0} K_{0} \Rrightarrow \ldots,
\]
working in the derived category of $E_{0}$-modules. As an object of the latter, $K_{0} \otimes_{E_0} K_{0}$ is a direct sum of its homology groups, and we have 
\[
H_{*}(K_{0} \otimes_{E_{0}} K_{0}) \simeq \Tor^{E_{0}}_{*}(K_{0}, K_{0})
\]
and more generally 
\begin{equation}
\label{equation:homology_of_cosimplicial_kzero_resolution_and_tor}
H_{*}(K_{0}^{\otimes_{E_0} n}) \simeq \Tor_{*}^{E_{0}}(K_{0}, K_{0})^{\otimes_{K_{0}} n-1},
\end{equation}
where on the left we have the derived tensor product of $E_{0}$-modules and on the right the ordinary tensor product of graded $K_{0}$-modules.

Using that $E_{0}$ is a regular ring with residue field $K_{0}$, the 
canonical isomorphism 
\[
\Tor^{1}_{E_{0}}(K_{0}, K_{0}) \simeq (\mfrak / \mfrak^{2})^{\vee}
\]
with the tangent space extends to a grading-preserving isomorphism
\[
\Tor_{*}^{E_{0}}(K_{0}, K_{0}) \simeq \Lambda_{K_{0}}((\mfrak / \mfrak^{2})^{\vee})
\]
with the exterior algebra. This is flat over $K_{0}$, which is a field, so by standard arguments the cosimplicial object $K_{0}^{\otimes \bullet}$ on the right of (\ref{equation:homology_of_cosimplicial_kzero_resolution_and_tor}) encodes the cobar complex computing 
\[
\textnormal{Cotor}_{\Lambda_{K_{0}}((\mfrak / \mfrak^{2})^{\vee})}(K_{0}, K_{0}) \simeq \Sym_{K_{0}}(\mfrak/\mfrak^{2}) \simeq \bigoplus_{n \geq 0} \mfrak^{n}/\mfrak^{n+1}.
\]
Note that this is bigraded using the internal grading of the exterior algebra and the $\textnormal{Cotor}$-grading, with the summand $\mfrak^{n} / \mfrak^{n+1}$ being concentrated in degrees $(n, n)$.

Using the isomorphism in \cref{equation:homology_of_cosimplicial_kzero_resolution_and_tor} and passing to cohomology, we deduce that 
\[
H^{s}(H_{n}(K_{0}^{\otimes_{E_{0}} \bullet})) \simeq \begin{cases}
\mfrak^{n}/\mfrak^{n+1} & \mbox{when } $s=n$,\\
0 & \mbox{otherwise,}
\end{cases}
\]
where $H^{s}$ is the cohomology of the given cosimplicial abelian group. Thus, the totalization spectral sequence for $H_{n}(K_{0}^{\otimes_{E_{0}} \bullet})$, which we consider as a cosimplicial object of the heart of the derived $\infty$-category of $E_{0}$-modules, collapses and we deduce that 
\[
\Tot(H_{n}(K_{0}^{\otimes \bullet})) \simeq \Sigma^{-n} \mfrak^{n} / \mfrak^{n+1},
\]
It follows that the graded pieces of the homological d\'{e}calage filtration of $E_{*}$ are given by 
\[
\mathrm{fib}(\Dec_{n} \rightarrow \Dec_{n-1}) \simeq \Tot(\Sigma^{n} H_{n}(K_{*}^{\otimes \bullet})) \simeq \mfrak^{n} / \mfrak^{n+1} \otimes_{E_{0}} E_{*}
\]
and so by induction the tower itself must be of the claimed form. 
\end{proof}

\begin{theorem}
\label{theorem:k_based_adams_at_large_primes_iso_to_filtration_by_powers}
If $2p-2 > n^{2}+n+1$, the $K$-based Adams spectral sequence can be given an additional grading so that it becomes isomorphic to the spectral sequence induced by applying $\Ext_{E_{*}E}(E_{*}, -)$ to the filtration of $E_{*}$ given by powers $\mfrak^{n} E_{*}$ of the maximal ideal. 
\end{theorem}

\begin{proof}
This is a combination of \cref{proposition:k_based_adams_has_an_additional_grading_for_which_it_is_ek_adams} and \cref{proposition:adams_sseq_assocaited_to_e_k}.
\end{proof}

\begin{remark}
\label{remark:k_based_adams_at_large_primes_same_as_filtration_by_powers_in_group_cohomology}
If we pick $E = E_{n}$ to be the Morava $E$-theory of the Honda formal group law over $\mathbb{F}_{p^{n}}$, then since $\mfrak^{n} / \mfrak^{n+1} E_{*}$ is $\mfrak$-torsion, we have
\[
\Ext_{E_{*}E}^*(E_{*}, \mfrak^{n} / \mfrak^{n+1} E_{*}) \simeq \rmH_{\cts}^{*}(\mathbb{G}_{n}, \mfrak^{n} / \mfrak^{n+1} E_{*}).
\]
Thus, the spectral sequence of \cref{theorem:k_based_adams_at_large_primes_iso_to_filtration_by_powers} is isomorphic to the conditionally convergent spectral sequence
\[
\rmH_{\cts}^{*}(\mathbb{G}_{n}, \mfrak^{n}/\mfrak^{n+1} E_{*}) \implies \rmH_{\cts}^{*}(\mathbb{G}_{n}, E_{*})
\]
obtained by the same filtration of $E_{*}$, but in continuous $\mathbb{G}_{n}$-modules. 
\end{remark}

\section{Height one $K$-based Adams spectral sequence at an odd prime}
\label{sec:somecomputations}

Let us describe how using our methods one obtains an explicit description of the $K$-based Adams spectral sequence at $n=1$ and $p$ odd. Since in this case 
\[
2p-2 \geq 4 > 3 = n^{2}+n+1,
\]
\cref{theorem:k_based_adams_at_large_primes_iso_to_filtration_by_powers} and \cref{remark:k_based_adams_at_large_primes_same_as_filtration_by_powers_in_group_cohomology} imply that the relevant spectral sequence is isomorphic to the one computing cohomology $\rmH_{\cts}^{*}(\mathbb{G}_{n}, E_{*})$ using the filtration by powers of the maximal ideal, where $E_{0}$ is the Lubin--Tate ring of the Honda formal group law. 

\begin{notation}
We have $E_{*} \simeq \mathbb{Z}_{p}[u^{\pm 1}]$, $\fm = (p)$ and $\mathbb{G}_{1} \simeq \mathbb{Z}_{p}^{\times}$ with the action determined by  $\lambda_{*}(u) = \lambda u$ for $\lambda \in \mathbb{Z}_{p}$. We will denote the associated graded of $E_{*}$ by
\[
A_{*} \simeq \mathbb{F}_{p}[b][u^{\pm1}],
\]
with $|b| = (0, 1)$ the equivalence class of $p \in E_{0}$, and $|u| = (2, 0)$, where the first degree is internal and the second one is coming from the filtration.
\end{notation}

The $p$-adic group $\mathbb{Z}_{p}^{\times}$ is topologically cyclic at odd primes, generated by $\psi = \sigma (1 + p)$ for $\sigma$ a primitive $(p-1)$-th root of unity. It follows that for a profinite $\mathbb{Z}_{p}^{\times}$-module $M$, its cohomology can be computed using the $2$-term complex
\begin{equation}
\label{equation:2_term_complex_computing_cohomology}
\partial\colon M \rightarrow M,
\end{equation}
where $\partial = id_{M} - \psi_{*}$.

The complex of (\ref{equation:2_term_complex_computing_cohomology}) is a quasi-isomorphic quotient of the standard group cohomology cochain complex
\begin{equation}
\label{equation:standard_cochain_complex_computing_zptimes_cohomology}
M \rightarrow \Map_{\cts}(\mathbb{Z}_{p}^{\times}, M) \rightarrow \Map_{\cts}(\mathbb{Z}^{\times}_{p} \times \mathbb{Z}_{p}^{\times}, M) \rightarrow \ldots,
\end{equation}
the quotient map being given by the identity in degree zero and evaluation at the generator $\psi \in \mathbb{Z}_{p}^{\times}$ in degree one. 

From this explicit description we see that if we equip $M$ with a filtration, the above quotient map is a quasi-isomorphism of filtered complexes, where we consider both (\ref{equation:2_term_complex_computing_cohomology}) and  (\ref{equation:standard_cochain_complex_computing_zptimes_cohomology}) with the filtration induced from that of $M$. In particular, both filtrations induced isomorphic spectral sequences. 

\begin{remark}
The 2-term complex of (\ref{equation:2_term_complex_computing_cohomology}) can be used to directly compute the cohomology groups $\rmH_{\cts}^{*}(\mathbb{Z}_{p}^{\times}, \mathbb{Z}_{p}[u^{\pm 1}])$, recovering the classical image of $J$ pattern \cite[Lecture 35]{lurie2010chromatic}. Our goal is not to compute the above groups, which are well-known, but rather analyze the structure of the spectral sequence associated to filtration by powers of $p$. 
\end{remark}

We first determine the cohomology groups of the associated graded $A_*$. We have $\psi_{*} b = b$, because the latter can be represented by $p \in E_{0}$ which is necessarily acted on trivially, and since $\psi_{*} u \equiv \sigma u \pmod p$, we have $\psi_{*} u = [\sigma] u$ in the associated graded, where $[\sigma] \in \mathbb{F}_{p}$ is the image of $\sigma$. Thus, in the associated graded we have that $\partial(b) = 0$ and $\partial(u^{k})$ is a unit multiple of $u^{k}$ for $(p-1) \nmid k$ and zero otherwise. We deduce the following. 

\begin{proposition}
\label{proposition:e2_term_of_kass_at_n1_and_p_odd}
The $E_{2}$-term of the $K$-based Adams for $S^{0}$ at $n=1$ and $p > 2$ is given by 
\[
\Ext_{K_{*}K}^*(K_{*}, K_{*}) \cong \rmH_{\cts}^{*}(\mathbb{G}_{1}, A_{*}) \cong \rmH_{\cts}^{*}(\mathbb{Z}_{p}^{\times}, \mathbb{F}_{p}[b][u^{\pm 1}]) \cong \mathbb{F}_{p}[b][v_{1}^{\pm 1}] \otimes \Lambda(\zeta),
\]
where $v_{1} = u^{p-1}$ is of degree $(2p-2, 0, 0)$ and $\zeta$ is the class of $b$ of degree $(0, 1, 1)$, with the last degree the  cohomological one.
\end{proposition}
By \cref{remark:k_based_adams_at_large_primes_same_as_filtration_by_powers_in_group_cohomology}, if the $K$-based Adams is isomorphic to the spectral sequence computing $\rmH_{\cts}^{*}(\mathbb{Z}_{p}^{\times}, \mathbb{Z}_{p}[u^{\pm 1}])$ induced by the filtration of (\ref{equation:2_term_complex_computing_cohomology}) by powers of $p$, with $E_{2}$-term as above and with differentials $d_{r}\colon E_{r} \rightarrow E_{r}$ of degree $(0, r, 1)$. We will now analyze this spectral sequence. 

We know that $\partial(p) = 0$ in $E_{0}$, so that $b$ is a permanent cycle. The same is true for $\zeta$, for degree reasons, and we deduce that all of the differentials are both $b$ and $\zeta$-linear, and so the structure of the spectral sequence is completely determined by what it does on $v_{1}^{k}$. One computes with no difficulty that if we write $k = p^{n}m$, where $p \nmid m$, then
\[
\partial(u^{(p-1)k}) \equiv \text{unit} \cdot p^{n+1} u^{(p-1)k} \pmod {p^{n+2}}
\]
from which we deduce that 
\begin{equation}
\label{equation:differentials_in_k_based_ass_for_sphere_at_height_one_odd_prime}
    d_{i}(v_{1}^{k}) = 
    \begin{cases}
        0 & \text{for } i < n+1 \\
        \zeta b^{n+1} v_{1}^{k} & \text{for } i = {n+1}.
    \end{cases}
\end{equation}
Note that the first computation is also implied by the Leibniz rule for $p$-th powers. The extension problems are resolved using multiplication by $b$, recovering the classical answer that 
\[
\rmH_{\cts}^{s, t}(\mathbb{Z}_{p}^{\times}, \mathbb{Z}_{p}[u^{\pm 1}]) = \begin{cases} \mathbb{Z}_{p} & \mbox{if } (s, t) = (0, 0), \\
\mathbb{Z}_{p} \{\zeta\} & \mbox{if } (s, t) = (1, 0), \\
\mathbb{Z}/p^{n+1} \{\zeta u^{(p-1)p^{n}m} \} & \mbox{if } (s, t) = (1, (2p-2)p^{n}m) \mbox{, where } p \nmid m, \\
0 & \mbox{otherwise,} \end{cases}
\]
where $\{-\}$ denotes a generator of the given group. Thus, the spectral sequence is as in the following familiar ``image of J'' pattern
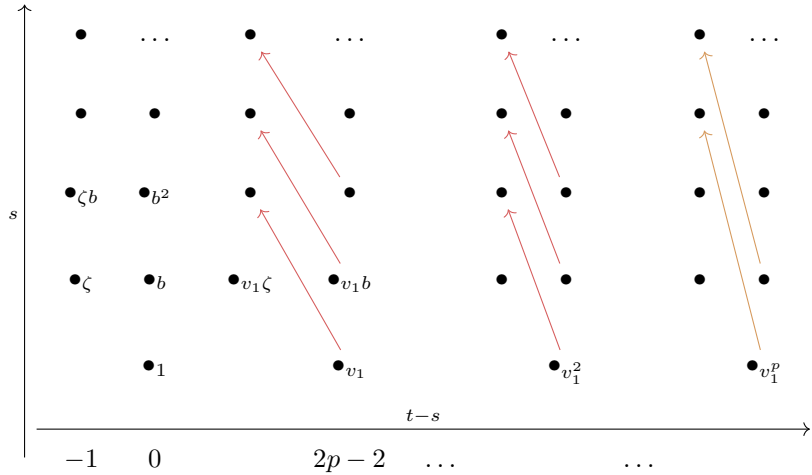
\[
\begin{tikzcd}[column sep=tiny]
	&& {} \\
	&& {} & \bullet & \ldots && \bullet & \ldots && \bullet & \ldots && \bullet & \ldots \\
	&&& \bullet & \bullet && \bullet & \bullet && \bullet & \bullet && \bullet & \bullet \\
	&&& {\bullet_{\zeta b}} & {\bullet_{b^{2}}} && \bullet & \bullet && \bullet & \bullet && \bullet & \bullet \\
	{} &&& {\bullet_{\zeta}} & {\bullet_{b}} && {\bullet_{v_{1} \zeta}} & {\bullet_{v_{1}b}} && \bullet & \bullet && \bullet & \bullet \\
	{} &&& {} & {\bullet_{1}} &&& {\bullet_{v_{1}}} &&& {\bullet_{v_{1}^{2}}} &&& {\bullet_{v_{1}^{p}}} \\
	&& {} & {-1} & 0 &&& {2p-2} & \ldots &&& \ldots &&& {} \\
	&& {} &&&&&&&& {} && {} & {} & {} & {}
	\arrow[color={rgb,255:red,214;green,92;blue,92}, from=6-8, to=4-7]
	\arrow[color={rgb,255:red,214;green,92;blue,92}, from=5-8, to=3-7]
	\arrow[color={rgb,255:red,214;green,92;blue,92}, from=6-11, to=4-10]
	\arrow[color={rgb,255:red,214;green,92;blue,92}, from=5-11, to=3-10]
	\arrow[color={rgb,255:red,214;green,153;blue,92}, from=6-14, to=3-13]
	\arrow["{t-s}", shift left=5, from=7-3, to=7-15]
	\arrow[color={rgb,255:red,214;green,92;blue,92}, from=4-8, to=2-7]
	\arrow[color={rgb,255:red,214;green,92;blue,92}, from=4-11, to=2-10]
	\arrow[color={rgb,255:red,214;green,153;blue,92}, from=5-14, to=2-13]
	\arrow["s", shorten >=12pt, from=7-3, to=1-3]
\end{tikzcd}
\]
with even longer differentials supported on $v_{1}^{p^{k}}$ for $k \geq 1$. 

It is interesting to observe that this shows that already at height one, the $K$-based Adams spectral sequence has non-zero differentials of arbitrary length. 

\begin{example}[Non-completely convergent $K$-Adams]
\label{ex:noncompleteconvergence}
Using the above calculation, we can give an example of a $K$-local spectrum for which its $K$-based Adams spectral sequence is not completely convergent. Note that we know that it is always conditionally convergent by \cref{prop:knnilpotentcompletion}. 

To this end, we will construct a spectrum $X$ for which $\lim_r^1E_r^{0,0}(X)$ does not vanish at $n =1$ and $p > 2$. Set $\lambda_k = 2(p-1)p^k$ and consider the spectrum
\[
X = L_{K(1)}\bigoplus_{k\ge 0}S^{-\lambda_k}.
\]
In bidegree $(s,t) = (0,0)$, the $E_2$-page is given by
\[
E_2^{0,0}(X) \cong \bigoplus_{k\ge 0} \F_p\{v_1^{p^k}\},
\]
where the generators correspond to the maps $v_1^{p^k}\colon K(n)_* \to K(n)_*S^{\lambda_k}$. The formulas for the differentials of \cref{equation:differentials_in_k_based_ass_for_sphere_at_height_one_odd_prime} show that (up to reindexing) the filtration on $E_2^{0,0}(X)$ is given by 
\[
\ldots \subset \bigoplus_{k \ge 2}\F_p\{v_1^{p^k}\} \subset \bigoplus_{k \ge 1}\F_p\{v_1^{p^k}\} \subset \bigoplus_{k \ge 0}\F_p\{v_1^{p^k}\} = E_2^{0,0}(X).
\]
Therefore, we obtain an exact sequence
\[
0 \to \lim_rE_r^{s,t}(X) \to E_2^{0,0}(X) = \bigoplus_{k\ge 0} \F_p\{v_1^{p^k}\} \to \prod_{k \ge 0}\F_p\{v_1^{p^k}\} \to \lim_r^1E_r^{s,t}(X) \to 0. 
\]
In particular, we see that $\lim_rE_r^{s,t}(X) = 0$, while $\lim_r^1E_r^{s,t}(X) \neq 0$.
\end{example}

\part{Homology of inverse limits and the algebraic chromatic splitting conjecture}

In the third part of the work, we largely move away from Morava $K$-theory, and we focus on homology of $K$-local spectra. From the perspective of $E$-local category, $K$-localization is a form of completion, and so we first focus on homology of inverse limits, construction a spectral sequence computing these in the generality of an arbitrary adapted homology theory. 

In the case of $K$-localization, the $E_{2}$-page of this spectral sequence is given by derived functors of limits in $E_{*}E$-comodules. We describe the latter in terms of cohomology of the Morava stabilizer group, and use it to compute the zeroth one at all heights and primes. At height one, we compute these derived functors completely. 

\section{The homology of inverse limits}
\label{section:homology_of_inverse_limits}

We begin with a general review of the context for the construction of the modified Adams spectral sequence, following the approach of Devinatz and Hopkins \cite{dev_morava}. This approach will then be employed to construct a spectral sequence that computes the homology of inverse limits in a presentable stable $\infty$-category.

In particular, this gives rise to a spectral sequence for computing the $E$-homology of the inverse limit of a tower of spectra $(X_{\alpha})$ from the derived functors of inverse limits of the tower of $E_*E$-comodules $(E_*X_{\alpha})$, for suitable ring spectra $E$. The question of when this spectral sequence converges is subtle, and will be studied in detail in the case of Morava $E$-theory. The material in this section is based on unpublished work of Mike Hopkins and Hal Sadofsky.

\subsection{Adapted homology theories}\label{ssec:adaptedhomology}

In this short section we will recall basic facts about adapted homology theories. Informally, these are exactly those homology theories which admit an Adams spectral sequence based on injectives. Everything here is classical, although our presentation will be most close to \cite[\S 2]{patchkoria2021adams}.

A locally graded $\infty$-category is an $\infty$-category equipped with a distinguished autoequivalence $[1]_{\ccat}\colon \ccat \rightarrow \ccat$. A locally graded functor $f\colon \ccat \rightarrow \dcat$ is functor equipped with a natural isomorphism $f \circ [1]_{\ccat} \simeq [1]_{\dcat} \circ f$. 

\begin{definition}
Let $\ccat$ be a stable $\infty$-category, considered as a locally graded $\infty$-category using the suspension functor, and $\acat$ an arbitrary locally graded abelian category. We say a locally graded functor $H\colon \ccat \rightarrow \acat$ is a \emph{homology theory} if for any cofibre sequence 
\[
X \rightarrow Y \rightarrow Z
\]
in $\ccat$, the induced diagram 
\[
H(X) \rightarrow H(Y) \rightarrow H(Z)
\]
in $\acat$ is exact in the middle. 
\end{definition}

\begin{remark}
The requirement that $H$ is locally graded amounts to specifying an isomorphism 
\[
H(\Sigma X) \simeq H(X)[1]
\]
natural in $X \in \ccat$. 
\end{remark}

In cases of interest to us, $H$ will be \emph{Grothendieck} in the sense that 
\begin{enumerate}
    \item $\ccat$ is presentable, 
    \item $\acat$ is Grothendieck abelian and 
    \item $H$ preserves arbitrary direct sums.
\end{enumerate}
In particular, this means that $\acat$ has enough injectives. Any injective object $I$ of $\acat$ gives rise to a cohomological functor
\[
\Hom_{\acat}(H(-),I)\colon \ccat^{\op} \to \acat.
\]
Since $\ccat$ is presentable and stable, Brown's representability theorem holds~\cite[Proposition 2.15]{patchkoria2021adams}, so that there exists some $D(I) \in \Ho\ccat$ representing $\Hom_{\acat}(H(-),I)$. It follows that, for any object $X \in \ccat$, there is a natural equivalence
\[
\Hom_{\acat}(H(X),I) \simeq \pi_{0}\Hom_{\ccat}(X,D(I)).
\]
Note that the object $D(I)$ is $H$-local, in the sense that for any object $Y \in \ccat$ with $H(Y) = 0$ we have $\Hom(Y,D(I)) =0$. Moreover, the identity map on $D(I)$ corresponds under this isomorphism to a natural counit map $H(D(I)) \to I$. 

\begin{definition}\label{def:adapted}
We say $H\colon \ccat \to \acat$ is \emph{adapted} if the counit map $H(D(I)) \to I$ is an equivalence for any injective $I \in \acat$.
\end{definition}

\begin{example}
Rational homology, viewed as a functor from spectra to graded rational vector spaces, is an adapted homology theory. In this case, the lift $D(V)$ of a (graded) $\Q$-vector space $V$ is the (generalized) Eilenberg--MacLane spectrum $HV$.
\end{example}

\begin{example}
\label{example:mod_p_homology_adapted_when_valued_in_comodules}
In contrast to the previous example, mod $p$ homology is not adapted for any $p$; indeed, the lift of $\F_p$ is $H\F_p$, but its mod $p$-cohomology is the mod $p$ Steenrod algebra $\cA_p$ which is larger than $\F_p$. 

However, $H_*(-,\F_p)$ has more structure: it takes values in the category $\Comod_{\cA_p}$ of graded $\cA_p$-comodules. It turns out that the homology theory
\[
H_*(-;\F_p)\colon \Sp \to \Comod_{\cA_p}
\]
is adapted. 
\end{example}

\begin{remark}
The phenomena visible in  \cref{example:mod_p_homology_adapted_when_valued_in_comodules} is typical in the sense that asking for $H$ to be adapted is to ask for $\acat$ to encode all available homological information. This generalizes considerably, one can show that any homology theory factors uniquely through an adapted one followed by an exact comonadic functor of abelian categories \cite[\S 3.3]{patchkoria2021adams}
\end{remark}

In fact, the previous example generalizes considerably. Let $R$ be a homotopy commutative ring spectrum which satisfies the Adams condition, that is, it can be written as a filtered colimit of finite spectra $X_{\alpha}$ with the properties that 

\begin{enumerate}
    \item $R_*X_{\alpha}$ is a finitely generated projective $R_*$-module and 
    \item the K\"{u}nneth map $E^{*}X_{\alpha} \rightarrow \Hom_{E_{*}}(E_{*}X_{\alpha}, E_{*})$ is an isomorphism. 
\end{enumerate}
Note that the second condition follows from the first whenever $R$ can be made $\mathbf{E}_{1}$. 

For example, $H \mathbb{F}_{p}$ or any Landweber exact homology theory satisfy the Adams condition \cite[Theorems 1.4.7-9]{hovey_htptheory}, but $H\Z$ does not~\cite[Page  17]{goersshopkins_moduli}.  This condition guarantees that the associated Hopf algebroid $(R_*,R_*R)$ is Adams and thus has a well-behaved category of comodules $\Comod_{R_*R}$; see \cite[Section 1.4]{hovey_htptheory} for details.

One way in which the Adams-type condition implies that the category $\Comod_{R_{*}R}$ is well-behaved also in the topological sense is the following result of Devinatz. 

\begin{lemma}[{\cite[Theorem 1.5 and Section 2]{dev_morava}}]\label{lem:dlift}
Let $R$ be a topologically flat commutative ring spectrum. For any injective $R_*R$-comodule $I$, the identity map on $D(I)$ induces an isomorphism $R_*D(I) \cong I$. In other words, the functor
\[
R_0\colon \spectra \to \Comod_{R_*R}
\]
is an adapted homology theory. 
\qed
\end{lemma}

\subsection{Construction of the spectral sequence}
\label{subsection:construction_of_the_spectral_sequence}

The starting point of our construction is the existence of modified Adams towers for any homology theory $H\colon \ccat \to \acat$. We collect its properties in the next result: 

\begin{lemma}\label{lem:modadamstower}
Consider an adapted homology theory $H\colon \ccat \to \acat$. Let $X \in \ccat$ and suppose
\[
\xymatrix{0 \ar[r] & H(X)=C_0 \ar[r]^-{\eta_0} & I_0 \ar[r]^-{\tau_0} & I_1 \ar[r]^-{\tau_1} & \ldots \\
& & C_0 \ar[u]^-{\eta_0} & C_1 \ar[u]^-{\eta_1} & \ldots}
\]
is an injective resolution of $H(X)$ in $\acat$, where $C_i = \ker(\tau_i)$ for all $i$. There exists a tower of objects in $\ccat$ over $X$ of the form
\[
\xymatrix{X=X_0 \ar[d]_{f_0} & X_1 \ar[l]_-{g_0} \ar[d]_{f_1} & X_2 \ar[l]_-{g_1} \ar[d]_{f_2} & \ldots \ar[l]_-{g_2} \\ 
\Sigma^{0}D(I_0) \ar@{-->}[ru]_{\delta_0} & \Sigma^{-1}D(I_1) \ar@{-->}[ru]_{\delta_1} & \Sigma^{-2}D(I_2) \ar@{-->}[ru]_{\delta_2} & \ldots,}
\]
where the dotted maps are the boundary maps shifting degree by $1$, and  such that the following properties are satisfied for all $i$:
\begin{enumerate}
	\item $X_{i+1} \simeq \fib(f_i)$;
	\item $H(X_i) \cong C_i[{-i}]$;
	\item $H(f_i) = \eta_i[{-i}]$;
	\item $H(g_i)=0$;
	\item $H(\delta_i)$ is the surjection $I_i[{-i}] \to C_{i+1}[{-i}]$.  
\end{enumerate}
If $X$ is $H$-local, then $X_i$ and $D(I_i)$ are $H$-local for all $i$. Moreover, this tower is weakly functorial; that is, a map $X \to Y$ can be extended to a (non-canonical) map of towers $(X_i) \to (Y_i)$. 
\end{lemma}
\begin{proof}
This lemma is an axiomatization of the results proven in \cite[Section 1]{dev_morava}. The tower so constructed consists of $H$-local spectra as noticed in \cref{ssec:adaptedhomology}, while naturality follows from the naturality of injective resolutions.
\end{proof}

\begin{notation}
We refer to the tower $(X_i)$ constructed in \cref{lem:modadamstower} as the modified Adams tower of $X$ with respect to the adapted homology theory $H\colon \ccat \to \acat$.
\end{notation}

Since the limit $\varprojlim I_{\alpha}$ of an injective tower of injectives in $\acat$ is injective, the construction of the previous subsection yields an object $D(\varprojlim I_{\alpha})$ of $\ccat$. In order to proceed, we need an auxiliary characterization of injective towers of objects in a Grothendieck abelian category $\acat$.

\begin{lemma}\label{lem:charinjtower}
Let $\acat$ be a Grothendieck abelian category. An object $I = (I_{\alpha}) \in \acat^{\N}$ is injective if and only if $I_{\alpha}$ is injective for all $\alpha \in J$ and all structure maps in $I$ are split epimorphisms. 
\end{lemma}
\begin{proof}
This is proven in \cite[Proposition 1.1]{jannsen_contetcohom}. We sketch the argument of the implication we need. Let $j\colon \N^{\delta} \to \N$ be the inclusion of the discrete set of natural numbers into the poset of natural numbers. This induces an adjunction $(j^*,j_*)$ between diagram categories. If $I$ is injective, then the natural monomorphism $I \to j_*j^*I$ is split. By construction, the structure maps of $j_*j^*I$ are projections, so $I$ also has split surjective structure maps, and the claim follows.
\end{proof}

\begin{proposition}\label{prop:dlim}
If $I=(I_{\alpha}) \in \Inj(\cA^{\N})$ is an injective tower in $\acat$, then there is a preferred equivalence
\[
\xymatrix{D(\varprojlim_{\acat} I_{\alpha}) \ar[r]^-{\sim} & \varprojlim_{\ccat} D(I_{\alpha}).}
\]
of objects in $\ccat$, well-defined up to homotopy. 
\end{proposition}
\begin{proof}
For brevity, let us write $I_{\infty} := \varprojlim_{\acat} I_{\alpha}$; note that this is an injective object of $\acat$. The canonical structure maps $\varprojlim_{\acat)} I_{\alpha} \to I_{\alpha}$ give arrows
\[
D(I_{\infty}) \rightarrow D(I_{\alpha}),
\]
well-defined and compatible up to homotopy. By the Milnor exact sequence, these can be lifted to a homotopy class of maps 
\[
\xymatrix{\phi\colon D(I_{\infty}) \ar[r] & \varprojlim D(I_\alpha),}
\]
in $\ccat$. We claim this is an equivalence, which we will check by verifying that both sides represent the same functor in the homotopy category. 

Indeed, for any $X \in \ccat$, there are isomorphisms
\begin{align*}
\textstyle
[X,D(I_{\infty})]_* & \cong \Hom_{\acat}^*(H(X), I_{\infty}) \\
& \cong \varprojlim \ \Hom_{\acat}^*(H(X), I_\alpha) \\
& \cong \varprojlim \ [X,D(I_\alpha)]_*.
\end{align*}
The Milnor sequence associated to the limit $\varprojlim D(I_{n})$ takes the form 
\begin{equation}\label{eq:milnorseq}
\xymatrix{0 \ar[r] & \varprojlim^1[X,D(I_{n})]_{*+1} \ar[r] & [X,\varprojlim D(I_{\alpha})]_* \ar[r] & \varprojlim[X,D(I_{\alpha})]_{*} \ar[r] & 0.}
\end{equation}
It follows from \cref{lem:charinjtower} that the structure maps in the tower 
\[
([X,D(I_{\alpha})]_{*+1}) \cong (\Hom_{\acat}^{*+1}(H(X),I_{\alpha}))
\]
are split epimorphisms, so the $\varprojlim^1$-term in \eqref{eq:milnorseq} vanishes. Thus, the Yoneda lemma implies that $\phi$ is an equivalence, as claimed. 
\end{proof}

The spectral sequence we construct will involve derived functors of the limit, and so we first establish a consistent notation for these. 
\begin{notation}
If $\acat$ is a Grothendieck abelian category and $F\colon J \rightarrow \acat$ is a diagram, then we write $\textstyle \varprojlim_{\acat} F$ for the limit taken in $\acat$ itself. Formation of limits defines a functor 
\[
\textstyle\varprojlim_{\acat}\colon \Fun(J, \acat) \rightarrow \acat
\]
and we denote its right derived functors by 
\[
\textstyle\varprojlim^{s}_{\acat} F.
\]

Alternatively, $\acat$ has an associated derived $\infty$-category, and we write 
\[
\textstyle\varprojlim_{\dcat(\acat)} F
\]
for the limit of the composite of $F$ with the inclusion with the heart. These two different notions of limits, one in $\acat$ and the other in the derived $\infty$-category, are related by the formula 
\[
\textstyle\varprojlim^{s}_{\acat} F \simeq H_{-s}(\textstyle\varprojlim_{\dcat(\acat)} F),
\]
where on the right hand side we have the homology of an object of the derived $\infty$-category. Thus, $\textstyle\varprojlim_{\dcat(\acat)} F$ can be thought of as the total derived functor of the limit. 
\end{notation}

\begin{theorem}
\label{thm:limss}
Let $H\colon \ccat \to \acat$ be an adapted homology theory. If $(X^{\alpha}) \in \ccat^{\N}$ is a tower in $\ccat$, then there exists a natural spectral sequence in $\acat$ of the form
\begin{equation}\label{eq:generalss}
\textstyle E_2^{s,t} \cong \varprojlim_{\acat}^s H(X^{\alpha})[-t]_{\acat} \implies H(\varprojlim X^{\alpha})[s-t]_{\acat}
\end{equation}
with differentials $d_r^{s,t}\colon E_r^{s,t} \to E_r^{s+r,t+r-1}$.
\end{theorem}
\begin{proof}
If $H$ is adapted, then so is the induced homology theory 
\[
H\colon \ccat^{\N} \rightarrow \acat^{\N}
\]
between the $\infty$-categories of towers \cite[Example 8.24]{patchkoria2021adams}. Thus, we have Adams resolutions of towers and we let $(X_i^{\alpha})_{i \geq 0}$ 
be a modified Adams tower of $(X^{\alpha})$ as in \Cref{lem:modadamstower}. 

Taking limits in the Adams resolution of $X^{\alpha}$ yields a tower in $\ccat$ of the form 
\[
\xymatrix{\varprojlim X_0^{\alpha} \ar[d]_{f_0} & \varprojlim X_1^{\alpha} \ar[l]_-{g_0} \ar[d]_{f_1} & \varprojlim X_2^{\alpha} \ar[l]_-{g_1} \ar[d]_{f_2} & \ldots \ar[l]_-{g_2} \\ 
\varprojlim\Sigma^{0}D(I_0^{\alpha}) \ar@{-->}[ru]_{\delta_0} & \varprojlim\Sigma^{-1}D(I_1^{\alpha}) \ar@{-->}[ru]_{\delta_1} & \varprojlim\Sigma^{-2}D(I_2^{\alpha}) \ar@{-->}[ru]_{\delta_2} & \ldots.}
\]
For brevity, let us write $[k] := [k]_{\acat}: \acat \rightarrow \acat$ for the internal degree shift. Setting $D_1^{s,t} = H(\varprojlim X_{s}^{\alpha})[s-t]$ and $E_1^{s,t}=H(\varprojlim \Sigma^{-s}D(I_s^{\alpha}))[s-t]$, this induces an exact couple of bigraded objects in $\acat$ of the form 
\[
\xymatrix{D \ar[rr]^-{i_1=H(g_i)[s-t]}  && D \ar[ld]^{j_1=H(f_i)[s-t]} \\
& E, \ar[lu]^{k_1=H(\delta_i)[s-t]} }
\] 
where the maps have bidegree $|i_1|=(-1,-1)$, $|j_1|=(0,0)$, and $|k_1|=(1,0)$. The $E_1$-term can be identified using \Cref{prop:dlim} and \Cref{lem:dlift} as
\[
\textstyle 
E_1^{s,t} = \varprojlim \Sigma^{-s}D(I_s^{\alpha})[s-t] \cong D(\varprojlim_{\acat} I_s^{\alpha})[-t] = (\varprojlim_{\acat} I_s^{\alpha})[-t]
\]
and $d_1 = \varprojlim_{\acat} \tau_s^{\alpha}$. This gives the $E_2$-term
\[
\textstyle E_2^{s,t} \cong \varprojlim_{\acat}^s(H_*X^{\alpha})[-t],
\]
as needed. It is clear from this construction and \Cref{lem:modadamstower} that the resulting spectral sequence is natural in the tower $X^{\alpha}$ and that all pages and differentials are in $\acat$. 
\end{proof}

\begin{warning}
Note that since the local grading $[1]_{\acat}\colon \acat \rightarrow \acat$ is an equivalence, it commutes with derived functors of the limit, and in the spectral sequence of \cref{thm:limss} we have 
\[
E_{2}^{s, t} := \textstyle\varprojlim_{\acat}^{s} H(X^{\alpha})[-t]_{\acat} \simeq (\textstyle\varprojlim_{\acat}^{s} H(x^{\alpha}))[-t]_{\acat} \simeq E_{2}^{s}[-t]_{\acat}
\]
Thus, the $t = 0$ line in the spectral sequence already determines the whole $E_{2}$-page. This is similar to the case of the Bockstein spectral sequence, which also has a periodic $E_{2}$-page. 
\end{warning}

\begin{remark}
Suppose that $(X^{\alpha})$ is a tower under $X$. Then, the image 
\[
\mathrm{im}(H(X)) \subseteq \textstyle\varprojlim^{0}_{\acat} H(X^{\alpha})
\]
consists of permanent cycles; that is, it is in the kernel of all the differentials in the spectral sequence of \cref{thm:limss}.

To see this, note that the assumption gives a map from the constant tower on $X$ to $(X^{\alpha})$, inducing a map of spectral sequences
\[
\xymatrix{\varprojlim_{\acat}^sH(X) \ar@{=>}[r] \ar[d] & H(X) \ar[d] \\
\varprojlim^{s}_{\acat} H(X^{\alpha}) \ar@{=>}[r] & H(\varprojlim X^{\alpha})}
\]
The $E_2$-page of the top spectral sequence is concentrated in 
\[
E_2^{0,t} \cong \textstyle\varprojlim_{\acat}^{0}H(X)[-t]_{\acat} \cong H(X)[-t]_{\acat},
\]
so that the spectral sequence collapses. The comparison map is readily identified with the induced  morphism $H(X) \to \varprojlim_{\acat} H(X^{\alpha})$, so the claim follows.
\end{remark}

\subsection{Derived limits of comodules}

We now specialize the above to the case of a homology theory on spectra corresponding to a ring spectrum of Adams type. In order to study the $E_2$-term of the spectral sequence of \cref{thm:limss}, we begin with a general discussion of inverse limits of comodules over suitable flat Hopf algebroids.

Let $(A, \Psi)$ be an Adams Hopf algebroid in the sense of \cite{hovey_htptheory}[1.4.3], so that $\Psi$ is a filtered colimit of dualizable comodules. We will describe an approach to computing limits in the derived $\infty$-category $\dcat(\Psi) = \dcat(\Comod_{\Psi})$ of comodules by reducing to the case of modules. The forgetful functor $\epsilon_*\colon \Comod_{\Psi} \to \Mod(A)$ is \emph{left} adjoint to the extended (or cofree) comodule functor $\epsilon^*$ which sends an $A$-module $M$ to $\Psi \otimes_A M$. These functors are exact and thus give rise to an adjunction:
\[
\epsilon_{*} \dashv \epsilon^{*}\colon \dcat(\Psi) \leftrightarrows \dcat(A).
\]
In terms of chain complexes, both can be computed levelwise.

The corresponding monad $\epsilon^*\epsilon_*$ can be used to construct a resolution of any comodule. Concretely, if $M$ is a $\Psi$-comodule, then we have the associated cobar complex
\[
0 \rightarrow M \rightarrow \Psi \otimes _{A} M \rightarrow \Psi \otimes _{A} \Psi \otimes _{A} M \rightarrow \ldots,
\]
which is a resolution of $M$ by extended comodules, that is, those of the form $\epsilon^{*} N = \Psi \otimes _{A} N$ for an $A$-module $N$. 

One way to phrase that the cobar complex is exact is to say that the augmented cosimplicial object from which it arises, namely
\begin{equation}\label{eq:amitsur}
    M \rightarrow \Psi \otimes _{A} M \rightrightarrows \Psi \otimes _{A} \Psi \otimes_{A} M \Rrightarrow \ldots,
\end{equation}
is a limit diagram in the derived $\infty$-category $\dcat(\Psi)$. Now, since limit diagrams are stable under levelwise limits, it follows that if $(M_{i})$ is a diagram of comodules, then the diagram 
\[
\textstyle\varprojlim _{\dcat(\Psi)} M_{i} \rightarrow \varprojlim _{\dcat(\Psi)} \Psi \otimes _{A} M_{i} \rightrightarrows \varprojlim _{\dcat(\Psi)} \Psi \otimes _{A} \Psi \otimes_{A} M_i \Rrightarrow \ldots,
\]
is also limit, giving an approach to computing $\varprojlim _{\dcat(\Psi)} M_{i}$. To see this, notice that since $\epsilon^{*}$ is a right adjoint and thus preserves limits, we have 
\[
\textstyle\varprojlim _{\dcat(\Psi)} \Psi \otimes _{A} N_{i} \simeq \varprojlim _{\dcat(\Psi)} \epsilon^{*} N_{i} \simeq \epsilon^{*} (\varprojlim _{\dcat(A)} N_{i}) \simeq \Psi \otimes _{A} (\varprojlim _{\dcat(A)} N_{i}).
\]
Applying this to the cosimplicial diagram \eqref{eq:amitsur}, we see that we get a limit diagram of the form 
\[
\textstyle\varprojlim _{\dcat(\Psi)} M_{i} \rightarrow \Psi \otimes _{A} (\varprojlim _{\dcat(A)} M_{i}) \rightrightarrows \Psi \otimes _{A} (\varprojlim _{\dcat(A)} \Psi \otimes_A M_{i}) \Rrightarrow \ldots
\]
which expresses the limit in the derived category of comodules using only limits taken in the category of modules. Passing to homology, which encodes the derived functors of the limit in the categories of $A$-modules and $\Psi$-comodules, we deduce the following. 

\begin{proposition}
\label{prop:spectral_sequence_computing_limits_of_comodules}
If $M_{i}$ is a diagram of comodules, then there exists a spectral sequence of comodules of signature
\[
E_{1}^{s, t} := \textstyle\Psi \otimes_{A} \varprojlim_{A}^s  (\Psi^{\otimes_{A} t} \otimes_{A} M_{i}) \implies  \varprojlim_{\Psi}^{s+t} M_{i}.
\]
computing the derived functors of the limit in comodules. 
\end{proposition}

\begin{remark}
Note that in the particular case of sequential inverse limits, that is, diagrams of the form 
\[
\ldots \rightarrow M_{2} \rightarrow M_{1} \rightarrow M_{0},
\]
this spectral sequence will collapse after at most two pages, since in this case we have $\textstyle\varprojlim^{s} _{A} N_{i} = 0$ for $s \neq 0, -1$ and any diagram $N_{i}$ of $A$-modules.
\end{remark}

\begin{remark}
\label{rem:limit_of_comodules_as_a_chain_complex_under_mittag_leffler}
In an even more specific case, let us assume that we work with sequential inverse limits and that the maps $M_{n+1} \rightarrow M_{n}$ are all epimorphisms. If that is the case, the same is true for $\Psi \otimes _{A } \ldots \otimes _{A} M_{n+1} \rightarrow \Psi \otimes _{A} \ldots \otimes _{A} M_{n}$ and we deduce using the Mittag-Leffler condition that all of the $\textstyle\varprojlim^{1} _{A}$-terms vanish. In this case the spectral sequence of  \cref{prop:spectral_sequence_computing_limits_of_comodules} has only one potentially non-zero differential.

Phrasing it without the use of spectral sequences, we see that in this the diagram of comodules
\[
\textstyle\Psi \otimes _{A} (\varprojlim _{A} M_{i}) \rightarrow \Psi \otimes _{A} (\varprojlim _{A} \Psi \otimes _{A} M_{i}) \rightarrow \ldots,
\]
is an explicit presentation of the derived limit $\varprojlim _{\dcat(\Psi)} M_{i}$ as a chain complex. 
\end{remark}

\begin{example}\label{ex:limss}
Suppose that $R$ is a homotopy commutative ring spectrum of Adams-type. As $R_{*}: \spectra \rightarrow \Comod_{R_{*}R}$ is an adapted homology theory by \cref{lem:dlift}, \cref{thm:limss} specializes to give a spectral sequence
\[
\textstyle E_2^{s,t} \cong \varprojlim_{R_*R}^s(R_{*+t}X^{\alpha}) \implies R_{*+s+t}(\lim_{\alpha} X^{\alpha})
\]
for a tower of spectra $(X^{\alpha})$, where the derived functors of the limit are taken in $R_{*}R$-comodules. 

This is a spectral sequence in $R_{*}R$-comodules, so that in particular the $E_{2}$-page is \emph{trigraded}. Passing to internal degree zero, we obtain a more-pleasant looking bigraded spectral sequence 
\[
\textstyle E_2^{s,t} \cong (\varprojlim_{R_*R}^s R_{*}X^{\alpha})_{t} \implies R_{s+t}(\lim_{\alpha} X^{\alpha})
\]
Note that in general we cannot rewrite the $E_{2}$-page as ``$\textstyle\varprojlim^{s}_{R_{*}R}(R_{t}X^{\alpha})$'', as the derived functors of the limit are computed in $R_{*}R$-comodules, a structure which cannot be restricted to a single degree.
\end{example}

\begin{warning}
Even for quite reasonable $R$, such as $R = H\Q$, convergence of the spectral sequence of \cref{ex:limss} is a subtle problem. For example, let $(M_{p^i})$ be the tower of mod $p^i$ Moore spectra with the canonical structure maps, and suppose $R=H\Q$. Since
\[
\lim(\ldots \longrightarrow M_{p^2} \longrightarrow M_{p}) \simeq S_p^0,
\]
the $p$-complete sphere, the abutment of the spectral sequence is non-trivial, while the $E_2$-term is zero. We will study this question in more detail in the next section in the case of Morava $E$-theory. 
\end{warning}

\begin{remark}
Same methods as those leading to \cref{ex:limss} have been employed by Hovey to set up a spectral sequence computing the $R$-homology of a product of spectra, see \cite{hovey_product}. In \cite[Appendix A]{peterson_coalg}, a different construction of the spectral sequence of \cref{ex:limss} is given, based on Adams resolutions rather than modified Adams resolutions.
\end{remark}

\section{The Morava $E$-homology of inverse limits}
\label{section:morava_e_homology_of_inverse_limits}

The goal of this section is to study the spectral sequence constructed in \cref{thm:limss} in the special case of (uncompleted) Morava $E$-homology. In particular, we will identify the $E_2$-page for the sphere spectrum in terms of the continuous cohomology of the Morava stablizer group and exhibit conditions on the tower that ensure convergence.  

We will work with a particular form of Morava $E$-theory, which we used previously in \S\ref{sec:cohomology}. For convenience of the reader, let us fix our notation for the remainder of the current work. 

We fix a prime $p$ and a height $n > 0$ and we let $\mathbf{G}_{0}$ be the Honda formal group law over $\mathbb{F}_{q} := \mathbb{F}_{p^{n}}$. We write $E$ for the associated Lubin-Tate spectrum, so that we have a non-canonical isomorphism 
\[
E_{*} \simeq W(\mathbb{F}_{p^{n}})\llbracket u_{1}, \ldots, u_{n-1}\rrbracket[u^{\pm 1}].
\]
We write $\mfrak = (p,u_1,\ldots, u_{n-1})$ for the maximal ideal of $E_0$. By Goerss-Hopkins-Miller, the spectrum $E$ is acted on by the extended Morava stabilizer group $\mathbb{G}_{n} := \mathrm{Aut}(\mathbf{G}_{0} / \mathbb{F}_{p}) \rtimes \mathrm{Gal}(\mathbb{F}_{q} / \mathbb{F}_{p})$.

\begin{warning}
We remind the reader that, unless otherwise noted, we work with uncompleted $E$-homology, the uncompleted homology cooperations $E_{*}E := \pi_{*}(E \otimes E)$, and uncompleted $E_*E$-comodules.
\end{warning}

\subsection{Continuous cohomology of filtered colimits} 

In this subsection, we prove that continuous cohomology groups of $\mathbb{G}_n$ commute with certain filtered colimits. To this end, it will be useful for us to consider the following variant on the $\mfrak$-adic topology on an $E_{0}$-module. 

\begin{definition}
Let $M$ be an $E_{0}$-module. The \emph{local $\mfrak$-adic} topology on $M$ is the linear topology in which a submodule $U \subseteq M$ is open if for every finitely generated $E_{0}$-submodule $M^{\prime} \subseteq M$,
\[
U \cap M^{\prime} \subseteq M^{\prime}
\]
is open in the usual $\mfrak$-adic topology on $M^{\prime}$; that is, we have $\mfrak^{n} M^{\prime} \subseteq U \cap M^{\prime}$ for some $n$ depending on $M^{\prime}$. 
\end{definition}

\begin{example}
Suppose that $M$ is finitely generated over $E_{0}$. Then, the local $\mfrak$-adic topology on $M$ coincides with the usual $\mfrak$-adic topology. 
\end{example}

\begin{example}
\label{example:local_m_adic_topology_a_filtered_colimit}
Suppose that $M$ is an arbitrary $E_{0}$-module. Then, we can write $M \simeq \varinjlim M_{\alpha}$ as a filtered colimit of finitely generated $E_0$-modules. The local $\mfrak$-adic topology coincides with the colimit topology if we equip each $M_{\alpha}$ with its usual $\mfrak$-adic topology. 
\end{example}
The importance of the local $\mfrak$-adic topology in our context comes down to the following observation.

\begin{remark}
Let $X$ be a spectrum. Then, $E_{0}X$ is a continuous $\mathbb{G}_{n}$-module with respect to its local $\mfrak$-adic topology. To see this, observe that we can write $X \simeq \varinjlim X_{\alpha}$ as a filtered colimit of finite spectra, so that 
\begin{equation}
\label{equation:homology_a_colimit_of_homologies_of_finite_spectra}
E_{0}X \simeq \varinjlim E_{0} X_{\alpha}. 
\end{equation}
Each of $E_{0} X_{\alpha}$ is a continuous $\mathbb{G}_{n}$-module when considered with its $\mfrak$-adic topology, and $E_{0}X$ becomes a topological $\mathbb{G}_{n}$-module since it has the colimit topology with respect to (\ref{equation:homology_a_colimit_of_homologies_of_finite_spectra})
\end{remark}
We will be interested in the continuous cohomology of $\mathbb{G}_{n}$ with coefficients in an $E_{0}$-module equipped with its local $\mfrak$-adic topology. As we will see, this has an elegant description in terms of cohomologies of its finitely generated submodules. The key is the following lemma. 

\begin{lemma}
\label{lemma:mapping_into_local_madic_module_lands_in_a_fg_submodule}
Let $K$ be a compact Hausdorff topological space and $M$ be an $E_{0}$-module equipped with its local $\mfrak$-adic topology. Then, any continuous map $f\colon K \rightarrow M$ factors through a finitely generated submodule. In other words, we have 
\[
\map_{\cts}(K, M) \simeq \varinjlim \map_{\cts}(K, M_{\alpha}),
\]
where $M_{\alpha}$ is the poset of finitely generated submodules of $M$.
\end{lemma}

\begin{proof}
Suppose by contradiction that for every finitely generated submodule $M^{\prime}$, there exists a $k \in K$ such that $f(k) \not\in M^{\prime}$. Proceeding inductively, we produce an increasing sequence 
\[
M_{0} \subseteq M_{1} \subseteq M_{2} \subseteq \ldots 
\]
of finitely generated submodules and points $k_{n} \in K$ such that $f(k_{n}) \in M_{n+1}$, but $f(k_{n}) \notin M_{n}$. 

We claim that an arbitrary finitely generated submodule $N$ contains at most finitely many of the $f(k_{n})$. To see this, let us write $M_{\infty} = \cup M_{n}$ and consider $N \cap M_{\infty}$; the latter is again finitely generated, as $E_{0}$ is noetherian, and its intersection with $\{ f(k_{n}) \}$ is the same as that of $N$. By finite generation, we have $N \cap M_{\infty} = N \cap M_{n}$ for sufficiently large $n$, proving the claim, as the latter contains at most $n$ of the $f(k_{n})$. 

Since each finitely generated submodule is Hausdorff, its finite subspaces are closed and discrete and we deduce that $T = \{ f(k_{n}) \}$ is closed and discrete as a subspace of $M$, as this is true for its intersection with each finitely generated submodule. This is a contradiction, as $T = f(f^{-1}(T))$ must be compact Hausdorff as an image of a closed subspace of $K$, so it cannot be both infinite and discrete. 
\end{proof}

\begin{proposition}
\label{proposition:mapping_out_of_ch_top_space_commutes_with_filtered_colimits_of_modules_with_local_madic_top}
Let $M \simeq \varinjlim M_{\alpha}$ be a filtered colimit diagram of $E_{0}$-modules equipped with their local $\mfrak$-adic topologies and let $K$ be compact Hausdorff. Then, the induced map 
\[
\theta\colon \varinjlim \map_{\cts}(K, M_{\alpha}) \rightarrow \map_{\cts}(K, M)
\]
is an isomorphism of abelian groups.
\end{proposition}

\begin{proof}
We first show surjectivity of $\theta$. Suppose that $f\colon K \rightarrow M$ is a continuous map, by \cref{lemma:mapping_into_local_madic_module_lands_in_a_fg_submodule} it factors through a finitely generated submodule $N \subseteq M$. If we write $N_{\alpha} = N \times_{M} M_{\alpha}$, then since filtered colimits are exact we have 
\[
N \simeq \varinjlim N_{\alpha}
\]
As $N$ is finitely generated and $E_{0}$ is noetherian, it is finitely presented and we deduce that there exists a section $N \rightarrow N_{\alpha}$ for some $\alpha$. Then, the composite
\[
K \rightarrow N \rightarrow N_{\alpha} \rightarrow M_{\alpha}
\]
determines the needed element of the filtered colimit. 

For injectivity, suppose that we have a continuous map $f\colon K \rightarrow M_{\alpha}$ such that the composite $s_{\alpha} \circ f\colon K \rightarrow M$ where $s_{\alpha}\colon M_{\alpha} \rightarrow M$ is the canonical map, is zero. By another application of \cref{lemma:mapping_into_local_madic_module_lands_in_a_fg_submodule}, $K$ factors through a finitely generated submodule $N_{\alpha} \subseteq M_{\alpha}$ and we necessarily have $N_{\alpha} \subseteq \mathrm{ker}(s_{\alpha})$. As $N_{\alpha}$ is finitely generated, the second conidtion implies that we can find a larger index $\beta$ such that the composite 
\[
N_{\alpha} \rightarrow M_{\alpha} \rightarrow M_{\beta}
\]
is zero. It follows that $f$ determines the zero element of the filtered colimit, as needed. 
\end{proof}

\begin{corollary}
\label{corollary:continuous_cohomology_commutes_with_filtered_colimits}
Let $M \simeq \varinjlim M_{\alpha}$ be a filtered colimit of $E_{0}$-modules equipped with their local $\mfrak$-adic topology and compatible continuous $\mathbb{G}_{n}$-actions. Then, 
\[
\rmH^{*}_{\cts}(\mathbb{G}_{n}, M) \simeq \varinjlim \rmH^{*}_{\cts}(\mathbb{G}_{n}, M_{\alpha}).
\]
In particular, for any $E_{0}$-module $N$, we have
\[
\rmH^{*}_{\cts}(\mathbb{G}_{n}, N) \simeq \varinjlim \rmH^{*}_{\cts}(\mathbb{G}_{n}, N_{\alpha}).
\]
where the colimit is taken over the poset $N_{\alpha}$ of finitely generated submodules. 
\end{corollary}

\begin{proof}
The continuous cohomology is computed by the standard cochain complex
\[
M \rightarrow \map_{\cts}(\mathbb{G}_{n}, M) \rightarrow \map_{\cts}(\mathbb{G}_{n} \times \mathbb{G}_{n}, M) \rightarrow \ldots 
\]
of continuous cochains. The statement follows from an application of \cref{proposition:mapping_out_of_ch_top_space_commutes_with_filtered_colimits_of_modules_with_local_madic_top} to each term separately, as $\mathbb{G}_{n} \times \ldots \times \mathbb{G}_{n}$ is compact Hausdorff and taking cohomology of cochain complexes commutes with filtered colimits. 
\end{proof}

\subsection{Inverse limits of $E_*E$-comodules and continuous cohomology}\label{ssec:inverselimitsofcomodules}

In this section, we will give a description of derived functors of the limit 
\[
\textstyle\varprojlim_{E_{*}E}^{s} M / \mfrak^{k} M \simeq \rmH_{-s}(\ \textstyle\varprojlim_{\dcat(E_{*}E)} M / \mfrak^{k} M)
\]
in $E_{*}E$-comodules as cohomology of the Morava stabilizer group. The importance of these derived functors stems from the fact that they form the $E_{2}$-page of the spectral sequence constructed in \S\ref{subsection:construction_of_the_spectral_sequence}.

\begin{lemma}
\label{lem:inverselimascontcohom}
If $M$ is a dualizable $E_{*}E$-comodule, then we have a canonical isomorphism 
\[
\Ext_{\dcat(E_{*}E)}^{s}(E_{*}, \textstyle\varprojlim_{\dcat(E_{*}E)} M / \mfrak^{k} M) := \pi_{-s}\Hom_{\dcat(E_{*}E)}(E_{*}, \textstyle\varprojlim_{\dcat(E_{*}E)} M / \mfrak^{k} M) \simeq \rmH_{\cts}^{s}(\mathbb{G}_{n}, M)
\]
for any $s\ge 0$, where $E_{*} / \mfrak^{k}$ is the $\mfrak$-adic tower of $E_{*}$. 
\end{lemma}
\begin{proof}
For brevity, if $X \in \dcat(E_{*}E)$, let us write 
\[
\Ext^{s}(X) := \Ext^{s}_{\dcat(E_{*}E)}(E_{*}, X) \simeq \pi_{-s} \Hom_{\dcat(E_{*}E)}(E_{*}, X)
\]
Since we have
\[
\Hom_{\dcat(E_{*}E)}(E, \textstyle\varprojlim _{\dcat(E_{*}E)} M / \mfrak^{k} M) \simeq \varprojlim\Hom_{\dcat(E_{*}E)}(E_{*},  M / \mfrak^{k} M),
\]
there is a Milnor sequence
\[
0 \to \textstyle\varprojlim^1 \Ext^{s-1}(M / \mfrak^{k} M) \to \Ext^{s}(\textstyle\varprojlim_{\dcat(E_{*}E)} M / \mfrak^{k} M) \to \varprojlim \Ext^{s}(M / \mfrak^{k} M) \to 0,
\]
where we have omitted subscripts for simplicity and the derived functors of the limit are taken in abelian groups.

We first claim that the $\varprojlim^1$-term above vanishes. Note that since $M / \mfrak^{k}M$ is finitely generated and $\mfrak$-torsion for any $k \geq 1$, we have a canonical isomorphism 
\[
\Ext(M / \mfrak^{k}M) \simeq \rmH_{\cts}^{s}(\mathbb{G}_{n}, M/\mfrak^{k}M)
\]
with the cohomology of the Morava stabilizer group. For $k = 1$, we have a further isomorphism
\[
\Ext(M / \mfrak M) \simeq \Ext_{E_{*}E}(E_{*}, K_{*} \otimes_{E_{*}} M) \simeq \Ext_{E_{*}K}(K_{*}, K_{*} \otimes_{E_{*}} M)
\]
and the group on the right is degreewise finite by \cref{corollary:coh_finiteness_of_ke_comodules}. It follows by induction that $\Ext(M_{*} / \mfrak^k M)$ is degreewise finite for each $k\ge 1$. Therefore, the groups $\Ext(M / \mfrak^k M)$ satisfy the Mittag-Leffler condition as $k$ varies, so that the corresponding $\varprojlim^1$-term vanishes as claimed.

Consequently, we get a string of isomorphisms
\begin{align*}
    \Ext^{s}(\textstyle\varprojlim_{\dcat(E_{*}E)} M/ \mfrak^{k} M) &  \cong \varprojlim \Ext^{s}(E_{*},  M / \mfrak^{k} M) \\
    & \cong \textstyle\varprojlim \rmH_{\cts}^s(\mathbb{G}_n, M / \mfrak^k M) \\
    & \cong \rmH_{\cts}^s(\mathbb{G}_n, M).
\end{align*}
Here, the last isomorphism uses that $M$ is $\mfrak$-adically complete as a finitely generated $E_{*}$-module, and that taking continuous cohomology of a finitely generated profinite group commutes with taking countable inverse limits of finite modules \cite[\S 7]{neukirch2013cohomology}.
\end{proof}

In order to pass from \cref{lem:inverselimascontcohom} to a description of the derived functors $\textstyle\varprojlim^{s}_{E_{*}E} E_{*} / \mfrak^{k}$, we have to understand the homology groups of $\varprojlim_{\dcat(E_{*}E)} M / \mfrak^{k} M$, rather than its homotopy. To do so, we will make use of the Adams condition, which involves certain filtered colimits. To commute past the latter, we will need a technical lemma stating that homotopy groups  
\[
\Ext^{s} (E_{*}, -) := \pi_{-s} \Hom_{\dcat(E_{*}E)}(E_*,-)
\]
preserve certain filtered colimits of bounded complexes. If $p > n+1$, then the monoidal unit of $\dcat(E_{*}E)$ is compact, so no boundedness hypotheses are required, but this can fail for general $n$. In general, we still have the following result. 

\begin{lemma}\label{lem:boundedcompactness}
Suppose $(L_{\alpha})_{\alpha}$ is a filtered system of complexes of $E_*E$-comodules which are uniformly bounded above in the standard $t$-structure on $\dcat(E_{*}E)$;  that is, there exists an $N$ such that for all $s>N$ and all $\alpha$ we have $H_s(L_{\alpha}) = 0$. Then the canonical comparison map 
\[
\phi\colon \varinjlim_{\alpha} \Ext^{s}(L_{\alpha}) \longrightarrow \Ext^{s}(\varinjlim_{\alpha}L_{\alpha})
\]
is an isomorphism for any $s \in \mathbb{Z}$.
\end{lemma}
\begin{proof}
By shifting the system $(L_{\alpha})_{\alpha}$ if necessary, we may reduce to the case of $s = 0$. In other words, we have to show that  that the canonical map
\[
\varinjlim_{\alpha} \ [E_*,L_{\alpha}] \longrightarrow [E_*,\varinjlim_{\alpha}L_{\alpha}]
\]
between homotopy classes of maps in the derived $\infty$-category, is an isomorphism. 

Since $E_*$ is connective in the standard $t$-structure, we can then replace $L_{\alpha}$ by its connective cover $\tau_{\ge 0}L_{\alpha}$; that is, reduce to the case in which there exists a non-negative integer $N$ such that $L_{\alpha}$ has homology concentrated in degrees $[0,N]$ for all $\alpha$. By induction and the five-lemma, we can reduce further to the case that $L_{\alpha}$ is concentrated in a single non-negative degree $[s,s]$ for all $\alpha$. 

Summarizing these reduction steps, it remains to prove that the canonical map
\[
\varinjlim_{\alpha}\Ext_{E_{*}E}^s(E_*,L_{\alpha}) \longrightarrow \Ext_{E_{*}E}^s(E_*,\varinjlim_{\alpha}L_{\alpha})
\]
is an isomorphism for any filtered system of comodules $L_{\alpha} \in \Comod_{E_*E}$. This is clear, these $\Ext$-groups can be computed by the cobar complex, formation of which commutes with filtered colimits. 
\end{proof}

\begin{proposition}\label{prop:inverselimascontcohom}
For any dualizable $E_{*}E$-comodule $M$, there exists a canonical isomorphism
\[
\textstyle\varprojlim_{E_{*}E}^s M / \mfrak^{k} M \cong \rmH^s_{\cts}(\mathbb{G}_n,E_*E \otimes_{E_{*}} M).
\]
between the derived functors of the limit in comodules and the continuous cohomology of the Morava stabilizer group with coefficients in $E_{*}E \otimes_{E_{*}} M$ equipped with its local $\mfrak$-adic topology. 
\end{proposition}

\begin{proof}
We have 
\[
\textstyle{\varprojlim_{E_{*}E}^s} M_{*} / \mfrak^{k} M \simeq H_{-s}(\textstyle\varprojlim_{\dcat(E_{*}E)} M_{*} / \mfrak^{k} M)
\]
which we can further rewrite as 
\[
H_{-s}(\textstyle\varprojlim_{\dcat(E_{*}E)} M / \mfrak^{k} M) \simeq \Ext^{s}(E_{*}E \otimes_{E_{*}} \varprojlim_{\dcat(E_{*}E)} M_{*} / \mfrak^{k} M),
\]
where the tensor product is the derived one. Writing $E_{*}E \simeq \varinjlim N_{\alpha}$ as a filtered colimit of dualizable comodules, we have 
\begin{align*}
        \Ext^{s}(E_{*}E \otimes_{E_{*}} \textstyle\varprojlim_{\dcat(E_{*}E)} M / \mfrak^{k} M) & \cong \Ext^{s}(\varinjlim \ (N  \otimes_{E_{*}} \textstyle\varprojlim_{\dcat(E_{*}E)} M / \mfrak^{k} M)) \\
        & \cong \varinjlim \Ext^{s}(N  \otimes_{E_{*}} \textstyle\varprojlim_{\dcat(E_{*}E)} M / \mfrak^{k} M) \\
        & \cong \varinjlim \Ext^{s}(\textstyle\varprojlim_{\dcat(E_{*}E)} N  \otimes_{E_{*}} M / \mfrak^{k}(N  \otimes_{E_{*}}  M))  \\ 
        & \cong \varinjlim \ \rmH_{\cts}^{s}(\mathbb{G}_{n}, N  \otimes_{E_{*}} M)  \\
        & \cong \rmH_{\cts}^{s}(\mathbb{G}_{n}, \varinjlim \ N  \otimes_{E_{*}} M)  \\
        & \cong \rmH^s_{\cts}(\mathbb{G}_n,E_{*}E \otimes_{E_{*}} M).
    \end{align*}
Here, the first isomorphism uses that the derived tensor product commutes with filtered colimits and the second one is \cref{lem:boundedcompactness}. The third isomorphism uses that tensoring with a dualizable object commutes with limit, the fourth one is \cref{lem:inverselimascontcohom} while the fifth one is \cref{corollary:continuous_cohomology_commutes_with_filtered_colimits}.
\end{proof}

\begin{remark}
 The relationship between derived functors of the limit and derived completion in the context of comodules over a flat Hopf algebroid is studied further in \cite{bhv4}. 
\end{remark}

\subsection{The inverse limit spectral sequence and its convergence}

In this subsection, we study the convergence of the inverse limit spectral sequence based on Morava $E$-theory and combine the previous results in a spectral sequence computing the (uncompleted) $E$-homology of the $K$-local sphere spectrum from continuous group cohomology with coefficients in the (uncompleted) cooperations $E_*E$.

\begin{proposition}
\label{prop:limssconvergence}
If $R=E$ is Morava $E$-theory and $(X^{\alpha}) \in \spectra_E^{\N}$ is a tower of $E$-local spectra, then the spectral sequence of \cref{ex:limss} converges conditionally and strongly, and has a horizontal vanishing line at the $E_r$-page for some $r\ge 2$. 
\end{proposition}
\begin{proof}
By \Cref{lem:modadamstower} and assumption on $X^{\alpha}$, the modified Adams tower $(X_i^{\alpha})$ is $E$-local for all $i$. Therefore, the proof of the smash product theorem due to Hopkins and Ravenel~\cite[Theorem 7.5.6]{orangebook} yields a constant $N$ independent of $\alpha$ such that 
\[
\xymatrix{g_s^{\alpha}\circ \ldots \circ g_{s+N-1}^{\alpha}\colon X_{s+N}^{\alpha} \ar[r] & X_s^{\alpha}}
\]
is null for all $\alpha$; see \cite[Theorem 5.10]{hs_localcohom}. Consequently,
\begin{equation}\label{eq:zerocomposite}
\xymatrix{g_s\circ \ldots \circ g_{s+N-1}\colon \varprojlim X_{s+N}^{\alpha} \ar[r] & \varprojlim X_s^{\alpha}}
\end{equation}
is null as well, so \cite{boardman_convergence} shows that the spectral sequence is conditionally convergent. Similarly, we obtain the horizontal vanishing line from \eqref{eq:zerocomposite}, which in turn gives strong convergence.
\end{proof}

\begin{theorem}
If $(X^{\alpha}) \in \spectra_E^{\N}$ is a tower of $E$-local spectra, then there is a strongly convergent spectral sequence 
\[
\xymatrix{E_2^{s,t} \cong (\varprojlim_{E_*E}^s E_*X^{\alpha})_t \implies E_{t-s}(\varprojlim X^{\alpha}),}
\]
with a horizontal vanishing line at the $E_r$-page for some $r\ge 2$. In particular, this spectral sequence specializes to:
\begin{equation}\label{eq:elimssk(n)localsphere}
\xymatrix{\rmH_{\cts}^s(\mathbb{G}_n,E_{t}E) \cong (\varprojlim_{E_*E}^s E_*/\mfrak^k)_t \implies E_{t-s}(L_KS^0).}
\end{equation}
\end{theorem}
\begin{proof}
The spectral sequence was constructed in \cref{thm:limss}, see also \cref{ex:limss}, while the convergence properties were established in \cref{prop:limssconvergence}. The final statement follows by applying this spectral sequence to a cofinal tower of $E$-local generalized type $n$ Moore spectra as in \cite[Section 4]{hovey_morava_1999}. The identification of the $E_2$-page in terms of continuous cohomology was proven in \cref{prop:inverselimascontcohom}.
\end{proof}

\begin{remark}
The spectral sequence \eqref{eq:elimssk(n)localsphere} may be thought of as an $E_*$-homology version of the $K$-local $E$-based Adams spectral sequence for the sphere, which has signature
\[
\rmH_{\cts}^s(\mathbb{G}_n,E_*)_t \implies \pi_*L_{K}S^0
\]
and is also strongly convergent. In contrast, working with completed $E_*$-homology would result in an isomorphism $E_* \cong E_*^{\vee}L_{K}S^0 \cong \rmH_{\cts}^*(\mathbb{G}_n,E_*^{\vee}E_*)$. This amply highlights the extra information contained in the terms $\rmH_{\cts}^*(\mathbb{G}_n,E_*E)$.
\end{remark}

\section{Digression: Cohomology of $\mathbb{G}_{n}$ as $\Ext$ in comodules}\label{sec:cohomofgn}

In \cref{lem:inverselimascontcohom}, we have expressed continuous cohomology of a Morava stabilizer group with coefficients in a dualizable $E_{*}E$-comodule $M$ as homotopy groups of a the derived completion. In this short section, we will use a form of local duality in the context of comodules to express it as an $\Ext$-groups between actual comodules, rather than objects of the derived $\infty$-category. Using this calculation we will define a comparison map between $\Ext$-groups and continuous cohomology. 

Let $I_k = (p,v_1,\ldots,v_{k-1})$ be the $k$-th invariant chromatic ideal in $E_*$; in particular, $I_n = \mathfrak{m}$. We define $E_* = E_*/I_0^{\infty}$ and then iteratively construct $E_*E$-comodules via the cofiber sequences
\begin{equation}\label{eq:cousin}
E_{*} / I_k^{\infty} \rightarrow v_{k}^{-1} E_{*} / I_k^{\infty} \rightarrow E_{*} / I_{k+1}^{\infty}
\end{equation}
in $\dcat(E_*E)$. Note that $E_{*} / I_k^{\infty} \simeq \varinjlim E_{*} / v_{0}^{l_{0}}, \ldots, v_{k-1}^{l_{k-1}}$, where the colimit is taken over a set of indices with $l_{i} \rightarrow \infty$. For the constituent pieces, we have the following self-duality result:

\begin{lemma}
\label{lemma:quotients_of_estar_are_selfdual_in_the_derived_category}
In the derived category of $E_{*}E$-comodules, the comodule $E_{*} / v_{0}^{i_{0}}, v_{1}^{i_{1}}, \ldots, v_{k}^{i_{k}}$ is self-dual with a shift in the sense that there exists an equivalence
\[
D(E_{*} / v_{0}^{i_{0}}, v_{1}^{i_{1}}, \ldots, v_{k}^{i_{k}}) \simeq \Sigma^{-k} E_{*} / v_{0}^{i_{0}}, v_{1}^{i_{1}}, \ldots, v_{k}^{i_{k}},
\]
where $DX = F(X, E_{*})$ denotes the monoidal dual. 
\end{lemma}
\begin{proof}
We prove this by induction on $k$, where $k = -1$ is clear, since $E_{*}$ is the monoidal unit. 

Now assume that the statement is known for $E_{*} / v_{0}^{i_{0}}, v_{1}^{i_{1}}, \ldots, v_{k-1}^{i_{k-1}}$. Multiplication by $v_{k}^{i_{k}}$ induces a cofiber sequence 
\[
E_{*} / v_{0}^{i_{0}}, v_{1}^{i_{1}}, \ldots, v_{k-1}^{i_{k-1}} \rightarrow  E_{*} / v_{0}^{i_{0}}, v_{1}^{i_{1}}, \ldots, v_{k-1}^{i_{k-1}} \rightarrow E_{*} / v_{0}^{i_{0}}, v_{1}^{i_{1}}, \ldots, v_{k}^{i_{k}}
\]
which after taking duals yields a cofiber sequence
\[
D(E_{*} / v_{0}^{i_{0}}, v_{1}^{i_{1}}, \ldots, v_{k}^{i_{k}}) \rightarrow D(E_{*} / v_{0}^{i_{0}}, v_{1}^{i_{1}}, \ldots, v_{k-1}^{i_{k-1}}) \rightarrow D(E_{*} / v_{0}^{i_{0}}, v_{1}^{i_{1}}, \ldots, v_{k-1}^{i_{k-1}}).
\]
By the inductive assumption, this sequence can be rewritten as 
\[
D(E_{*} / v_{0}^{i_{0}}, v_{1}^{i_{1}}, \ldots, v_{k}^{i_{k}}) \rightarrow \Sigma^{k-1} E_{*} / v_{0}^{i_{0}}, v_{1}^{i_{1}}, \ldots, v_{k-1}^{i_{k-1}} \rightarrow \Sigma^{k-1} E_{*} / v_{0}^{i_{0}}, v_{1}^{i_{1}}, \ldots, v_{k-1}^{i_{k-1}}.
\]
The self-map of  $\Sigma^{k-1} E_{*} / v_{0}^{i_{0}}, v_{1}^{i_{1}}, \ldots, v_{k-1}^{i_{k-1}} \rightarrow \Sigma^{k-1} E_{*} / v_{0}^{i_{0}}, v_{1}^{i_{1}}, \ldots, v_{k-1}^{i_{k-1}}$ appearing on the right cannot be non-zero, since it is a dual to the non-zero class $v_{k}^{i_{k}}$. It follows that it must be itself a unit times $v_{k}^{i_{k}}$ since there is nothing else in this degree, see for example \cite[Thm.~4.3.2(b)]{greenbook}. The assertion then follows.
\end{proof}

\begin{lemma}\label{lem:comodulecompletion}
If $M$ is a finitely generated $E_{*}E$-comodule, then its derived comodule completion of $M$ is given by
\[
\textstyle\varprojlim_{\dcat(E_*E)}(M\otimes_{E_{*}} E_*/v_{0}^{k_{0}}, \ldots, v_{n-1}^{k_{n-1}}) \simeq \textstyle\varprojlim_{\dcat(E_*E)}M/v_{0}^{k_{0}}, \ldots, v_{n-1}^{k_{n-1}},
\]
where on the left hand side we have used the derived tensor product of comodules.
\end{lemma}
\begin{proof}
This is proven in \cite[Proposition 2.25]{bhv4}
\end{proof}

The next result is an incarnation of local duality for comodules \cite[\S5]{bhv1}, relating local cohomology (derived torsion) and local homology (derived completion), combined with \cref{lem:inverselimascontcohom}.

\begin{theorem}
\label{thm:cohomology_of_gn_as_ext_in_comodules}
Let $M$ be a finitely generated $E_{*}E$-comodule. Then, there exists an isomorphism
\[
\Ext^{n+s}_{E_{*}E}(E_{*} / I_n^{\infty}, M) \simeq \rmH_{\cts}^{s}(\mathbb{G}_{n}, M)
\]
between $\Ext$-groups in comodules and the cohomology of $\mathbb{G}_{n}$, for any $s\ge 0$.
\end{theorem}
\begin{proof}
As noted before, we have $E_{*} / I_n^{\infty} \simeq \varinjlim E_{*} / v_{0}^{k_{0}}, \ldots, v_{n-1}^{k_{n-1}}$, where the colimit is taken over a set of indices with $k_{i} \rightarrow \infty$. By by the duality equivalence of  \cref{lemma:quotients_of_estar_are_selfdual_in_the_derived_category} 
we have 
\[
\Hom_{\dcat(E_{*}E)}(\varinjlim\Sigma^{-n}E_{*} / v_{0}^{k_{0}}, \ldots, v_{n-1}^{k_{n-1}}, M) \simeq \Hom_{\dcat(E_{*}E)}(E_{*}, \textstyle\varprojlim_{\dcat(E_{*}E)} M \otimes E_*/v_{0}^{k_{0}}, \ldots, v_{n-1}^{k_{n-1}}).
\]
In light of the \cref{lem:comodulecompletion}, the latter term is equivalent to $\Hom_{\dcat(E_{*}E)}(E_{*}, \varprojlim M/v_{0}^{k_{0}}, \ldots, v_{n-1}^{k_{n-1}})$. It follows that 
\[
\Ext^{n+s}_{E_{*}E}(E_{*} / I_n^{\infty}, M) \simeq \varprojlim \Ext^{s}_{E_{*}E}(E_{*}, M/v_{0}^{k_{0}}, \ldots, v_{n-1}^{k_{n-1}}),
\]
up to possible $\varprojlim^{1}$-terms. However, all of the Ext groups on the right are finite, since $M$ is finitely generated, so these $\varprojlim^{1}$-terms vanish. We then deduce as in  \cref{lem:inverselimascontcohom} that 
\[
\Ext^{n+s}_{E_{*}E}(E_{*} / I_n^{\infty}, M) \simeq \varprojlim \Ext_{E_{*}E}^{s}(E_{*}, M / v_{0}^{k_{0}}, \ldots, v_{n-1}^{k_{n-1}}) \simeq \rmH_{\cts}^{s}(\mathbb{G}_{n}, M),
\]
which is what we wanted to show.
\end{proof}

The boundary maps $\delta_k\colon E_*/I_{k+1}^{\infty} \to \Sigma E_*/I_{k}^{\infty}$ associated with the cofiber sequences \eqref{eq:cousin} compose to give a `residue map' in the derived category:
\[
\delta := \delta_{n-1} \circ \ldots \circ \delta_{0} \colon E_*/I_{n}^{\infty} \to \Sigma^n E_*.
\]
For any finitely generated $E_*E$-comodule $N$, the map $\delta$ induces a natural map on Ext-groups that fits, for any $s$, into a commutative square
\begin{equation}\label{eq:comparisonmap}
\vcenter{
\xymatrix{\Ext^{s}_{E_{*}E}(E_{*}, N) \ar[r]^-{\Ext(\delta,N)} \ar[d]_{\cong} & \Ext^{s+n}_{E_{*}E}(E_*/I_{n}^{\infty}, N) \ar[d]^{\cong} \\
[N^{\dual},E_*]^{-s} \ar[r] & [N^{\dual}, \varprojlim _{\dcat(E_{*}E)} E_{*} / I_n^{k}]^{-s}.}
}
\end{equation}
Here, the left vertical isomorphism is by duality, while the right vertical isomorphism is established as in the proof of the previous theorem. Unwinding the construction, the bottom horizontal map is then induced by the canonical  completion map $E_* \to \varprojlim _{\dcat(E_{*}E)} E_{*} / I_n^{k}$.

\begin{definition}\label{def:comparisonmap}
For a finitely generated $E_*E$-comodule $N$, we define the comparison map $\alpha_N$ as the composite
\[
\alpha_N\colon \Ext^{s}_{E_{*}E}(E_{*}, N) \to \Ext^{n+s}_{E_{*}E}(E_{*} / I_n^{\infty}, N) \simeq \rmH_{\cts}^{s}(\mathbb{G}_{n}, N)
\]
of the map in \eqref{eq:comparisonmap} and the isomorphism of \cref{thm:cohomology_of_gn_as_ext_in_comodules}.
\end{definition}

Note that, by construction, the maps $\alpha_{N}$ are natural and compatible with the long exact sequences of $\Ext$-groups and cohomology. 

\begin{remark}
Alternatively, it is possible to define the comparison map $\alpha_M$ directly using the cobar complex. More precisely, one can show that if $M$ is a finitely generated $E_{*}E$-comodule, then the differentials 
in the cobar complex
\[
M \rightarrow E_{*}E \otimes_{E_{*}} M \rightarrow E_{*}E \otimes_{E_{*}} E_{*}E \otimes_{E_{*}} M \rightarrow \ldots
\]
respect the $\mfrak$-adic filtration (even though they are not $E_{*}$-linear). The completion of the cobar complex is exactly the cochain complex computing continuous $\mathbb{G}_{n}$-cohomology, and the induced map on homology can be shown to be exactly $\alpha_{M}$. 
\end{remark}

\section{Algebraic chromatic splitting conjecture}\label{sec:specialcaseofacss}

In this penultimate section we use the comparison map of the previous section to compute the zeroth derived limit as $\varprojlim^{0}_{E_{*}E} E_{*} / \mfrak^{k} \simeq E_{*} \otimes _{\Z} \Z_{p}$. This result should be viewed in light of Hopkins' algebraic chromatic splitting conjecture, see \cite[\S14]{report_on_e_theory_conjectures}, which we briefly review for context.

\begin{notation}\label{not:abbreviations}
To simplify notation, if $N$ is a finitely generated $E_{*}E$-comodule, we will sometimes use abbreviations $\Ext^{s}(N) := \Ext^{s}_{E_{*}E}(E_{*}, N)$ and $\rmH^{s}(N) := \rmH_{\cts}^{s}(\mathbb{G}_{n}, N)$.
\end{notation}

\subsection{The algebraic chromatic splitting conjecture}

Motivated by computations in height 2, this conjecture was formulated by Hopkins and recorded later in \cite[Conjecture 134]{report_on_e_theory_conjectures}\footnote{We note that the statement given there contains a typo: the left hand side of the equation should have $v_{n-1}$ inverted.} as well as \cite[Conjecture 6.10]{barthel2020chromatic}, and studied further in \cite{bh_acht}. To state it, we first recall the construction of the class $\zeta \in \pi_{-1}L_{K(n)}S^0$ due to Devinatz and Hopkins~\cite[Section 8]{devinatzhopkins_hfp}.

If we write $\mathbb{G}_n^1$ for the kernel of the determinant map $\det\colon \mathbb{G}_n \to \Z_p$, there is a residual action of $\Z_p$ on $E_n^{h\mathbb{G}_n^1}$. For $t$ is a topological generator of $\Z_p$, we then obtain a fiber sequence of spectra
\[
\xymatrix{L_{K(n)}S^0 \ar[r] & E_n^{h\mathbb{G}_n^1} \ar[r]^-{\id-t} & E_n^{h\mathbb{G}_n^1} \ar[r]^-{\delta} & L_{K(n)}S^1.}
\]
The composition $\delta\circ \eta \colon S^0 \to L_{K(n)}S^0$ of the unit map with the boundary map gives rise to an element $\zeta \in \pi_{-1}L_{K(n)}S^0$. In \cite[Proposition 8.2]{devinatzhopkins_hfp}, Devinatz and Hopkins prove that $\zeta$ is non-trivial as long as the height $n>0$.

\begin{conjecture}[Algebraic chromatic splitting conjecture]\label{conj:acsc}
For all heights $n \ge 2$ and sufficiently large $p$, there are isomorphisms of $E_*E$-comodules
\[
\textstyle\varprojlim_{E_{*}E}^sE_*/(p,\ldots,v_{n-2},v_{n-1}^i) \cong 
\begin{cases}
E_*/(p,\ldots,v_{n-2}) & s=0 \\
v_{n-1}^{-1}E_*/(p,\ldots,v_{n-2}) & s=1 \\
0 & \mathrm{otherwise.}
\end{cases}
\]
Moreover, these isomorphisms are topologically realized by the unit map $\iota\colon S^0 \to L_{K(n)}S^0$ and the map $\zeta\colon S^{-1} \to L_{K(n)}S^0$, respectively.
\end{conjecture}

\begin{remark}
As the name suggests, the algebraic chromatic splitting conjecture provides an algebraic counterpart to Hopkins' chromatic splitting conjecture \cite[Conjecture 4.2]{hovey_csc} in the edge case: If the algebraic chromatic splitting conjecture holds at height $n$ and prime $p$, then the maps $\iota$ and $\zeta$ induce a splitting
\[
(\iota,\zeta)\colon L_{n-1}F \oplus \Sigma^{-1}L_{n-1}F \simeq L_{n-1}L_{K(n)}F
\]
for any finite spectrum $F$ of type $n-1$. Indeed, the conclusion of the algebraic chromatic splitting conjecture forces the inverse limit sequence (\ref{eq:generalss}) for $L_{K(n)}F$ to collapse at the $E_2$-page, thereby showing  $(\iota,\zeta)$ is an $E_*$-isomorphism and hence an $E$-local equivalence.
\end{remark}

\begin{example}
In \cref{sec:heightone}, we will verify a height 1 variant of the algebraic chromatic splitting conjecture at any prime, see in particular \cref{theorem:cohomology_of_ee_at_height_one} and also \cite[\S14]{report_on_e_theory_conjectures}. The appearance of the additional tensor factors is the only reason we excluded the case of height 1 from the statement of the algebraic chromatic splitting conjecture above.
\end{example}

\subsection{The comparison map in degree 0}

We work with \cref{not:abbreviations}. Our goal is to study the comparison map $\Ext^{s}(M) \rightarrow \rmH^{s}(M)$, constructed in \cref{def:comparisonmap}. We will focus on the particular case when $s = 0$; note that in this case this map is always injective as both the source and target are abelian subgroups of $M$. 

Thus, the real question is how far this comparison map is from being surjective. When $M$ is the monoidal unit, we have the following folklore result of Hopkins. 
\begin{lemma}[Hopkins]
\label{lemma:invariants_of_the_lubin_tate_ring}
We have $\rmH^{0}(E_{*}) \simeq \mathbb{Z}_{p}$ and the map $\Ext^{0}(E_{*}) \rightarrow \rmH^{0}(E_{*})$ is a $p$-completion; that is, it induces an isomorphism $\mathbb{Z}_{p} \otimes_{\mathbb{Z}} \Ext^{0}(E_{*}) \simeq \rmH^{0}(E_{*})$.
\end{lemma}
\begin{proof}
The first part is \cite[Lemma 1.33]{bobkova2018topological}. The second part follows from the fact that $\Ext^{0}(E_{*}) \simeq \mathbb{Z}_{(p)}$, as in both cases the relevant invariants are generated by the unit of the Lubin--Tate ring. 
\end{proof}

\begin{lemma}
\label{lemma:ext_h_comparison_mono_on_h1_for_e*}
The map $\Ext^{1}(E_{*}) \rightarrow \rmH^{1}(E_{*})$ is a monomorphism. 
\end{lemma}
\begin{proof}
Since it vanishes rationally, the source of this map is $p$-torsion. For a given $x \in \Ext^{1}(E_{*})$ we can thus find a $k$ and $y \in \Ext^{0}(E_{*}/p^{k})$ so that $x$ can be written as $x = \delta(y)$, where $\delta$ denotes the boundary homomorphism. We have a commutative diagram 
\[\begin{tikzcd}
	{\Ext^{0}(E_{*})} & {\Ext^{0}(E_{*})} & {\Ext^{0}(E_{*}/p^{k})} & {\Ext^{1}(E_{*})} \\
	{\rmH^{0}(E_{*})} & {\rmH^{0}(E_{*})} & {\rmH^{0}(E_{*}/p^{k})} & {\rmH^{1}(E_{*})}
	\arrow["{p^{k}}", from=2-1, to=2-2]
	\arrow[from=2-2, to=2-3]
	\arrow["\delta", from=2-3, to=2-4]
	\arrow[from=1-4, to=2-4]
	\arrow[from=1-3, to=2-3]
	\arrow["\delta", from=1-3, to=1-4]
	\arrow[from=1-2, to=2-2]
	\arrow[from=1-2, to=1-3]
	\arrow[from=1-1, to=2-1]
	\arrow[from=1-1, to=1-2]
\end{tikzcd}\]
with exact rows. Thus, it is enough to show that if the image $\widetilde{y} \in \rmH^{0}(E_{*}/p^{k})$ of $y$ can be lifted to $\rmH^{0}(E_{*})$, then $y$ can be lifted to $\Ext^{0}(E_{*})$. By \cref{lemma:invariants_of_the_lubin_tate_ring}, the latter two groups are isomorphic to respectively $\mathbb{Z}_{p}$ and $\mathbb{Z}_{(p)}$. 

By multiplying by an appropriate unit $\sigma$ in the $p$-adics congruent to $1$ modulo $p^{k}$, and replacing the chosen lift $\widetilde{y} \in \rmH^{0}(E_{*}) \simeq \mathbb{Z}_{(p)}$ of $\widetilde{x}$ by $\sigma \widetilde{y}$, we can assume that $\widetilde{y} \in \mathbb{Z}_{(p)}$. It follows that $x$ can be lifted to $\Ext^{0}(E_{*})$, which is what we wanted. 
\end{proof}

Our goal is to use an inductive argument, building on the calculation of invariants of $E_{*}$ due to Hopkins and extending it to a larger class of comodules. The starting point will be the following class. 

\begin{definition}
\label{definition:pure_comodule}
We say a finitely generated $E_{*}E$-comodule $M$ is \emph{pure} if it belongs to the smallest class of comodules closed under extensions and containing all shifts of $E_{*}$. 
\end{definition}

\begin{remark}
\label{remark:pure_comodules_closed_under_taking_duals}
Any pure $M$ is finitely generated and projective over $E_{*}$, and so dualizable. It follows that if 
\[
0 \rightarrow K \rightarrow M \rightarrow N \rightarrow 0
\]
is a short exact sequence of pure comodules, then it is split over $E_{*}$, so that the induced sequence 
\[
0 \rightarrow N^{\dual} \rightarrow M^{\dual} \rightarrow K^{\dual} \rightarrow 0
\]
of linear duals is again exact. We deduce that the class of pure comodules is closed under taking duals. 
\end{remark}

\begin{remark}
It is plausible that it follows from the arguments of Hovey and Strickland, who establish an analogue of the Landweber filtration for $E_{*}E$-comodules \cite{hovey2005comodules}, that all comodules finitely generated and projective over $E_{*}$ are pure in the sense of \cref{definition:pure_comodule}. We did not check whether this is true, as we do not need this fact in our arguments. 
\end{remark}

\begin{lemma}
\label{lemma:comodules_generated_by_extensions_of_e*}
The category of $E_{*}E$-comodules is generated under colimits by pure comodules. 
\end{lemma}
\begin{proof}
Since $E \simeq \varinjlim E_{\alpha}$ is a filtered colimit of finite spectra, it follows formally that the category of $E_{*}E$-comodules is generated under colimits by shifts of comodules of the form $E_{*} E_{\alpha}$ \cite[Proposition 1.4.4]{hovey_htptheory}. 

Furthermore, we claim that since the $E_{\alpha}$ only have even cells, the resulting comodules $E_{*} E_{\alpha}$ are pure. Using a cell decomposition of $E_{\alpha}$ and the associated long exact sequences in homology, which will be short exact here, we deduce that each $E_{*}E_{\alpha}$ is a comodule which can be obtained using iterated extensions of shifts of $E_{*}$ and so is pure. 
\end{proof}

Recall that a $E_{*}E$-comodule is dualizable if and only if the underlying $E_*$-module is finitely generated and projective.

\begin{theorem}
\label{theorem:ext_0_coincides_with_cohomology_after_p_completion_for_fp_comodules}
If $N$ is a dualizable $E_{*}E$-comodule, then
\begin{enumerate}
    \item $\Ext^{0}(N)$ is a finitely generated $\mathbb{Z}_{(p)}$-module and 
    \item the map $\alpha_{N}\colon \Ext^{0}(N) \rightarrow \rmH^{0}(N)$ is a $p$-completion, that is, it induces an isomorphism $\Ext^{0}(N) \otimes _{\mathbb{Z}} \mathbb{Z}_{p} \simeq \rmH^{0}(N)$.
\end{enumerate}
\end{theorem}
\begin{proof}
For $N = E_{*}$ or its shifts, this is \cref{lemma:invariants_of_the_lubin_tate_ring}. 

We first claim that the result is true for comodules $C$ which are pure in the sense of \cref{definition:pure_comodule}. By induction on rank, we can assume that we have a short exact sequence
\[
0 \rightarrow E_{*} \rightarrow C \rightarrow C^{\prime} \rightarrow 0
\]
and that the result already holds for $C^{\prime}$. The long exact sequence of $\Ext$-groups implies that $\Ext^{0}(C)$ is a finitely generated module over $\mathbb{Z}_{(p)}$. Thus, we have a commutative diagram 
\[\begin{tikzcd}
	0 & {\mathbb{Z}_{p} \otimes \Ext^{0}(E_{*})} & {\mathbb{Z}_{p} \otimes \Ext^{0}(C)} & {\mathbb{Z}_{p} \otimes \Ext^{0}(C^{\prime})} & {\mathbb{Z}_{p} \otimes \Ext^{1}(E_{*})} \\
	0 & {\rmH^{0}(E_{*})} & {\rmH^{0}(C)} & {\rmH^{0}(C^{\prime})} & {\rmH^{1}(E_{*}),}
	\arrow[from=1-5, to=2-5]
	\arrow[from=1-4, to=2-4]
	\arrow[from=1-3, to=2-3]
	\arrow[from=1-2, to=2-2]
	\arrow[from=1-1, to=1-2]
	\arrow[from=2-1, to=2-2]
	\arrow[from=1-2, to=1-3]
	\arrow[from=2-2, to=2-3]
	\arrow[from=1-3, to=1-4]
	\arrow[from=2-3, to=2-4]
	\arrow[from=1-4, to=1-5]
	\arrow[from=2-4, to=2-5]
\end{tikzcd}
\]
where the top row is again exact, since $\mathbb{Z}_{p}$ is flat. By induction on rank, counting from the left, the first and third vertical maps are isomorphisms. By \cref{lemma:ext_h_comparison_mono_on_h1_for_e*}, the fourth one is a monomorphism, as $\mathbb{Z}_{p} \otimes \Ext^{1}(E_{*}) \simeq \Ext^{1}(E_{*})$ because the latter group is $p$-torsion. We deduce that the second one is an isomorphism using the five-lemma, which was our claim. 

Now suppose that $N$ is an arbitrary dualizable comodule; by \cref{lemma:comodules_generated_by_extensions_of_e*} we can find a surjection $C \rightarrow N$ from a pure comodule. By taking duals, we can assume that we instead have a monomorphism $N \hookrightarrow C$, where $C$ is again pure by \cref{remark:pure_comodules_closed_under_taking_duals}. Since we already know $\Ext^{0}(C)$ and $\rmH^{0}(C)$ are free finitely generated over respectively $\mathbb{Z}_{(p)}$ and $\mathbb{Z}_{p}$, we deduce that the same is true for their submodules $\Ext^{0}(N)$ and $\rmH^{0}(N)$, proving $(1)$.

Let us denote the cokernel of $N \hookrightarrow C$ by $C^{\prime}$. We then have another commutative diagram
\[
\xymatrix{0 \ar[r] & \mathbb{Z}_{p} \otimes \Ext^{0}(N)  \ar[r] \ar[d] &  \mathbb{Z}_{p} \otimes \Ext^{0}(C)  \ar[r] \ar[d] & \mathbb{Z}_{p} \otimes \Ext^{0}(C^{\prime}) \ar[d] \\
0 \ar[r] & \rmH^{0}(N) \ar[r] & \rmH^{0}(C) \ar[r] & \rmH^{0}(C^{\prime}).}
\]
By what we have proven above, the middle map is an isomorphism as well. Since the right one is injective, we deduce from the four-lemma that the left one is surjective, ending the argument. 
\end{proof}

\begin{remark}
One would like to prove that  \cref{theorem:ext_0_coincides_with_cohomology_after_p_completion_for_fp_comodules} holds for all finitely generated comodules, rather than only the finite  projective ones. It is possible this more general result can be deduced formally from the projective case, but we were not able to do so. 

The key ingredient in the above proof is the calculation of $\rmH_{\cts}^{0}(\mathbb{G}_{n}, E_{*})$; it seems likely that the version for arbitrary comodules would follow from a more general calculation of $\rmH_{\cts}^{0}(\mathbb{G}_{n}, E_{*} / I_{k})$ by Landweber filtration arguments. 
\end{remark}

\subsection{The comodule completion of Morava $E$-theory}

We now come to the main theorem of this section. One may view the computation of this inverse limit at all heights and primes as uniform evidence for the validity for the algebraic chromatic splitting conjecture and hence of the chromatic splitting conjecture itself.

\begin{theorem}
\label{thm:zeroth_derived_limit_of_completion}
We have 
\[
\textstyle\varprojlim_{E_{*}E} E_{*} / \mfrak^{k} \simeq \rmH_{\cts}^0(\mathbb{G}_n,E_*E) \simeq E_{*} \otimes _{\mathbb{Z}} \mathbb{Z}_{p},
\]
where the left hand side is the limit in the category of $E_{*}E$-comodules.
\end{theorem}
\begin{proof}
The first isomorphism is an instance of \cref{prop:inverselimascontcohom}. By \cref{corollary:continuous_cohomology_commutes_with_filtered_colimits}, if $M_{\alpha}$ is a filtered diagram of dualizable $E_{*}E$-comodules such that $\varinjlim M_{\alpha} \simeq E_{*}E$, then 
\[
\rmH_{\cts}^0(\mathbb{G}_n,E_*E) \simeq \varinjlim \rmH_{\cts}^0(\mathbb{G}_n, M_{\alpha}) \simeq \Ext^{0}_{E_{*}E}(E_{*}, M_{\alpha}) \otimes_{\mathbb{Z}} \mathbb{Z}_{p} \simeq E_{*} \otimes_{\mathbb{Z}} \mathbb{Z}_{p}
\]
where the second isomorphism is \cref{theorem:ext_0_coincides_with_cohomology_after_p_completion_for_fp_comodules}. This ends the argument. 
\end{proof}

\begin{remark}
Since $E_{*} S^{0}_{p} \simeq E_{*} \otimes _{MU_{*}} MU_{*}S^{0}_{p} \simeq E_{*} \otimes \mathbb{Z}_{p}$, one can rephrase  \cref{thm:zeroth_derived_limit_of_completion} as saying that $\varprojlim^{0} E_{*} / \mfrak^{k} \simeq E_{*}S^{0}_{p}$, which relates this result to the topological chromatic splitting conjecture. 
\end{remark}

\begin{remark}
There is a short exact sequence of $E_*E$-modules
\[
\xymatrix{0 \ar[r] & \Sigma^{|v_{n-1}|}E_*/I_{n-1} \ar[r]^-{v_{n-1}} & E_*/I_{n-1} \ar[r] & E_*/I_{n} \ar[r] & 0.}
\]
Applying the functor $\varprojlim_{\dcat(E_{*}E)}E_*/I_n^k\otimes-$ and considering the induced long exact sequence in cohomology, we deduce that $v_{n-1}$ acts on $\varprojlim_{\dcat(E_{*}E)}^sE_*/(p,\ldots,v_{n-2},v_{n-1}^i)$ injectively for $s=0$ and bijectively for $s>0$. This provides further evidence for \cref{conj:acsc}. 
\end{remark}

\section{Cooperations at height one}\label{sec:heightone}

In the previous section, we discussed the importance of the derived functors of the limit in the category of comodules and its relation to the chromatic splitting conjecture. In \cref{prop:inverselimascontcohom}, we gave an explicit formula for these in the form of an isomorphism
\[
\textstyle\varprojlim^{s}_{E_{*}E} E_{*} / \mfrak^{k}E_{*} \simeq \rmH^{s}_{\cts}(\mathbb{G}_{n}, E_{*}E),
\]
where on the left hand side we have the $s$-th derived functor of the limit in $E_{*}E$-comodules and  the right hand side we have continuous cohomology of the Morava stabilizer group. To make use of this, we would like to understand $E_{*}E$ as a $\mathbb{G}_{n}$-representation; in this section, we will do so for $n = 1$.

\begin{warning}
Let us warn the reader again that we are interested in the uncompleted cooperations; that is, $E_{*}E = \pi_{*}(E \otimes E)$. The description of the completed cooperations 
\[
E^{\vee}_{*}E := \pi_{*}(L_{K}(E \otimes E)) \simeq (E_{*}E)^{\vee}_{\mfrak}
\]
is well-known, as we have $E^{\vee}_*E \simeq \map_{\cts}(\mathbb{G}_{n}, E_{*})$ at any height by a result of Strickland \cite{strickland2000gross} or \cite{hovey_operations}.
\end{warning}

\begin{notation}
Instead of the Honda formal group law, it will be convenient to work with the multiplicative formal group law over $\mathbb{F}_{p}$, which we will denote by $\Gamma$. This is of height one, so that the Lubin--Tate ring is isomorphic to the $p$-adics, with the universal deformation the multiplicative formal group law over $\mathbb{Z}_{p}$. 

By uniqueness part of the Goerss--Hopkins--Miller theorem, the Lubin--Tate spectrum associated to this universal deformation is given by
\[
E := KU^{\vee}_{p},
\]
the $p$-complete complex $K$-theory. The action of $\mathbb{G}_{1}$ is given by Adams operations. 
\end{notation}

\begin{lemma}
\label{lemma:ee_at_height_one_in_relation_to_eku}
There is an isomorphism 
\[
E_{*}E \simeq E_{*}KU \otimes_{\mathbb{Z}} \mathbb{Z}_{p}
\]
of $\mathbb{G}_{1}$-representations, where on the right hand side the action is trivial on the $\mathbb{Z}_{p}$-factor. 
\end{lemma}

\begin{proof}
We have $KU_{*} \simeq \mathbb{Z}[u^{\pm 1}]$ and $E_{*} \simeq KU_{*} \otimes \mathbb{Z}_{p}$, which is flat over $KU_*$. It then follows from flatness that we have an isomorphism 
\[
E_{*}E \simeq E_{*}KU \otimes_{\mathbb{Z}} \mathbb{Z}_{p}
\]
induced by the multiplication map 
\[
(E \otimes KU) \otimes_{KU} E \rightarrow E \otimes E
\]
Since this multiplication map is $\mathbb{G}_{1}$-equivariant with respect to the action on the left $E$ factor, this is in fact an isomorphism of representations.
\end{proof}

\begin{corollary}
\label{corollary:coh_of_ee_is_coh_of_eku_otimes_zp}
There is an isomorphism 
\[
\rmH^{s}_{\cts}(\mathbb{G}_{1}, E_{*}E) \simeq \rmH^{s}_{\cts}(\mathbb{G}_{1}, E_{*}KU) \otimes_{\mathbb{Z}} \mathbb{Z}_{p}.
\]
\end{corollary}

\begin{proof}
Let us write $\rmH^{*}(M) := \rmH^{*}_{\cts}(\mathbb{G}_{1}, M)$ for brevity. As $\mathbb{Z}_{p}$ is flat, it can be written as a filtered colimit $\mathbb{Z}_{p} \simeq F_{\alpha}$ of finitely generated, free abelian groups by Lazard's theorem. We then have 
\[
\rmH^{s}(E_{*}E) \simeq \rmH^{s}(E_{*}KU \otimes_{\mathbb{Z}} \mathbb{Z}_{p}) \simeq \rmH^{s}(\varinjlim E_{*}KU \otimes_{\mathbb{Z}} F_{\alpha}) \simeq \varinjlim \rmH^{s}(E_{*}KU \otimes_{\mathbb{Z}} F_{\alpha}), 
\]
where the last equivalence used \cref{corollary:continuous_cohomology_commutes_with_filtered_colimits}, and further 
\[
\varinjlim \rmH^{s}(E_{*}KU \otimes_{\mathbb{Z}} F_{\alpha}) \simeq \varinjlim \rmH^{s}(E_{*}KU) \otimes_{\mathbb{Z}} F_{\alpha} \simeq \rmH^{s}(E_{*}KU) \otimes_{\mathbb{Z}} \mathbb{Z}_{p},
\]
where the first equivalence used additivity of cohomology.
\end{proof}

It follows from \cref{corollary:coh_of_ee_is_coh_of_eku_otimes_zp} that to understand the cohomology with coefficients in $E_{*}E$, it is enough to understand $E_{*}KU$. 

As we are working with coefficients in the homology of a spectrum rather than an arbitrary $\mathbb{G}_{1}$-module, we have a canonical family of invariants. To be more precise, we have a Hurewicz homomorphism
\[
KU_{*} \rightarrow E_{*}KU
\]
whose image necessarily consists of $\mathbb{G}_{1}$-invariants. It follows that the cohomology groups 
\[
\rmH^{*}_{\cts}(\mathbb{G}_{1}, E_{*}KU)
\]
are $2$-periodic. As they're also even, as $E_{*}KU$ is, it is enough to determine the cohomology with coefficients in $E_{0}KU$. 

\begin{lemma}
\label{lemma:completion_map_on_e0ku_is_injective}
The composite
\[
E_{0}KU \rightarrow  E_{0}^{\vee}KU \simeq E_{0}^{\vee}E \simeq \map_{\cts}(\mathbb{G}_{1}, \mathbb{Z}_{p})
\]
is injective. 
\end{lemma}

\begin{proof}
Note that by Landweber exactness we have 
\[
E_{0}KU \simeq \mathbb{Z}_{p} \otimes _{\mathbb{Z}} KU_{0}KU.
\]
This is flat over $\mathbb{Z}_{p}$ and we deduce from \cref{proposition:collapse_of_k_based_ass_for_flat_e_modules} that the above composite can be identified with a $p$-completion map. By a result of Adams, Harris, and Switzer, $KU_{0}KU$ is free as an abelian group \cite{adams1971hopf}. It follows that the $p$-completion map 
\[
\mathbb{Z}_{p} \otimes_{\mathbb{Z}} KU_{0}KU \rightarrow (\mathbb{Z}_{p} \otimes_{\mathbb{Z}} KU_{0}KU)^{\vee}_{p} \simeq (KU_{0}KU)^{\vee}_{p} \simeq E_{0}^{\vee}E
\]
is injective as needed.
\end{proof}
The completion map of \cref{lemma:completion_map_on_e0ku_is_injective} is $\mathbb{G}_{1}$-equivariant, and we deduce that $E_{0}KU$ is naturally a subrepresentation of $\map_{\cts}(\mathbb{G}_{1}, \mathbb{Z}_{p})$. On the latter, $\mathbb{G}_{1}$ acts by right multiplication on the source. Our goal is to describe the given subrepresentation.

By unwrapping the definitions, the map of  \cref{lemma:completion_map_on_e0ku_is_injective} can be described explicitly. By standard results about weakly periodic Landweber exact homology theories, $KU_{0}KU$ is the algebra classifying automorphisms of the multiplicative formal group law. If $g \in \mathbb{G}_{1}$ is an automorphisms over the $p$-adics, then we can write 
\[
g(x) = \sum_{i} b_{i}(g) x^{i},
\]
where the $b_{i}$ depend continuously on $g$. The above expression defines an automorphism of the multiplicative formal group law over $\map_{\cts}(\mathbb{G}_{1}, \mathbb{Z}_{p})$; this is the automorphism classified by the injective map 
\[
KU_{0}KU \rightarrow \map_{\cts}(\mathbb{G}_{1}, \mathbb{Z}_{p}).
\]
Combined with \cref{lemma:completion_map_on_e0ku_is_injective}, this proves the following. 

\begin{proposition}
\label{proposition:e0ku_as_a_subalgebra_generated_by_certain_functions}
The ring $E_{0}KU$ is the $\mathbb{Z}_{p}$-subalgebra of $\map_{\cts}(\mathbb{G}_{1}, \mathbb{Z}_{p})$ generated by the functions $b_{i}\colon \mathbb{G}_{1} \rightarrow \mathbb{Z}_{p}$ for $i \geq 1$ as well as $b_{1}^{-1}$. 
\end{proposition}
As we work with the multiplicative formal group law, these functions $b_{i}$ can be made explicit. Assume first that $a \in \mathbb{Z}$, so that the corresponding multiple of the identity of the multiplicative formal group law is given by 
\[
[a](x) = (1+x)^{a}-1 = \sum^{\infty}_{k > 0} {a \choose k} x^{k}.
\]
The last expression makes sense even if $a$ is a $p$-adic integer, where we take 
\[
{a \choose k} = \frac{a(a-1)\ldots(a-k+1)}{k!},
\]
and the sum is now possibly infinite. One can check that these are automorphisms of the multiplicative formal group law even if $a$ is $p$-adic, so that the identification 
\[
\mathbb{Z}_{p}^{\times} \simeq \mathbb{G}_{1}
\]
is given by $a \mapsto [a](x) = \sum^{\infty}_{k > 0} {a \choose k} x^{k}$. Thus, \cref{proposition:e0ku_as_a_subalgebra_generated_by_certain_functions} can be stated in the following more explicit form. 

\begin{proposition}
\label{proposition:e0ku_as_subalgebra_generated_by_binomial_coefficients}
The algebra $E_{0}KU$ is the $\mathbb{Z}_{p}$-subalgebra of $\map_{\cts}(\mathbb{G}_{1}, \mathbb{Z}_{p})$ generated by the functions $b_{i}(a) = {a \choose i}$ for $i \geq 1$ as well as $b_{1}^{-1}(a) = \frac{1}{a}$.
\end{proposition}

\begin{remark}
The functions $b_{i}$ appearing in \cref{proposition:e0ku_as_subalgebra_generated_by_binomial_coefficients} are well-known to $p$-adic analysts. By a fundamental theorem of Mahler, any continuous $p$-adic function on $\mathbb{Z}_{p}$ can be written as an infinite convergent sum of the $b_{i}$-s in a unique way \cite{mahler1958interpolation}.
\end{remark}

This immediately implies the following calculation. 

\begin{proposition}
\label{proposition:h0_of_eku}
We have $\rmH^{0}_{\cts}(\mathbb{G}_{1}, E_{0}KU) \simeq \mathbb{Z}_{p}$. 
\end{proposition}

\begin{proof}
This is immediate from \cref{proposition:e0ku_as_subalgebra_generated_by_binomial_coefficients}, as the only functions $\mathbb{G}_{1} \rightarrow \mathbb{Z}_{p}$ invariant under translation are the constants. 
\end{proof}
To compute higher cohomology groups, it is convenient to observe that they are necessarily rational. 

\begin{lemma}
\label{lemma:higher_coh_groups_of_eku_are_rational}
The groups $\rmH^{i}_{\cts}(\mathbb{G}_{1}, E_{0}KU)$ are rational vector spaces for $i > 0$. 
\end{lemma}

\begin{proof}
Since the inclusion $E_{0}KU \hookrightarrow \map_{\cts}(\mathbb{G}_{1}, \mathbb{Z}_{p})$ is an isomorphism after $p$-completion, as we observed in the proof of \cref{lemma:completion_map_on_e0ku_is_injective}, we deduce that it is an isomorphism modulo $p$. It follows that 
\[
\rmH^{i}_{\cts}(\mathbb{G}_{1}, E_{0}KU \otimes_{\mathbb{Z}_{p}} \mathbb{F}_{p}) \simeq \rmH^{i}_{\cts}(\mathbb{G}_{1}, \map_{\cts}(\mathbb{G}_{1}, \mathbb{F}_{p}))
\]
and the right hand side vanishes for $i > 0$. The long exact sequence of cohomology associated to 
\[
0 \rightarrow E_{0}KU \rightarrow E_{0}KU \rightarrow E_{0}KU \otimes_{\mathbb{Z}_{p}} \mathbb{F}_{p} \rightarrow 0,
\]
where the first map is multiplication by $p$, gives the required result. 
\end{proof}

\begin{proposition}
\label{proposition:h1_coh_of_eku}
We have 
\[
\rmH^{1}_{\cts}(\mathbb{G}_{1}, E_{0}KU) \simeq \mathbb{Q}_{p}
\]
and the higher cohomology groups vanish. 
\end{proposition}

\begin{proof}
By \cref{lemma:higher_coh_groups_of_eku_are_rational}, all cohomology groups of positive degree are rational, so that we can replace $E_{0}KU$ by $\mathbb{Q}_{p} \otimes_{\mathbb{Z}_{p}} E_{0}KU$ without changing these. The cohomology with the coefficients in the latter can be computed directly.

As each binomial coefficient $b_{i}(a) = {a \choose i}$ is a polynomial in $a$ with rational coefficients, we deduce from \cref{proposition:e0ku_as_subalgebra_generated_by_binomial_coefficients} that $\mathbb{Q}_{p} \otimes_{\mathbb{Z}_{p}} E_{0}KU$ can be identified with the $\mathbb{Q}_{p}$-subalgebra of $\map_{\cts}(\mathbb{G}_{1}, \mathbb{Q}_{p})$ generated by $b_{1}(a) = a$ and its inverse. 

It follows that there is an isomorphism of representations 
\[
\mathbb{Q}_{p} \otimes_{\mathbb{Z}_{p}} E_{0}KU \simeq \mathbb{Q}_{p}[b^{\pm 1}],
\]
where $b = b_{1}$ and the action is determined by $g \cdot b = gb$ for $g \in \mathbb{G}_{1}$. This decomposes as a direct sum of character representations
\[
\mathbb{Q}_{p}[b^{\pm 1}] \simeq \bigoplus _{k \in \mathbb{Z}} \mathbb{Q}_{p}(k),
\]
where we identify $\mathbb{Q}_{p}(k)$ with the subrepresentation generated by $b^{k}$. It follows that 
\[
\rmH^{i}_{\cts}(\mathbb{G}_{1}, \mathbb{Q}_{p}[b^{\pm 1}]) \simeq \bigoplus_{k \in \mathbb{Z}} \rmH^{i}_{\cts}(\mathbb{G}_{1}, \mathbb{Q}_{p}(k)).
\]

Let $U \subseteq \mathbb{G}_{1} \simeq \mathbb{Z}_{p}^{\times}$ be a normal subgroup of finite index such that $U \simeq \mathbb{Z}_{p}$. As finite groups have no rational cohomology, Lyndon--Hochschild--Serre spectral sequence collapses and gives an isomorphism 
\[
\rmH^{i}(\mathbb{G}_{1}, \mathbb{Q}_{p}(k)) \simeq \rmH^{i}(U, \mathbb{Q}_{p}(k))^{\mathbb{G}_{1}/U}.
\]
It follows that the cohomology groups vanish above degree $1$, as $U \simeq \mathbb{Z}_{p}$ is of cohomological dimension one. 

As $U \simeq \mathbb{Z}_{p}$ is cyclic as a profinite group, its cohomology can be computed using a two-step cochain complex, as we did in \cref{sec:somecomputations}. It follows that the only one-dimensional representation of $U$ over $\mathbb{Q}_{p}$ with non-trivial cohomology is the trivial one. As the characters $\mathbb{Q}_{p}(k)$ restrict to non-trivial representations for $k \neq 0$, we deduce that 
\[
\rmH^{1}_{\cts}(\mathbb{G}_{1}, \mathbb{Q}_{p}[b^{\pm 1}]) \simeq  \rmH^{1}_{\cts}(\mathbb{G}_{1}, \mathbb{Q}_{p}) \simeq \mathbb{Q}_{p}
\]
as claimed.
\end{proof}

\begin{theorem}
\label{theorem:cohomology_of_ee_at_height_one}
At height $n=1$ and any prime, we have 
\[
\rmH_{\cts}^{s}(\mathbb{G}_{1}, E_{*}E) \simeq \begin{cases}
E_{*} \otimes_{\mathbb{Z}} \mathbb{Z}_{p} & \mbox{when } $s=0$;\\
E_{*} \otimes_{\mathbb{Z}} \mathbb{Q}_{p}& \mbox{when } $s=1$;\\
0 & \mbox{otherwise}.
\end{cases}
\]
\end{theorem}

\begin{proof}
As we observed previously that the cohomology with coefficients in $E_{*}E$ is concentrated in even internal degrees and $2$-periodic, it is enough to establish this in internal degree zero, where this is a combination of \cref{corollary:coh_of_ee_is_coh_of_eku_otimes_zp}, \cref{proposition:h0_of_eku} and \cref{proposition:h1_coh_of_eku}.
\end{proof}

\begin{remark}
As we observed at the beginning, the cohomology groups of \cref{theorem:cohomology_of_ee_at_height_one} coincide with derived functors of completion in comodules by \cref{prop:inverselimascontcohom}. Using the latter perspective, a sketch of calculation of the rationalization of the above groups at $p > 2$ appears in a talk by Hopkins, recorded in \cite[Talk 14]{report_on_e_theory_conjectures}. We were independently informed by Hopkins that he first made this calculation in the 1990s and that it was part of the motivation towards the algebraic chromatic splitting conjecture. 
\end{remark}

\addtocontents{toc}{\vspace{5mm}}
\biblio
\bibliography{bibliography}\bibliographystyle{alpha}
\end{document}